\documentclass[11pt,reqno]{amsart}
\usepackage[margin=1in]{geometry} 
\usepackage{amsmath, amsthm, mathtools, amssymb, enumitem}
\usepackage{stmaryrd}
\usepackage{mathrsfs}
\usepackage{aligned-overset}

\usepackage[
    backend=biber,
    style=alphabetic,
    sorting=anyt,
    maxalphanames=99,
    maxbibnames=99,
    backref=true
]{biblatex}
\DeclareFieldFormat[article]{pages}{#1}
\renewbibmacro{in:}{}
\AtEveryBibitem{\clearfield{month}}
\AtEveryBibitem{\clearfield{day}}

\DefineBibliographyStrings{english}{
    backrefpage = {cited on page},
    backrefpages = {cited on pages},
}

\usepackage[x11names]{xcolor}
\usepackage{orcidlink}

\usepackage{subcaption}

\usepackage{graphicx}
\graphicspath{{./figures/}}
\usepackage{hyperref}
\hypersetup{
    colorlinks=true,
    urlcolor=blue,
    linkcolor=blue,
    citecolor=SpringGreen3
}

\usepackage{tabto}

\makeatletter
\newcommand{\boxup}{\mathbin{\mathpalette\boxup@{}}}

\newcommand{\boxup@}[2]{%
  \setbox0=\hbox{$#1\boxplus$}%
  \rlap{%
        \raisebox{-\dp0 + 0.04\ht0}{%
        \hspace{0.07\wd0}%
        \color{gray!70}%
        \rule{0.42\wd0}{0.47\ht0}%
      }%
    }%
  \box0%
}
\makeatother

\usepackage{xspace}

\renewcommand{\geq}{\geqslant}
\renewcommand{\ge}{\geqslant}
\renewcommand{\leq}{\leqslant}
\renewcommand{\le}{\leqslant}

\newcommand{\tal}{\widetilde{\alpha}}
\newcommand{\tell}{\widetilde{\ell}}
\newcommand{\Easy}{\mathsf{Easy}}
\newcommand{\Dep}{\mathsf{Dep}}
\newcommand{\tomega}{\widetilde{\omega}}
\newcommand{\tsigma}{\widetilde{\sigma}}

\DeclareMathOperator{\Proj}{\mathsf{Proj}}
\DeclareMathOperator{\Size}{\mathsf{Size}}
\newcommand{\Remove}{\mathsf{Remove}}

\newcommand{\Seg}{\mathsf{Seg}}
\newcommand{\Bad}{\textsf{Bad}\xspace}
\newcommand{\Major}{\textsf{Major}\xspace}

\newcommand{\dist}{\mathrm{d}_1}
\newcommand{\cen}{\mathsf{cen}}
\newcommand{\bl}{\mathsf{bl}}
\newcommand{\tr}{\mathsf{tr}}

\newcommand{\Zdir}{\overrightarrow{\mathbb{Z}}}

\newcommand{\BB}{\mathcal{B}}
\newcommand{\BBlift}{\widetilde{\BB}}
\newcommand{\Path}{P}
\newcommand{\Pathlift}{\widetilde{\Path}}
\newcommand{\Paths}[1]{\mathcal{P}^{#1, M}}

\newcommand{\kbox}[1]{[0,8^{-#1})^d}
\newcommand{\kboxlift}[1]{[0,8^{-#1})^d \times [0,10^{-#1})}

\newcommand{\dirpc}{p_{c}^{\rightarrow}}

\newcommand{\al}{\alpha}

\newcommand{\0}{\mathbf{0}}
\newcommand{\ee}{\mathbf{e}}
\newcommand{\qq}{\mathfrak{q}}

\renewcommand{\th}{\textsuperscript{th}}

\newcommand{\norm}[1]{\left\lVert#1\right\rVert}

\DeclarePairedDelimiter\ceil{\lceil}{\rceil}
\DeclarePairedDelimiter\floor{\lfloor}{\rfloor}

\newcommand{\lb}{\llbracket}
\newcommand{\rb}{\rrbracket}

\newcommand{\wt}{\widetilde}

\newcommand{\R}{\mathbb{R}}  
\newcommand{\Z}{\mathbb{Z}}
\newcommand{\Zpos}{\mathbb{Z}_{\ge0}}

\newcommand{\e}{\varepsilon}
\renewcommand{\d}{\delta}

\newcommand{\ls}{\lesssim}
\newcommand{\gs}{\gtrsim}

\newcommand{\E}{\mathbb{E}}
\renewcommand{\P}{\mathbb{P}}

\newcommand{\1}{\mathbf{1}}

\newcommand{\msf}{\mathsf}

\newcommand{\mrm}{\mathrm}

\newcommand{\cA}{\mathcal{A}}
\newcommand{\cB}{\mathcal{B}}

\newcommand{\cG}{\mathcal{G}}
\newcommand{\cH}{\mathcal{H}}

\newcommand{\cL}{\mathcal{L}}

\newcommand{\cN}{\mathcal{N}}

\newcommand{\cP}{\mathcal{P}}

\newcommand{\cS}{\mathcal{S}}
\newcommand{\cT}{\mathcal{T}}

\newcommand{\sB}{\mathsf{B}}
\newcommand{\sC}{\mathsf{C}}

\newcommand{\sL}{\mathsf{L}}

\newcommand{\sa}{\mathsf{a}}
\newcommand{\sfb}{\mathsf{b}}

\newcommand{\sm}{\mathsf{m}}

\newcommand{\su}{\mathsf{u}}
\newcommand{\sv}{\mathsf{v}}

\newcommand{\sx}{\mathsf{x}}
\newcommand{\sy}{\mathsf{y}}

\newcommand{\bP}{\mathbb{P}}

\newcommand{\bR}{\mathbb{R}}
\newcommand{\bS}{\mathbb{S}}

\newcommand{\bZ}{\mathbb{Z}}

\newcommand{\bfE}{\mathbf{E}}

\newcommand{\bfP}{\mathbf{P}}

\newcommand{\br}{\mathfrak{b}}

\DeclareRobustCommand{\SkipTocEntry}[5]{} 

\usepackage[capitalize]{cleveref}
\crefname{maintheorem}{Theorem}{Theorems}
\crefname{figure}{Figure}{Figures}

\numberwithin{equation}{section}

\theoremstyle{plain}
\newtheorem{theorem}{Theorem}[section]

\newtheorem{proposition}[theorem]{Proposition}

\newtheorem{lemma}[theorem]{Lemma}

\theoremstyle{remark}
\newtheorem{definition}[theorem]{Definition}
\newtheorem{remark}[theorem]{Remark}

\title{Directed fractal percolation and non-Lipschitz variants}

\author{Shirshendu Ganguly \and Victor Ginsburg \and Kaihao Jing}

\address{Shirshendu Ganguly\\
Department of Statistics,
University of California, Berkeley, CA, USA} 
\email{sganguly@berkeley.edu}

\address{Victor Ginsburg \orcidlink{0000-0001-9399-6748}\\
Department of Mathematics, University of California, Berkeley, CA, USA}
\email{victor@math.berkeley.edu}

\address{Kaihao Jing\\
Department of Statistics, University of California, Berkeley, CA, USA}
\email{khjing@berkeley.edu}

\begin{document}

\begin{abstract}
While models of undirected and directed geometry in random i.i.d. disorder with rapidly decaying tails are expected to behave similarly and exhibit features of the Kardar--Parisi--Zhang universality class, the scenario when the background noise is fractal is expected to be significantly different. A canonical example of the latter is
Liouville Quantum Gravity, where the planar Gaussian free field and its thick points form the fractal background. A toy model capturing some of the essential features is given by 
\emph{fractal percolation}, a random Cantor set introduced by Mandelbrot \cite{ManIntermittentTurbulenceSelfsimilar1974} to study intermittency in fluids. In the simplest setting, dyadic cubes of different scales are retained independently with probability $p$ (a parameter of the model), and the intersection of all retained cubes forms the set of open sites.
Connectivity properties of such sets have been intensely studied. 
Across \cite{CCDConnectivityPropertiesMandelbrots1988,ChaAbsenceDirectedFractal1995}, a surprising result was proven. 
Namely, while there is a phase transition for connectivity of the fractal percolation set (like in usual bond percolation),  directed percolation never occurs.
In a similar spirit, in geometric measure theory, results have been obtained on sizes of intersections of the fractal percolation set with graphs of Lipschitz functions \cite{BJJ+FractalPercolationUnrectifiable2021,SSSetsWhichAre2015,SSLargestSliceFractal2025}.
As a step towards understanding last passage percolation (LPP) driven by the Gaussian free field (initiated in \cite{GGNLastPassagePercolation2026}), we study the problem of quantifying LPP in fractal percolation.
Developing a multi-scale framework, we show that a directed path of length $n$ cannot pass through more than $n e^{-\log n / \log\log n}$ open sites, and thus establish a polynomial gap from linear growth of the passage time (up to the spurious $\log\log n$ factor).
The obstruction to high LPP values is the Lipschitz nature of directed paths.
Along these lines, 
\cite{ChaLengthShortestCrossing1996} showed that
any connected subset of fractal percolation must have Hausdorff dimension strictly larger than one. In our second main result, we show that one can construct directed paths passing through open sites in fractal percolation provided they move exponentially fast in the spatial direction---a model we introduce and term as \emph{unrectifiable directed percolation.}
\end{abstract}

\maketitle

\setcounter{tocdepth}{2}
\tableofcontents

\begin{figure}[ht]
    \centering
    \begin{subfigure}[t]{0.4\textwidth}
        \includegraphics[width=\linewidth]{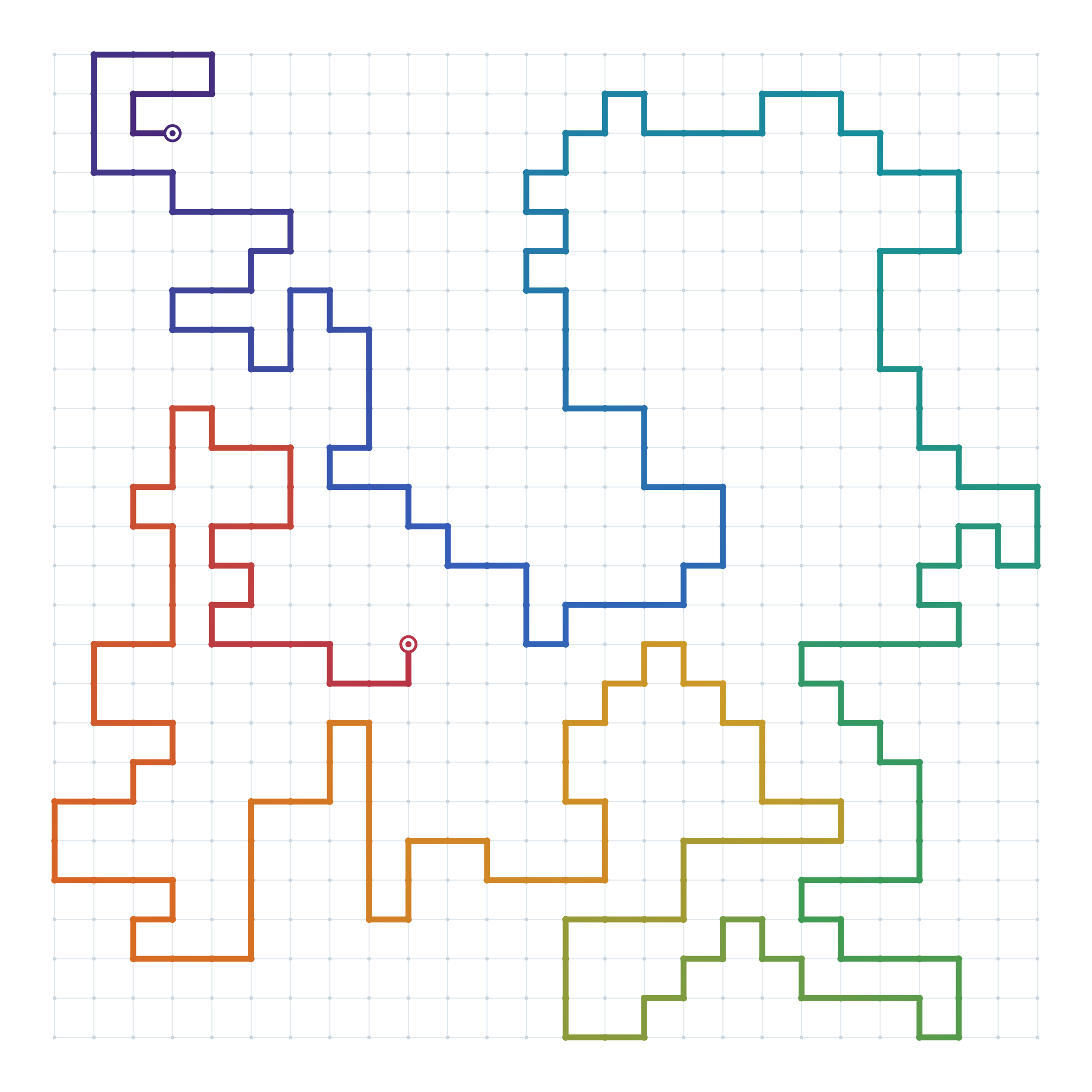}
        \caption{A tortuous path in $\Z^2$.}
    \end{subfigure}
    \hfill
    \begin{subfigure}[t]{0.5\textwidth}
        \includegraphics[width=\linewidth,height=0.8\linewidth]{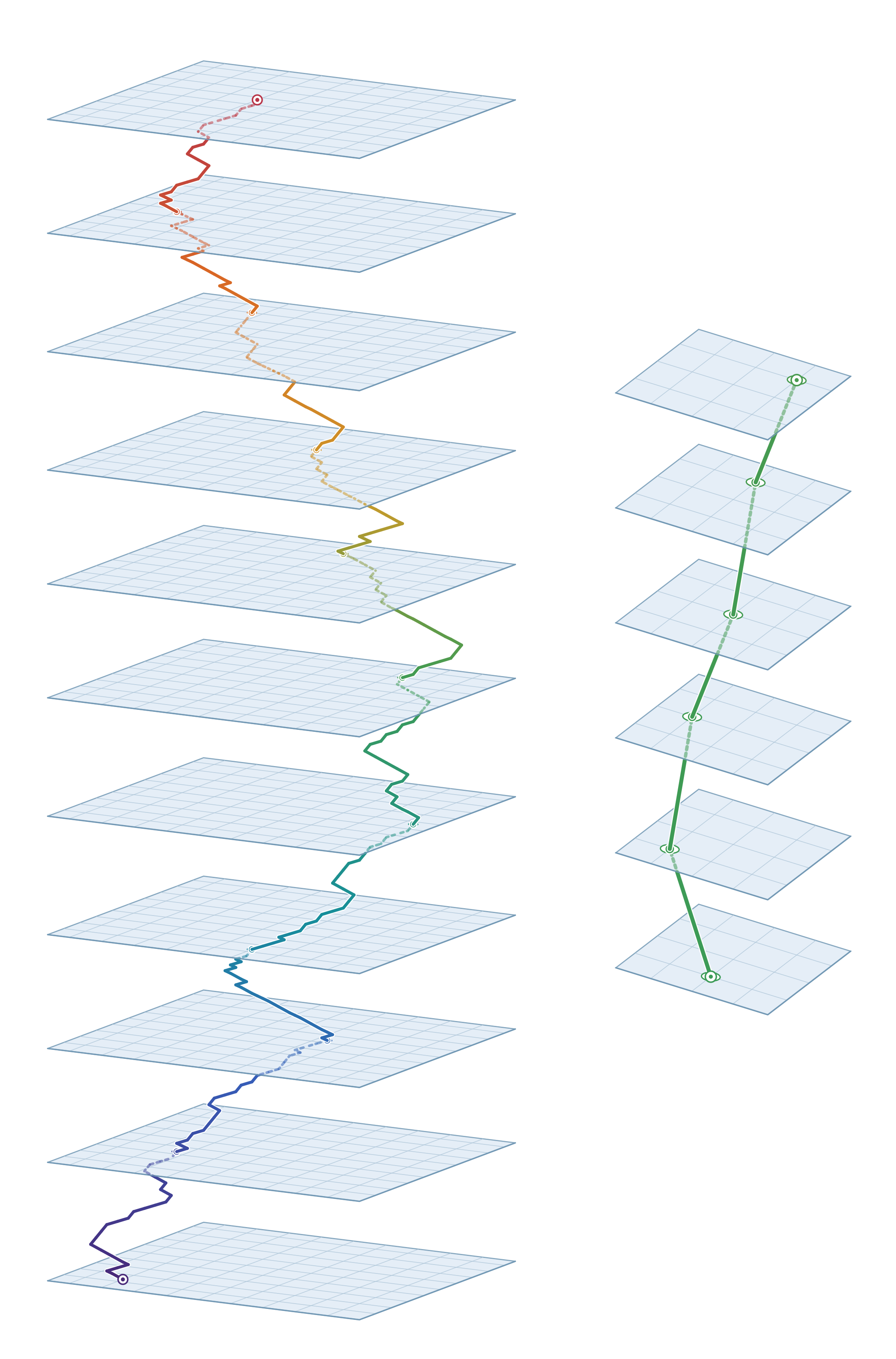}
        \caption{The path is ``lifted'' to $\Z^{2}\times \Z$, where the latter has directed edges $(x,t) \to (y,t+1)$ for $\norm{x-y}_1 = 1$.
        Orthogonal projection onto $\Z^2\times\{0\}$ coordinate recovers the 2D path.
        The left depicts the large-scale trajectory; the right depicts the lattice structure.}
    \end{subfigure}
  \caption{}
  \label{fig:lattice-lift}
\end{figure}

\section{Introduction}\label{s:intro}

Geometry driven by random disorder is a central topic of interest with far-reaching connections to a wide variety of models ranging from first/last passage percolation to quantum field theory.
A particularly intriguing class of questions arise when the background randomness exhibits fractal behavior.
An important classical question of a related flavor, albeit without external randomness, involves constructing diffusions such as Brownian motion on fractal sets such as the Sierpinski gasket, see e.g. \cite{KusDiffusionProcessFractal1987,GolRandomWalksDiffusions1987,BPBrownianMotionSierpinski1988}.
A natural random analogue of the Sierpinski gasket is \emph{fractal percolation}, a random Cantor set introduced by Mandelbrot \cite{ManIntermittentTurbulenceSelfsimilar1974} to study intermittency in fluids.
(see the surveys \cite{ChaAspectsFractalPercolation1995,RSGeometryFractalPercolation2013,SVFractalPercolations2018,KTRecentProgressFractal2025} and the many references therein for more background).
A related but much richer object is the two-dimensional Gaussian free field (GFF), whose level sets comprise a continuum of fractal percolation-like objects.
The GFF is the key building block for \emph{Liouville quantum gravity}, a canonical model of random planar geometry,
given informally by the random Riemmanian metric
\[
    e^{\gamma h(z)}\,\mathrm{d}z,
    \qquad z\in\R^2,
\]
where $\gamma>0$ and $h$ is the GFF.
For background on Liouville quantum gravity, see  \cite{GwyRandomSurfacesLiouville2020,DDGIntroductionLiouvilleQuantum2023,BPGaussianFreeField2025}.
The Brownian motion on such a random Riemannian surface is called \emph{Liouville Brownian motion}.
A characteristic feature of such diffusions is that they exhibit significant \emph{trapping}, getting stuck due to the inhomogeneity of the background, in stark contrast to Brownian motion in Euclidean space.
This naturally leads to the broad question of quantifying the disparity between Euclidean geometry and the geometry induced by such random fractal backgrounds.
This is the central focus of this article.
The model of choice is the aforementioned fractal percolation, a random Cantor set which is arguably the simplest model for a random fractal.
We next define this precisely to facilitate the forthcoming discussion and the statements of our main results. 

\subsection{Fractal percolation}\label{s:intro-fractal-percolation}

Fixing a dimension $d\ge 2$,
the model has two further parameters, $p\in(0,1)$ and an integer $\br\ge 2$.
Given these, we define a random subset of $\R^d$ as follows. 
For each non-negative integer $k\in\Zpos$, each box of the form
\begin{align}\label{e:k-boxes}
    x+[0,\br^{-k}]^d, \qquad x\in \br^{-k}\Z^d
\end{align}
is independently \emph{retained} with probability $p$, and \emph{deleted} with probability $1-p$.
We refer to boxes of the above form as \emph{$\br^{-k}$-boxes}.
Let $A^{(k)}$ be the union of all retained  $\br^{-k}$-boxes.
The fractal percolation configuration $A_\infty$ consists of those points that are retained at \emph{every} scale:
\begin{align}\label{e:fractal-perc-definition}
    A_\infty :=\bigcap_{k=0}^\infty A^{(k)}.
\end{align}
We denote the law of $A_\infty$ by $\P_p$.
To emphasize the dependence on the ambient dimension $d$ and the branching rate $\br$, we will sometimes write $A_\infty(d,\br)$.
We also write $A_{n}=A_{n}(d,\br):=\bigcap_{k=0}^n A^{(k)}$ for the set of points retained up to scale $n$.

It is well known that $A_\infty$ exhibits many phase transitions as $p$ ranges from $0$ to $1$ (for instance, see the surveys \cite{ChaAspectsFractalPercolation1995,RSGeometryFractalPercolation2013,KTRecentProgressFractal2025}).
For a simple example, let $S_n$ be the number of $\br^{-n}$-boxes in $[0,1]^d$ that are retained up to scale $n$ (equivalently, $S_n$ is the number of $\br^{-n}$-boxes covering $A_n\cap[0,1]^d$).
Then $\{S_n\}_{n\ge 0}$ forms a branching process, with
\begin{align}\label{e:firstmoment}
    \E_p[S_n] = \br^{dn}p^n.
\end{align}
Thus $\{S_n\}_{n\ge 0}$ is subcritical for $p<\br^{-d}$, critical for $p=\br^{-d}$, and supercritical for $p>\br^{-d}$.
It follows by translation invariance that $A_\infty$ is a.s. empty for $p\le \br^{-d}$, 
and a.s. nonempty for $p>\br^{-d}$.

The Hausdorff dimension $\dim A_\infty$ of the fractal percolation configuration is easy to guess heuristically from the first moment calculation \eqref{e:firstmoment}.
Namely, when $p>\br^{-d}$, we have
\begin{align}\label{e:dimension}
    \dim A_\infty = d + \log_\br p\qquad\text{almost surely.}
\end{align}
This formula was predicted by \cite{ManFractalGeometryNature1986} and proven by \cite{CCDConnectivityPropertiesMandelbrots1988}.
From now on we slightly abuse notation and denote $\dim A_\infty := \max\{d+\log_\br p, 0\}$.
Note that $\dim A_\infty$ ranges from $0$ to $d$ as $p$ ranges from $0$ to $1$.
Let us mention that a rich dimension theory of fractal percolation has emerged, see \cite{KTRecentProgressFractal2025} for a modern review.

\subsection{Percolative properties}
The next natural question is about connectivity: when does  $A_{\infty}$ have a large connected component?
Such percolative properties have implications in a variety of questions, as will become apparent through the course of this discussion.
Towards understanding such properties let us introduce some definitions.

By a \emph{path} in $\R^d$ we mean a continuous function $\Path : I\to\R^d$, where $I=[0,\infty)$ or $I=[0,T]$ for some $T\in(0,\infty)$.
We define the fractal percolation critical probability as
\[
    p_c(d,\br) := \inf\left\{p : 
        \P_p\bigl(
            \textup{$A_\infty(d,\br)$ contains an unbounded path}
        \bigr)
     > 0\right\}.
\]
It is known that $p_c(d,\br)\in(0,1)$ for all $d\ge 2$ and all $\br\ge 2$ \cite[Theorem 1]{CCDConnectivityPropertiesMandelbrots1988} (the result is formulated only for $d=2$, but the proof works in any dimension).
The key difficulty in the study of this model, as opposed to usual Bernoulli percolation, is that the environment exhibits a hierarchical correlation structure. Nonetheless, one might expect that as the dimension increases the correlation becomes easier to control, leading to the emergence of a more mean-field picture. 
Towards this, Ding, Gwynne, and Zhuang \cite{DGZPercolationThickPoints2026} recently initiated the study of high-dimensional behavior of $p_c$, proving the following theorem.
\begin{theorem}[{\cite[Theorem 1.5]{DGZPercolationThickPoints2026}}]\label{t:dgz-pc}
    $\lim_{d\to\infty}p_c(d,2)=0$.
\end{theorem}

Thus the above statement establishes that, similar to bond percolation, the critical probability converges to $0$ as the dimension $d$ increases, primarily because the model approaches approximately a branching process on a tree of degree $2d$ where the critical percolation probability is $\frac{1}{2d}.$

Going further into geometric properties, a quantification of connectivity of fractal percolation involves the study of the \emph{chemical distance} (internal length metric) and the geometry of the corresponding \emph{geodesics} (shortest paths) which will be often termed as \emph{chemical geodesics}.
Towards this, in sharp contrast to supercritical bond percolation  where the induced metric is coarsely isometric to the Euclidean metric, in \cite{ChaLengthShortestCrossing1996} Chayes proved the following theorem {(see also \cite{OrzLowerBoundBoxcounting1998}).} 

\begin{theorem}[\cite{ChaLengthShortestCrossing1996}]\label{t:chayes-chemical-geodesic}
    In dimension $d=2$, for any $p \ge p_c(2,\br)$,
    there exists $\e=\e(p,\br)>0$ such that almost surely, any path $\Path\subset A_{\infty}$ must have dimension at least $1+\e$. In fact, {he actually proves a discrete version: any crossing of $A_n\cap [0,1]^2$ passes through at least $\br^{(1+\e)n}$ boxes.}
\end{theorem}

The above result reflects that for two points at unit-order distance, 
the chemical geodesic connecting them 
cannot be linear at any scale: every small segment of the path deviates significantly from the straight line joining its endpoints.
This naturally leads to the problem of quantifying \emph{how far} locally linear paths are from chemical geodesics.
Such questions have garnered interest from a geometric measure theory perspective as well.
The following is a representative result.

\begin{theorem}[\cite{BJJ+FractalPercolationUnrectifiable2021}]\label{t:unrectifiable}
    Almost surely, 
    for every Lipschitz path $\Path\subset \R^d$ (i.e. $\Path$ viewed as a continuous function $[0,1]\to \R^d$ is Lipschitz), the $1$-Hausdorff measure of $\Path\cap A_\infty$ is zero. 
    (The reader is also referred to \cite{BJJ+FractalPercolationUnrectifiable2021} for many other results of a similar flavor.)
\end{theorem}

Lipschitz paths are locally linear, so the above result essentially says that any such path's intersection with $A_\infty$ accounts for a vanishing fraction of its total length.
This phenomenon is the object of investigation in this paper through the lens of directed percolation in fractal environments.
Directed percolation is a model of directed fluid flow, wherein one studies discrete analogues of continuum percolative properties of Lipschitz paths.
While the central motivation of the present paper is to understand random directed geometry (last passage percolation) driven by correlated environments such as the Gaussian free field or fractal percolation, to contextualize the main results we will first briefly segue into a general discussion of last passage percolation.

\subsection{Last passage percolation: KPZ universality and beyond}
\emph{Last passage percolation} (LPP) is a model of a zero-temperature directed polymer in a random environment $\{\omega_v\}_{v\in \Z^{d}}$.
For simplicity in the present discussion we assume $d=2$, with definitions for general $d\ge 2$ postponed to Definition \ref{def:directed}.
For any two ordered points in $\Z^{2}$, say $(0,0)$ and $(2^n, 2^n)$, one considers the maximum energy up-right lattice path connecting them, where the energy of a path is the sum of the random variables $\omega_v$ over $v$ lying on the path.
Beyond the connection to polymers, LPP also  models the flow of a fluid through a disordered porous medium with a gravitational force making the fluid move monotonically ``downwards;''
its undirected counterpart, known as \emph{first passage percolation}, was introduced by Hammersley and Welsh \cite{HWFirstPassagePercolationSubadditive1965} as indeed a model for the flow of a fluid through a random medium.
When the environment variables $\{\omega_v\}_{v\in\Z^{2}}$ exhibit fast decay of correlation and are light tailed, the LPP model is expected to exhibit fluctuation exponents characteristic of the Kardar--Parisi--Zhang universality class \cite{KPZDynamicScalingGrowing1986}.
The last few decades have witnessed an explosion of activity aimed at rigorously establishing such predictions (for background, see the surveys \cite{QuaIntroductionKPZ2011,CorKardarParisiZhangUniversality2016,ADH50YearsFirstpassage2017,GanRandomMetricGeometries2022,ZygDirectedPolymersRandom2024} and the references therein).
On the other hand, the universality picture is much less understood when the underlying disorder is correlated, even in the non-rigorous literature.

\addtocontents{toc}{\SkipTocEntry}
\subsection*{LPP on a fractal percolation environment}

While undirected percolation can be understood equivalently in terms of lattice paths or continuous paths, interpreting directed percolation requires more care.
Since the physical relevance of the lattice model is more apparent, we will use it to define directedness, and then afterwards comment on continuum analogues.

\begin{definition}[Directed paths]\label{def:up-right}
    Assume dimension $d=2$.
    For $\e>0$, a lattice path in $\e\Z^2$ is called \emph{directed} (or \emph{up-right}) if it only takes steps in $\{\e\ee_1, \e \ee_2\}$, where $\ee_i$ is the $i\th$ standard basis vector.

    Note that an up-right lattice path can be extended to a continuous path in $\R^2$ by linearly interpolating between consecutive vertices (so each edge is traversed in time $\e$).
\end{definition}

\begin{remark}
    Definition \ref{def:up-right} can be generalized to dimensions $d >2$; one natural such generalization is presented later in Definition \ref{def:directed}.
\end{remark}

\begin{remark}[Continuum analogues of directedness]\label{r:directed-lipschitz}
    Notice that up-right paths in $\e\Z^2$ connecting, say $(0,0)$ and $(1,1)$, extended to continuous paths in $\R^2$ by linear interpolation,
    are of finite variation with their length staying bounded uniformly in the path and in $\e$. Consequently, as the lattice spacing $\e\to 0$, any uniform limit of such directed paths is also Lipschitz.
    In particular, the trace of such a path has Hausdorff dimension $1$. 
    This heuristically suggests that the percolative properties of directed lattice paths should match the percolative properties of paths with $1$-dimensional traces.
\end{remark}

For concreteness we now restrict to the fractal percolation model with branching rate $\br=2$.
The study of directed fractal percolation was initiated by Chayes \cite{ChaAbsenceDirectedFractal1995}, who showed that in dimension $d=2$, for any $p\in(0,1)$, almost surely $A_\infty$ does not contain a directed path connecting the left and right sides of the unit square (see also \cite{ChaLengthShortestCrossing1996,OrzLowerBoundBoxcounting1998,CPPNoDirectedFractal1997}).
Note that Theorem \ref{t:unrectifiable} is actually a stronger statement.
A related result in the setting of the 2D GFF was obtained in \cite{DZLiouvilleFirstPassage2019}; the connection between GFF and fractal percolation is spelled out in the upcoming Section \ref{s:LPP-on-GFF}.

To quantify the failure of directed fractal percolation, we consider the following natural last passage percolation (LPP) model.
For $n\ge0$ and any up-right path $\pi$ in $2^{-n}\Z^2$ connecting $(0,0)$ and $(1,1)$, let $T_n(\pi)$ be the number of $2^{-n}$-boxes contained in $A_n$ that $\pi$ intersects.  
We define the \emph{last passage time} by
\begin{align}\label{e:intro-LPP}
    L_n := \max_\pi T_n(\pi),
\end{align}
where the maximum is over all directed paths connecting $(0,0)$ and $(1,1)$.
So for example, we have $L_n \le 2^{n+1}-1$, with equality if and only if $A_n$ contains an up-right path connecting $(0,0)$ and $(1,1)$.

With this language, the aforementioned result of \cite{ChaAbsenceDirectedFractal1995} says that for any $p\in(0,1)$, almost surely, 
\[
    L_n < 2^{n+1}-1
    \quad\text{for all sufficiently large }n.
\]
Similarly, the result of \cite{DZLiouvilleFirstPassage2019} says (in the related setting of the 2D GFF) that there exists $c\in(0,1)$ such that almost surely,
\[
    L_n \le c 2^{n+1}
    \quad \textup{for all sufficiently large }n,
\]
i.e. at least a constant fraction of any path lies outside the fractal set. It is natural to speculate about the true growth rate. For  reasons discussed shortly, we believe that $L_{n}$ grows as a polynomial with an exponent $\gamma< 1$, i.e., $L_{n}\approx 2^{\gamma n}$ where $\gamma=\gamma(p,d) \to 0$ as $p \to 0.$

Our first result falls slightly short of the above predicted polynomial upper bound and proves a quasi-polynomial (in $2^{n}$) improvement instead.
\begin{theorem}\label{t:sublinear}
    In dimension $d=2$,
    for all $p\in(0,1)$, there exists $c=c(p)>0$ such that for all $n\ge 1$,
    \begin{align}\label{frac lpp eq}
        \bP_p\left(
            L_n \ge 2^n e^{-c\frac{n}{\log n}}
        \right)
        \le e^{-c'n^{10}},
    \end{align}
    where $c'>0$ is a universal constant.
\end{theorem}
The exponent $10$ in the tail bound is arbitrary and the arguments can be modified to yield a tail bound of $e^{-n^C}$ for any $C>0$, provided the constant $c$ is chosen appropriately.  As already evident from the preceding discussion, the $\log n$ factor is expected to be spurious and stemming solely from the inefficiency in the proof method. Further elaboration on this appears later in Remark \ref{gap45}.

\begin{remark}[Beyond $d=2$]
Note that there is a natural version of \eqref{e:intro-LPP} in any dimension $d\ge 2$, given an appropriate higher-dimensional analogue of directedness (see, for example, Definition \ref{def:directed}). The restriction to dimension $d=2$ in Theorem \ref{t:sublinear} is technical in nature, and we briefly explain its origin. As will be elaborated in Section \ref{s:iop}, the proof is based on a multiscale analysis. The base-scale estimate, which is then iterated across scales, takes as input the soft statement that directed fractal percolation fails almost surely. Our proof strategy extends to any dimension in which this input is available. At present, however, the latter has only been established for $d=2$ in \cite{ChaAbsenceDirectedFractal1995} relying on planarity.\footnote{We in fact require a slight generalization of this statement; see Lemma \ref{l:easycross-dep} below.}
Nevertheless, an analogous result is expected to hold in every dimension, as suggested by the aforementioned work \cite{BJJ+FractalPercolationUnrectifiable2021}, proving that Lipschitz curves have $1$-Hausdorff-null intersections with the fractal percolation set for every $d\ge2$. While a further multiscale argument, reminiscent of the framework developed in \cite{DZLiouvilleFirstPassage2019}, is likely to yield the required input in arbitrary dimension, to avoid introducing substantial additional technicalities, we do not pursue this extension here.
\end{remark}

\begin{remark}[The true growth rate]
    As already indicated we expect that the gain in Theorem \ref{t:sublinear} should in fact be polynomial, and the true growth rate of $L_n$ is of order $2^{\gamma n + o(n)}$ for some $\gamma=\gamma(p,d)\in(0,1)$ with $\gamma(p,d)\to 0$ as $p\to 0$.

    To start, we explain why a polynomial lower bound holds.
    For simplicity we assume $d=2$, but the same result can be proved in any dimension.
    Assume $A_\infty$ is non-empty, or equivalently $\dim A_\infty > 0$.
    Fixing any $\gamma\in (0, \dim A_\infty)$,
    a straightforward branching process argument shows that conditional on $A_\infty \cap [0,1]^2$ being non-empty,
    for all sufficiently large $n$
    the set $A_n\cap [0,1]^2$ contains more than $2^{\gamma n}$ many $2^{-n}$-boxes. While sharper arguments may be formulated without too much difficulty, already a quick application of the Erd\H{o}s--Szekeres theorem \cite[Theorem 1.2]{RomSurprisingMathematicsLongest2015} and symmetry shows that with probability at least $1/2$, we have $L_n > \sqrt{2^{\gamma n}}$.

    Turning to the upper bound, an informal justification is essentially given by the fact that if there is no directed percolation, then each scale should contribute to a constant factor (independent of the scale) of closed sites among the sites that are retained from previous scales. Iterating this for the $n$ scales implies that the number of eventual open sites on any directed path must be smaller by a polynomial factor in $2^n,$ the length of the path.

    Our proof strategy is indeed guided by this intuition and further mathematical precision is postponed until Section \ref{s:iop}. Nonetheless, we conclude this discussion presently by recording a few relevant facts along this line.  Note that in the case when $\dim A_\infty <1,$ the trivial bound of $L_n \le S_n$ where $S_n$ is the number of retained $2^{-n}$-boxes in $[0,1]^d$ already yields a bound of the form $2^{\gamma n}$ with $\gamma<1$.
    Indeed, recall from \eqref{e:firstmoment}--\eqref{e:dimension} that
    \begin{align*}
        \E_p[S_n]
        = 2^{(\dim A_\infty)n}.
    \end{align*}
    Thus the main novelty of \cref{t:sublinear} is in the regime $\dim A_\infty \ge 1$. Strong bounds for any value of $\dim A_\infty$ were established in $d=2$ for LPP restricted to a smaller class of $1$-dimensional paths, namely straight lines,
    by Shmerkin and Suomala \cite{SSSetsWhichAre2015,SSLargestSliceFractal2025}.
    In the regime $\dim A_\infty < 1$, \cite[Theorem 1.1]{SSLargestSliceFractal2025} showed that almost surely, the maximum number of points in $A_\infty \cap [0,1]^2$ lying on any straight line is \emph{finite}.
    \footnote{In fact, they identified the critical value of $p$ for which there exist $k$ collinear points, for every $k\ge 3$.}
    Turning to the regime $\dim A_\infty \ge 1$,
    let $L^{\mrm{line}}_n$ be the maximum number of $2^{-n}$-boxes in $A_n\cap[0,1]^2$ intersecting any line.
    \cite[Theorem 1.3]{SSLargestSliceFractal2025} states that if $\dim A_\infty = 1$, then $L^{\mrm{line}}_n$ grows only polylogarithmically in $2^n$.
    Finally, when $\dim A_\infty > 1$, \cite[Theorem 3.1]{SSSetsWhichAre2015} implies that $L^{\mrm{line}}_n$ is at most $O(2^{(\dim A_\infty - 1)n})$, where the implied constant is random (see also the discussion in \cite{SSLargestSliceFractal2025} around (1.2)).
    The latter bound is optimal, since intersecting $A_\infty$ with any axis-parallel line yields a $1$-dimensional fractal percolation, which has dimension $\dim A_\infty - 1$ by \eqref{e:firstmoment}--\eqref{e:dimension}.
    Note that the above polynomial lower bound for $L_n$ implies that $L_n$ is polynomially larger than $L_n^{\mrm{line}}$ for any value of $\dim A_\infty\in (0,2)$, which reflects the significant difference in entropy between the set of directed paths and that of straight lines.
\end{remark}

All the results discussed so far can be summarized as follows.
Fractal percolation, even when connected, possesses a rich fractal geometry that forces any path connecting two points to be quite tortuous, winding around at every scale.
In particular, such paths are very far from directed, and in fact one can  obtain strong quantitative bounds for how small the intersection between any directed path and the fractal percolation set is.  
Conceptually, the main geometric obstruction preventing directed paths from percolating is that their fractal dimension is $1$.
Thus, one may speculate that a variant in which paths proceed monotonically in one direction but traverse the orthogonal directions much \emph{faster} (so as to have a fractal dimension greater than $1$) may permit directed percolation.\\

Our second result confirms that this can indeed be achieved by providing an explicit construction in high enough dimensions. We call this model \emph{unrectifiable directed percolation}.

\subsection{Unrectifiable directed fractal percolation}\label{s:8x10-intro}

To help motivate the construction, we first extend to any dimension $d\ge 2$ the notion of directed paths in the usual Euclidean lattice.

\begin{definition}[Directed paths in the Euclidean sense]\label{def:directed}
    For $\e>0$ and $d\ge 1$, we define the \emph{directed lattice} $\Zdir^{d+1}_\e$ to be the directed graph with vertex set $\e\Z^{d+1}$ and with directed edges
    \[
        (z,w)\in\msf{Edges}(\Zdir^{d+1}_\e) \quad\iff\quad
        \frac{1}{\e}(w-z) \in \{\ee_{d+1}\pm\ee_i : i\in\lb 1,d\rb\}.
    \]
    A \emph{directed lattice path} (or \emph{directed path}) is a path in $\Zdir^{d+1}_\e$.
    Observe that for any path $Q$ in $\e \Z^d$, say $Q = (Q(0), Q(\e), \dots, Q(T))$ for some  $T \in \e\Z_{\ge 1}$, its graph $\{(Q(t), t) : t \in \e\Z \cap [0,T]\}$ is a directed path.
    Conversely, any directed path is the graph of a unique such path $Q$.

    For any directed path $\Path : \e\Z\cap[0,T]\to \Zdir_\e^{d+1}$, we can identify it with the continuous path $[0,T]\to \R^{d+1}$ obtained by linearly interpolating between consecutive vertices in $\Zdir^{d+1}_\e$ so that each edge is traversed in time $\e$.
\end{definition}

\begin{remark}
    Given Definition \ref{def:directed}, the natural $(d+1)$-dimensional generalization of directed fractal percolation considers LPP using directed paths in $\Zdir^{d+1}_{\br^{-n}}$ with respect to the random environment $A_n(d+1,\br)$.
    
    Note that in dimension $d=2$, the above notion of directed paths is equivalent to the previous one (up-right paths), by performing a $45^\circ$ rotation.
    However, the corresponding fractal percolation models are slightly different, as the path's orientation relative to the $\br$-adic boxes differs.
    We used the up-right model in Theorem \ref{t:sublinear} just to simplify the notation in the proof; our arguments are robust and can be modified to also apply for directed paths in the sense of Definition \ref{def:directed}.
\end{remark}

We now define the ``unrectifiable directed fractal percolation'' model alluded to above.
For $k\in\Zpos$, consider boxes in $\R^{d+1}$ of the form
\[
    (x,t) + [0,8^{-k}]^d \times [0,10^{-k}],\qquad (x,t) \in 8^{-k}\Z^d \times 10^{-k}\Z.
\]
As elaborated shortly in Remark \ref{choices} below, the numbers $8$ and $10$ hold no particular significance.
We next define a fractal percolation configuration $\wt{A}_\infty=\wt{A}_\infty(d,8) \subset \R^{d+1}$ by independently retaining the above boxes with probability $p$, exactly as in \eqref{e:k-boxes}--\eqref{e:fractal-perc-definition}.
We denote the law of $\wt{A}_\infty$ by $\wt{\P}_p$.
We also denote by $\wt{A}_n$ the superset obtained by retaining and deleting boxes only at scales $k\in\lb 0,n\rb$.

Define the scaling map $\cS_n:\R^{d+1}\to\R^{d+1}$ by $\cS_n(x,t):=(8^{-n}x, 10^{-n}t)$, where $x\in\R^d$ and $t\in\R$.
Recalling the directed lattice $\Zdir^{d+1}_1$ from Definition \ref{def:directed} (note $\e=1$), we now define a scaled version:
\begin{align}\label{e:8x10-lattice}
    \wt{\Zdir^{d+1}_n} := \cS_n(\Zdir^{d+1}_1).
\end{align}

Let $\mathbb{S}^d$ be the unit $\ell^1$-sphere in $\R^{d+1}$.
For $n,M\ge 1$ and $p\in(0,1)$, define
\begin{align*}
    \theta(n,M,p,d) := \wt{\P}_p\left(
            \textup{$\wt{A}_n(d,8)$ contains a  lattice path in $\wt{\Zdir^{d+1}_n}$ connecting $\mathbb{S}^d$ to $M\mathbb{S}^d$}
        \right).
\end{align*}
We define the critical probability
\begin{align}\label{e:8x10-critical-prob}
    \wt{\dirpc}(d) := 
    \inf\left\{
        p \in (0,1) : 
        \inf_{n\ge 1}\,
        \inf_{M\ge 1}\,
        \theta(n,M,p,d)
        >0
    \right\}.
\end{align}
Note that a priori, it is not even clear that the set on the right-hand side is non-empty. 
Nonetheless, our next result states that as $d\to\infty$, the set actually contains almost every $p$.

\begin{theorem}\label{t:8x10}
    We have $\lim_{d\to\infty} \wt{\dirpc}(d) = 0$.
\end{theorem}

Some remarks are in order.

\begin{remark}\label{choices}
The proof of \cref{t:8x10} can be easily adapted to establish the same result for the model using  boxes of the form 
    \[
        (x,t) + [0, \br^{-k}]^d \times [0,(\br+2)^{-k}],
        \qquad (x,t) \in \br^{-k}\Z^d \times (\br+2)^{-k}\Z
    \]
    for any sufficiently large integer $\br$.
   However, for any percolation probability $p,$ increasing $\br$ (equivalently, decreasing the aspect ratio $\frac{\br+2}{\br}$) increases the minimum dimension in which directed percolation occurs.
\end{remark}

\begin{remark}
        Let us explain the term ``unrectifiable directed percolation.''
        To motivate this we consider the following equivalent construction of $\wt{A}_\infty$.
        Write $\beta:=\log_8(10)$ and consider the scaled $\ell^\infty$ metric on $\R^{d+1}$ given by
        \[
            \rho(z,w)
            :=\max\left\{
                |z_1-w_1|,\dots,|z_d-w_d|, |z_{d+1}-w_{d+1}|^{1/\beta}
            \right\}.
        \]
        Observe that $\wt{A}_\infty$ is the fractal percolation model obtained by defining $A_\infty$ as before, but using $\rho$-balls of diameter $8^{-k}$ instead of  $\ell^\infty$-balls of diameter $8^{-k}$.

        Since $\beta > 1$, changing the metric to $\rho$ has the effect of ``speeding up'' directed paths at small scales. 
        Indeed, consider a directed path in $\wt{\Zdir^{d+1}_n}$ (henceforth a \emph{$\rho$-directed path}).
        Traveling $10^{-n}$ units in the time axis (coordinate $d+1$) forces the $\rho$-directed path to also travel distance $\Theta(8^{-n})$ in space (coordinates $1,\dots,d$).
        In contrast, a Euclidean-directed path in $\Zdir^{d+1}_{8^{-n}}$ is Lipschitz and hence in time $10^{-n}$ only moves distance $\Theta(10^{-n})$.
        However, this microscopic speedup is not alone enough to establish percolation. 
        For example, a uniform limit of $\rho$-directed paths can be a Lipschitz path, which cannot percolate (as indicated earlier).

        To prove Theorem \ref{t:8x10}, we construct, for each $n\ge 1$, a family of $\rho$-directed paths $\{\Path_n\}$
        that move exponentially faster than Euclidean-directed paths at \emph{every scale}.
        Namely, for every $\Path_n$ and every $0\le k\le n$, traveling $10^{-k}$ units of time causes a $\Theta(8^{-k})$ spatial displacement.
        Any uniform limit of such paths $\Path_n$ as $n\to\infty$ must be unrectifiable---hence the name ``unrectifiable directed percolation.''
\end{remark}

\begin{remark}
    Theorem \ref{t:8x10} leaves open the intriguing question of determining whether unrectifiable directed fractal percolation occurs in low dimensions. 
    In dimension $d=2$, this was asked previously in \cite[Question 7.7]{KTRecentProgressFractal2025}. 
    In fact, it is not too difficult to see that planar fractal percolation using boxes of size $\br_1^{-k}\times \br_2^{-k}$ for any integers $\br_1,\br_2 \ge 2$ always fails to exhibit directed connectivity, essentially by the same argument as in \cite{ChaAbsenceDirectedFractal1995}.
    However, it is not unreasonable to speculate  that one should be able to construct unrectifiable directed percolation in dimensions three and higher by appropriately choosing the side lengths.
    Heuristically, this is because in the planar case (dimension $1+1$) where the spatial dimension is $1$, changing the aspect ratio cannot increase the size/dimension of the projection of a directed path onto the spatial coordinates, whereas this is not the case in dimension $d+1$ for $d\ge 2$.
    The reader is also referred to \cite[Section 3.3]{KTRecentProgressFractal2025} for a review of geometric and dynamical properties of fractal percolation using boxes of different aspect ratios.
\end{remark}

\subsection{LPP on the Gaussian free field (GFF)}\label{s:LPP-on-GFF}

We now elaborate on the connection between LPP on fractal percolation and LPP on the 2D GFF.
For simplicity, we actually explain the connection between fractal percolation and the \emph{branching random walk} (BRW), a canonical hierarchical approximation to the 2D GFF.

Fixing a dimension $d\ge 2$ and $n\ge 1$, the BRW is the Gaussian process $X_n$ on $\R^d$ defined as follows.
For every $k\ge 0$ and every $2^{-k}$-box $B$ as in \eqref{e:k-boxes}, we assign an i.i.d. standard Gaussian random variable $\xi_B$.
Then for $x\in \R^d$, the BRW's value at $x$ is 
\begin{equation}\label{brwdef345}
X_n(x) = \sum_{k=0}^{n-1} \xi_{B_k(x)},
\end{equation} where $B_k(x)$ is the unique $2^{-k}$-box containing $x$ (we fix some deterministic rule to break ties in case $x$ lies on the boundary between multiple boxes).

In \cite{GGNLastPassagePercolation2026} the first two named authors and Kyeongsik Nam studied LPP on the BRW and other hierarchical environments, proving that in any dimension $d\ge 2$, with high probability the last passage time on the BRW $X_n$ satisfies (see \cite[Theorem 3]{GGNLastPassagePercolation2026})
\begin{equation}\label{lb342}
    L_n^{\mrm{BRW}} \gs \frac{2^n n^{1/2}}{\log n}.
\end{equation}
An upper bound of $L_n^{\mrm{BRW}}\ls 2^n n$ follows from the well-known fact that the maximum of the BRW is of order $n$ in any dimension \cite{ZeiBranchingRandomWalks2016}.
While the precise exponent governing the $\mathrm{poly}(n)$ pre-factor remains unknown, it is not unreasonable to speculate that the lower bound is sharp, for reasons that will become clearer through the course of this discussion.

To probe the sharp behavior, it is natural to investigate how extrema of the BRW contribute to the last passage time.
As mentioned, the maximum of the BRW grows as $cn(1+o(1))$ for some $c=c(d)>0$ (in fact, $c = \sqrt{2d \log 2}$ \cite[Theorem 7]{ZeiBranchingRandomWalks2016}).
For $\alpha \in (0,c]$, we say that a point $x$ is \emph{$\alpha$-thick} if $X_n(x) \ge \alpha n$.
Let $\Gamma_n(\alpha)$ be the set of $\alpha$-thick points in $[0,1]^d$. Understanding $L_n^{\mrm{BRW}}$ or its GFF counterpart thus naturally warrants studying LPP on the thick point set $\Gamma_n(\alpha)$.
Now, studying LPP on fractal percolation is a natural first step towards this, since fractal percolation is a simple proxy for the thick point set.
Indeed, consider the subset of thick points
\begin{align*}
    S_n(\alpha) 
    = \left\{
        x \in [0,1]^d : \xi_{B_k(x)}
        \ge \alpha\textup{ for all } 0\le k \le n-1
    \right\}
    \subset \Gamma_n(\alpha).
\end{align*}
Then $S_n(\alpha)$ is precisely the fractal percolation model $A_{n-1}(d,2)$ with retention probability
\footnote{For simplicity we ignore the possibility of $x$ lying on the boundary between multiple $2^{-k}$-boxes.}
\begin{align*}
    p = \P(\xi \ge \alpha).
\end{align*}

At the same time as \cite{GGNLastPassagePercolation2026}, Ding, Gwynne, and Zhuang \cite{DGZPercolationThickPoints2026} independently studied high-dimensional aspects of a closely-related problem
(their work will actually play a key role in our proof of Theorem \ref{t:8x10}).
One of their main results was that for any $\alpha>0$ and for all sufficiently large dimension $d$, the set of $\alpha$-thick points contains an unbounded path almost surely (in fact, this is a corollary of their proof of Theorem \ref{t:dgz-pc}).
However, this path is far from Lipschitz/directed, having Hausdorff dimension $\log_8(10)>1$. 
Note that if there were a \emph{directed} crossing of the thick point set $\Gamma_n(\alpha)\cap [0,1]^d$ with constant probability, it would follow that $\E[L_n^{\mrm{BRW}}]\asymp 2^n n$, which by the concentration result \cite[Theorem 4]{GGNLastPassagePercolation2026} would imply that $L_n^{\mrm{BRW}}\asymp 2^n n$ with high probability.
However, we expect that this is not the case.
In fact, it is not unnatural to extrapolate from the proof strategy for Theorem \ref{t:sublinear} to conjecture that for any $\gamma>1/2$, the maximum number of points of value $n^{\gamma}$ on any directed path is at most $2^n e^{- n^{\beta}}$ for some $\beta=\beta(\gamma)>0$, which would then imply the sharpness of the lower bound in \eqref{lb342} up to sub-$\mathrm{poly}(n)$ factors.
Not unrelated to the above theme,  percolative properties of GFF level sets have been intensively studied as a canonical model of a dependent percolation with algebraic decay of correlation, see \cite{DLChemicalDistancesPercolation2018,DWPercolationLevelsetsGaussian2020,APPThickPointsPlanar2023}.\\

We next move on to a brief overview of the key ideas in the proofs of Theorems \ref{t:sublinear} and \ref{t:8x10}.

\subsection{Proof ideas}\label{s:iop}

We first outline the proof of Theorem \ref{t:sublinear}.
The broad idea is to show that with high probability, for any directed path $\pi$, it gets progressively fractured as one reveals the percolation environment, scale by scale. Suppose the path loses a constant fraction (say $1\%$) of its \emph{remaining} length due to the deletion of $2^{-k}$-boxes, for every $1\le k \le n$.
Indeed, this would readily imply that
\begin{align*}
    L_n \le 0.99^{n}\cdot 2^{n+1}.
\end{align*}
In the actual proof, for technical reasons stemming from the necessity of uniform concentration estimates, we require a bit of scale separation, and only establish this length loss along a slightly sparse sequence of scales $k_1<k_2<\cdots<k_M$ with gaps $k_{i+1}-k_i \asymp \log n$.
This then leads to a bound of
\begin{align*}
    L_n \le 0.99^{\Theta(n/\log n)} \cdot 2^{n+1},
\end{align*}
which is precisely Theorem \ref{t:sublinear}.
Our argument establishing that a constant fraction of length is lost per scale consists broadly of two steps.\\

\noindent
\textbf{Renormalization.} The first step is to show that for fractal percolation of depth $k\gg 1$ and a large constant $\ell$, with probability tending to $1$ as $k\to\infty$, any directed path loses at least a $2^{-\ell}$-fraction of its total length \emph{just from the deletion of boxes at scales $k-\ell,k-\ell+1,\dots,k$.}
In essence, this amounts to showing that  with high probability,
\begin{align}\label{e:one-scale-iop}
    L_{k-\ell,k} \le \left(1 - \frac{C}{2^\ell}\right)\cdot 2^{k+1},
\end{align} 
where $L_{k-\ell,k}$ is the last passage time in the fractal percolation configuration defined by only deleting boxes of size $2^{-(k-\ell)}, 2^{-(k-\ell+1)},\dots,2^{-k}$.
We prove \eqref{e:one-scale-iop} using a renormalization argument, which we now sketch a simplified version of.
For a $2^{-(k-\ell)}$-box $\msf{b}$, we say that $\msf{b}$ is \Bad if for every directed path $\pi$ in $2^{-k}\Z^2$ with $\pi \cap \msf{b} \ne \varnothing$, there is at least one deleted $2^{-k}$-box contained in $\pi\cap \msf{b}$.
Roughly speaking, since directed fractal percolation fails almost surely \cite{ChaAbsenceDirectedFractal1995}, we have
\begin{align}\label{e:iop-badness}
    \P(\msf{b} \text{ is \Bad}) \xrightarrow{\ell\to\infty} 1,
\end{align}
and the convergence is uniform in $k$ and $\msf{b}$.
To prove \eqref{e:one-scale-iop}, it suffices to show that with high probability, a constant fraction of $2^{-(k-\ell)}$-boxes along any directed path are \Bad.
Write $\overline{p} := \P(\msf{b} \text{ is not \Bad})$.
Observe that distinct $2^{-(k-\ell)}$-boxes are \Bad independently of each other.
Thus, identifying the set of $2^{-(k-\ell)}$-boxes with a copy of $2^{-(k-\ell)}\Z^2$, 
the task is essentially equivalent to showing that in the i.i.d. $\mathrm{Bernoulli}(\overline{p})$ environment given by $\omega(\msf{b}) := \1_{\msf{b}\text{ is not \Bad}}$, with high probability the last passage time is at most $c2^{k-\ell+1}$ for some $c<1$. It is known, for instance by coupling to a suitable Poisson process, that the last passage time  has an exponential tail at scale $O(\overline{p}^{1/2}2^{k-\ell})$ (see for instance \cite{MarLinearGrowthGreedy2002}). 
Since $\overline{p}\to 0$ as $\ell\to\infty$, we can choose $\ell$ large enough that the above implies a constant fraction of length decay, as desired.\\

\noindent
The actual implementation is nonetheless complicated by lattice effects and other subtleties pertaining to the geometry of directed paths. To overcome this, we actually work with an auxilliary locally dependent fractal percolation model that will be defined in the proof.
\\

\noindent
\textbf{Multiscale analysis.} The second step is to then iterate \eqref{e:one-scale-iop} at every scale.
We start with a directed path $\pi$ in $2^{-n}\Z^2$, and argue using \eqref{e:one-scale-iop} that a constant fraction of its length is lost from scales $k_1-\ell,k_1-\ell=1,\dots,k_1$, where $\ell \le k_1 < n$ (here we use that $\ell$ is just a constant).
After accounting for these scales, the path $\pi$ is fragmented into a  collection of shorter paths.
We then proceed to a finer scale $k_2 >  k_1$ and apply \eqref{e:one-scale-iop} to each of the shorter paths.
We iterate this until we finally reach scale $n$.
As mentioned, we require a scale separation of $k_{i+1}-k_i \asymp \log n$.
While the reasoning behind this precise separation will be made clear in the proof, let us already remark that a growing separation is required to ensure that the exponential tail bound on the passage time in the auxilliary Bernoulli environments constructed in the first step is enough to control the passage times uniformly over all possible  locations of the fragmented path segments.\\

\noindent
\textbf{Unrectifiable LPP.} The main realization behind  Theorem  \ref{t:8x10} is that if one can construct 
undirected connecting paths in fractal percolation, one can then perhaps \emph{lift} them to one higher dimension in an appropriate way to create percolating directed paths (see Figure \ref{fig:lattice-lift}).
Since undirected connections must be tortuous (see Theorem \ref{t:chayes-chemical-geodesic}), one can see, say, via dimension considerations, that for this strategy to be successful the path must move with unequal speeds in different coordinates eventually yielding an un-rectifiable path.
It turns out to be convenient to construct undirected connections whose tortuosity can be controlled at every scale.
Multi-scale constructions of paths had appeared earlier, for instance, in \cite{DGUpperBoundsLiouville2019,GGNLastPassagePercolation2026} and \cite{DGZPercolationThickPoints2026}. For our particular application the framework in the latter paper seems particularly convenient, which we adopt to prove Theorem \ref{t:dgz-pc}. 
We postpone further details to Section \ref{s:8x10}.

\subsection{Organization of the paper}

In Section \ref{s:general-notation} we record notation that will be in use throughout the paper.
In the rest of Section \ref{s:preliminaries}, we set up additional notation and basic inputs needed for the proof of Theorem \ref{t:sublinear}.
The latter is presented in Section \ref{s:sublinear}.
Theorem \ref{t:8x10} is proved in Section \ref{s:8x10}.
In Appendix \ref{a:easycross} we prove a crossing estimate used in the proof of Theorem \ref{t:sublinear}.
Finally, in Appendix \ref{a:return-prob} we prove a lemma on high-dimensional behavior of a certain random walk return probability used in the proof of Theorem \ref{t:8x10}.

\addtocontents{toc}{\SkipTocEntry}
\subsection*{Acknowledgements}

SG was supported by a Miller Research Professorship at the
Miller Institute for Basic Research in Science, NSF Career grant-1945172 and NSF grant-2553225.
VG was supported by NSF Career grant-1945172, and the NSF
Graduate Research Fellowship Program under Grant No. DGE-2146752.
KJ was supported by NSF Career grant-1945172 and the Citadel Securities Ph.D. Fellowship.

\addtocontents{toc}{\SkipTocEntry}
\subsection*{AI use disclosure}

A large language model was used to generate Figure \ref{fig:lattice-lift} and \LaTeX\  code for the symbol $\boxup$.

\section{Elements of fractal percolation and a priori estimates}\label{s:preliminaries}
In this section, we introduce some further notation and recall previously used notation, set up definitions, and record various results that will be used throughout the proof of Theorem \ref{t:sublinear}.

\subsection{General notation}\label{s:general-notation}
Throughout the paper we denote by $C$, $c$, etc., finite positive constants that do not depend on any parameter unless we clarify the way they depend on parameters and variables.
Uppercase constant $C$ may increase from one line to the next, and lowercase $c$ may decrease from one line to the next.
For a finite set $A$, we denote by $|A|$ the cardinality of $A$.
For any point $a\in \bR^{2}$, the set with $a$ as the single element is written as $a= \{a\}$. We will denote by $a_1$ and $a_2$ its two coordinates respectively.
For real numbers $s<t$ we write $\lb s,t\rb := [s,t]\cap\Z$.
For $d\ge 1$, we denote by $\dist$ the $\ell^1$-metric on $\R^d$, i.e. $\dist(x,y) = \sum_{i=1}^d |x_i-y_i|$.
For $x\in\R^d$ and $A\subset \R^d$, we denote $\dist(x,A):=\inf_{y\in A}\dist(x,y)$.

\subsection{Planar geometry notation}

For $k\in\Zpos$, we denote by $\Lambda_k$ the collection of all faces of $2^{-k}\bZ^2$, i.e.,
\begin{align*}
    \Lambda_{k}:= 
    \left\{x + [0,2^{-k})^{2} : x \in 2^{-k}\Z^{2}\right\}.
\end{align*}
We view $\Lambda_k$ as the dual lattice of $2^{-k}\bZ^2$, and refer to elements of $\Lambda_k$ as \emph{boxes}.
Later, we will encounter situations requiring relations between subsets $A \subset \Lambda_k$ and $B\subset \Lambda_{k'}$ for $k \geq k' \ge 0$. In such cases, we adopt a slight abuse of notation by introducing the following convention: the intersection $A\cap B$ means both the set-theoretic intersection of $A$ and $B$ viewed as subsets of $\R^2$, as well as the collection of boxes in the finer lattice $\Lambda_k$ that are contained in the former set-theoretic intersection. Formally, we define
\[
    A \cap B := \{\su \in \Lambda_k: \su \subset A \text{ and }\su \subset B\},
\]
where $\su\subset A$ (resp. $\su \subset B$) is understood by viewing $A$ and $B$ as subsets of $\bR^2$. 
Similarly, we write $\su \in B$ to mean that $\su \subset B$ when $\su$ and $B$ are viewed as subsets of $\bR^2$. 
Note that this avoids ambiguity: any subset of the coarser lattice $\Lambda_{k'}$ can be uniquely identified with a subset of $\bR^2$, which in turn admits a canonical partition into elements of the finer lattice $\Lambda_k$.
The same convention also applies to other set operations such as union $\cup$ and difference $\setminus$.

For any box $\su = x + [0,2^{-k})^{2}\subset \R^2$, we denote by $\mathsf{bl}(\su)$, $\mathsf{tr}(\su)$, and $\cen(\su)$ the bottom-left corner, top-right corner, and center of $\su$ respectively, i.e.
\begin{align*}
    \mathsf{bl}(\su)&:= x;\\
    \mathsf{tr}(\su)&:= x + (2^{-k},2^{-k});\\
    \cen(\su) &:= x + (2^{-k-1},2^{-k-1}). 
\end{align*}
For $a=(a_1,a_2)$, $b=(b_1,b_2)\in \R^2$, we write $a\preceq b $ if $a_i \le b_i$ for $i=1,2$, and $a\prec b$ if $a_i <b_i$ for $i=1,2$. 
For any boxes $\su, \sv \in \Lambda_k$, we write $\su \preceq \sv$ if $\cen(\su) \preceq \cen(\sv)$.
For any $k\in \Z_{\ge 1}$ and any box $\su\in \Lambda_{k}$, we denote by $\su^{\boxup}$ the box of size $2^{-k+1}\times 2^{-k+1}$ with $\bl(\su^{\boxup})=\bl(\su)$, i.e., the union of the four boxes in the superscript with the shaded box being $u$. 
That is, 
\begin{align}\label{e:boxup}
    \su^{\boxup} := 
    \bl(\su)
    + [0,2^{-k + 1})^{2}.
\end{align}
Equivalently, identifying $\su^{\boxup}$ with the four boxes in $\Lambda_{k}$ that cover it, we can write
\begin{align}\label{e:boxup-collection}
    \su^{\boxup} = 
     \left\{
        \su,\;\;
        \su + (2^{-k},0),\;\;
        \su + (2^{-k}, 2^{-k}),\;\;
        \su + (0, 2^{-k})
    \right\}.
\end{align}
{This identification will indeed turn out to be convenient in setting up some of the later notations that our arguments will rely on.}
For any subset $A\subset \bR^{2}$, we denote by $\Size_{k}(A)$ the number of boxes in $\Lambda_{k}$ contained in $A$, i.e. 
\begin{align}\label{e:size}
    \Size_{k}(A) := |\{ \su\in \Lambda_{k}: \su \subset A\}|.
\end{align}

\subsection{Fractal percolation}
Although fractal percolation was already defined in Section \ref{s:intro-fractal-percolation}, for our arguments we will need more {flexible} notation, which we now set up.
For any $k\in \Z_{\ge 1}$, let $\{0,1\}^{\Lambda_k}$ be the set of percolation configurations on $\Lambda_{k}$.
For $p\in(0,1)$, we denote by $\omega_{k}^{p}\in \{0,1\}^{\Lambda_k}$ the percolation configuration of the $k\th$ level of fractal percolation, i.e. 
$\omega_k^p = \{\omega_{k}^{p}(\su)\}_{\su \in \Lambda_{k}}$ where $\omega_{k}^{p}(\su)$s are i.i.d. Bernoulli random variables with parameter $p\in (0,1)$.
We denote the law of $\omega_k^p$ by {$\bP_{p}$}.
When the choice of $p$ is understood from context, we drop the superscript $p$ and just write $\omega_{k}$, $\omega_{k}(\su)$ and $\bP$.
{\textit{With a slight abuse of notation (which presents no risk of confusion), we also use $\omega_{k}$ to represent the union of boxes $\su$ in $\Lambda_{k}$ with $\omega_{k}(\su) = 1$ as a subset of $\bR^{2}$, i.e.}}
\begin{align*}
    \omega_{k}:= \bigcup_{\su\in \Lambda_{k}:\omega_{k}(\su)=1} \su.
\end{align*}
The fractal percolation configuration of depth $k$ is defined as $\sigma_{k} := \bigcap_{0< i\le k} \omega_{k}\subset \bR^{2}$.
Equivalently, for any $\su \in \Lambda_k$,
\begin{align*}
\sigma_{k}(\su)  :=
        \begin{cases}
            1, & \text{if $\omega_{i}(\sv) = 1$ for all $0< i\le k$ and $\sv \in\Lambda_{i}$ with $\su \subset \sv$;}\\
            0, & \text{otherwise.}
        \end{cases}
\end{align*}
We call a box $\su$ \emph{retained} if $\sigma_{k}(\su)=1$ and \emph{deleted} if $\sigma_k(\su)=0$.
For $k>k'\ge 0$, we define $\sigma_{k',k}:= \bigcap_{k'<i\le k} \omega_{i}\in \bR^{2}$.

\begin{remark}\label{r:half-open}
    Note that instead of defining fractal percolation using half-open boxes of the form $x+[0,2^{-k})^{2}$, the standard definition typically uses closed boxes of the form $x+[0,2^{-k}]^{2}$. 
    The fact that half-open boxes partition the plane will be notationally convenient for us, although our arguments can be easily adapted to handle closed boxes.
\end{remark}

\subsection{Directed paths and last passage percolation}

\begin{definition}[Directed paths]
    For any $k\in \Zpos$, a \emph{directed path} $\pi_k$ on $\Lambda_{k}$ is a sequence of boxes $\pi_{k} = (\su_{1},\cdots ,\su_{t})$, where $\su_{i} \in \Lambda_{k}$ for all $i\in\lb 1,t\rb$ and $\cen(\su_{i+1}) - \cen(\su_{i}) \in \{(2^{-k},0),(0,2^{-k})\}$ for all $i\in\lb 1,t-1\rb$.
   {Similar to  how $\omega_k$ can be used to represent both the configuration and the corresponding subset of $\R^2$, we also use $\pi_{k}$ to represent the union of boxes on $\pi_{k}$ as a subset of $\R^{2}$.}
    
    For any $a,b\in \R^2$, we say that a directed path $\pi_{k}=(\su_{1},\cdots ,\su_{t})$ \emph{connects $a$ and $b$} if $a\in \su_{1}$ and $b\in \su_{t}$.
    We denote by $\Pi_{k}(a,b)$ the collection of all directed paths on $\Lambda_{k}$ connecting $a$ and $b$.    
\end{definition}

\begin{definition}[Last passage time]\label{d:lpp}
For any directed path $\pi_{k}$ on $\Lambda_{k}$ and any percolation configuration $\tau_{k} \in \{0,1\}^{\Lambda_k}$, we define the \emph{passage time} of $\pi_{k}$ with respect to $\tau_{k}$ as 
\begin{align}\label{e:passage time}
    N_{\tau_{k}}(\pi_{k}):= \sum_{\su\in \pi_{k}} \tau_{k}(\su) = \Size_{k}(\pi_{k}\cap \tau_{k}).
\end{align}
For $a,b\in\R^2$, we define the \emph{last passage time} from $a$ to $b$ with respect to $\tau_{k}$ as
\begin{align}\label{e:lpp}
    L_{\tau_{k}}(a,b):= \max_{\pi_{k}\in \Pi_{k}(a,b)} N_{\tau_{k}}(\pi_{k}).
\end{align}
\end{definition}

Notice that for $\tau_k = \sigma_k$, the above is equivalent to \eqref{e:intro-LPP}, since a directed path $\pi_k = (\su_1,\dots,\su_t)$ in $\Lambda_k$ is naturally identified with the directed path $(\bl(\su_1),\dots,\bl(\su_t))$ in the sense of Definition \ref{def:up-right}.

\subsection{Geometry of directed paths}
In this subsection we discuss elementary geometric properties of directed paths.
{The proofs of the following geometric lemmas are straightforward from the definition. Thus, we just record the results here and omit the proofs. }

The first lemma compares the length of a directed path to the $\ell_{1}$-distance between its endpoints, where the latter is denoted by $\dist(\cdot,\cdot)$ (see Section \ref{s:general-notation}).

\begin{lemma}\label{length of path}
    For any $k\in \Zpos$ and any directed path $\pi_{k}= (\su_{1},\cdots, \su_{t})$ on $\Lambda_{k}$, we have 
    \begin{align*}
  {      t = \Size_{k}(\pi_{k})}
        &= 2^{k} \cdot \dist(\mathsf{bl}(\su_{1}),\mathsf{tr}(\su_{t})) -1\\
        &= 2^{k} \cdot \dist(\cen(\su_{1}),\cen(\su_{t})) +1.
    \end{align*}
    Moreover, if $\pi_{k}$ connects $a,b \in \bR^{2}$ with $a \prec b$, then 
    \begin{align*}
      2^{k} \cdot \dist(a,b) -1
      \le  \Size_{k}(\pi_{k})
      \le 2^{k}\cdot \dist(a,b) + 3.
    \end{align*}
\end{lemma}

The next lemma states a non-backtracking property of directed paths.
\begin{lemma}\label{nonbacktracking}
    For every $k\in\Zpos$ and every directed path $\pi_{k}= (\su_{1},\cdots,\su_{t})$ on $\Lambda_{k}$, the following holds.
    Fix any dyadic box $\sa \in \bigcup_{i\ge 0}\Lambda_i$ such that $\pi_{k}\cap  \sa \neq \varnothing$,   and let
    \[
        t_0:= \min\{ i\ge 1: \su_{i}\cap \sa \neq \varnothing\},
        \qquad 
        t_1:=\inf\{ i \ge t_{0}: \su_{i}\cap \sa = \varnothing\}.
    \]
    Then we have
    \begin{align*}
      \left(\bigcup_{t_1 < i \le t} \su_{i} \right)\cap  \sa
      =\varnothing,
    \end{align*}
    with the convention that if $t_1=\infty$ then the union equals $\varnothing$. Thus, in words, once a path leaves a box, it never returns to it.
\end{lemma}

Our last lemma on directed path geometry concerns restrictions of directed paths to coarser lattices.
For any integers $k>k'\ge 0$ and any directed path $\pi_{k}= (\su_{1},\cdots,\su_{t})$ on $\Lambda_k$,
we define the \emph{projection} of $\pi_{k}$ onto level $k'$ as
\begin{align}\label{e:proj}
    \Proj_{k'}(\pi_{k}) 
    :=
    \left\{\sv\in\Lambda_{k'} : \su_i \subset \sv\text{ for some $i\in\lb 1,t\rb$}\right\}.
\end{align}
That is, $\Proj_{k'}(\pi_k)$ is the set of boxes in $\Lambda_{k'}$ through which $\pi_k$ passes.
We label these boxes as $\Proj_{k'}(\pi_k)=\{\sv_1,\dots,\sv_m\}$ according to the order in which $\pi_k$ first visits them (so $\su_1\subset\sv_1$).
Note that by directedness of $\pi_k$, we must have $\sv_1\preceq\cdots\preceq \sv_m$, i.e. $\Proj_{k'}(\pi_k)$ is a directed path on $\Lambda_{k'}$.
Here we are implicitly using the fact that since lattice paths cannot directly jump from the top-right corner of a box to the bottom-left corner of the diagonally-adjacent box, projecting indeed produces a lattice path.
We summarize this discussion in the following lemma.
\begin{lemma}\label{projection of path}
    For any integers $k>k'\ge 0$, any $a,b\in \bR^{2}$ with $a\prec b$, and any directed path $\pi_{k}$ on $\Lambda_{k}$ connecting $a$ and $b$, the projection $\Proj_{k'}(\pi_{k})$ is a directed path on $\Lambda_{k'}$ connecting $a$ and $b$.
\end{lemma}
As usual, we  abuse notation by identifying $\Proj_{k'}(\pi_k)$ with the union of the boxes comprising it.

\subsection{Large deviations for Bernoulli last passage percolation}

A defining feature of fractal percolation is its hierarchical structure, consisting of superimposed Bernoulli percolation configurations at different spatial scales.
In our proof of Theorem \ref{t:sublinear}, we use a large deviations estimate for Bernoulli last passage percolation (LPP) to control the contribution of each Bernoulli layer to the overall last passage time.
We now formulate the key input on Bernoulli LPP large deviations.

Let $\{X_{v}\}_{v\in \bZ^{2}}$ be i.i.d. Bernoulli random variables with parameter $p\in(0,1)$. 
For any $n\ge 1$, let
\begin{align}\label{LPP123}
    S(n) := \max_{\pi} \sum_{v\in \pi} X_{v},
\end{align}
where the maximum is over all directed paths $\pi$ of length $n$ started at the origin. 
Note that the endpoints of the paths are not constrained.
The following lemma asserts an upper tail large deviations estimate for $S(n)$. 
A proof appears in the proof of \cite[Proposition 2.2]{MarLinearGrowthGreedy2002}.
\begin{lemma}[Large deviations for Bernoulli LPP]\label{blpp tail}
    There exist constants $c_* > 0$ and $c > 0$ such that for all $n \ge 1$, all $p \in (0,1)$, and all $t > c_* n p^{1/2}$,
    \begin{align}\label{blpp tail eq}
        \P_p(S(n) > t) \le e^{-ct}.
    \end{align}
\end{lemma}

\begin{remark}
It is not too difficult to see why the dependence on the parameters is of the form $n p^{1/2}.$
Consider tessellating the box $[0,n]^2$ with boxes of size $\frac{1}{p^{1/2}}\times\frac{1}{p^{1/2}}$.
Since each such box typically contains $O(1)$ many nonzero weights, $S(n)$ should be approximated by the last passage time across the box $[0,np^{1/2}]^2$ in a Poisson point process of intensity $1$.
It is well known that the latter is of order $np^{1/2}$ with an exponential upper tail at this scale, provided that $np^{1/2}$ is bounded away from $0$ (e.g. \cite{TalNewLookIndependence1996,ADHammersleysInteractingParticle1995}).
When $np^{1/2}=o(1)$, typically every Bernoulli variable in $[0,n]^2$ vanishes, so the upper tail of $S(n)$ is governed by the number of nonzero weights in $[0,n]^2$.
This leads to a Poisson-type tail at scale $n^2p\ll np^{1/2}$, and in particular \eqref{blpp tail eq} is lossy.
\end{remark}

\section{Fractal last passage time upper bound}\label{s:sublinear}

With the above preparation we are now in a position to embark on the proof of Theorem \ref{t:sublinear}.
As discussed already in Section \ref{s:iop}, the argument combines a renormalization step with multi-scale analysis.  
Throughout this section we adopt the notation of Section \ref{s:preliminaries}.

\subsection{Analysis of a few scales}\label{one scale sec}
The main result of this section is an a priori estimate for the last passage time (\cref{p:onescale}). 
In fact, our application will demand that we develop fractal percolation estimates for a $1$-dependent percolation model rather than the standard independent one; see \eqref{e:def-easy} and \eqref{e:def-dep}. Accordingly, the main proposition below is formulated for a general class of fractal percolation models. We refer to such models as a fractal percolation field, which we now introduce.
\subsubsection{Fractal percolation field}\label{s:fractalfield} 
For a sequence of independent percolation configurations $\tomega = \{\tomega_k \in \{0,1\}^{\Lambda_k}: k \in \bZ_{\ge 1}\}$, define 
    \begin{align*}
    \tsigma_{k^{\prime},k} &:= \bigcap_{i=k^{\prime}+1}^{k} \tomega_i\quad\text{for any $k > k^{\prime} \ge 0$;}\\
    \tsigma_k &:= \tsigma_{0,k};\\
    \tsigma &:= \{\tsigma_{k^{\prime},k}\}_{0 \le k^{\prime} < k}.
    \end{align*}
We call the pair $(\tomega, \tsigma)$ a \emph{fractal percolation field} if for any $k \in\Z_{\ge 1}$, the law of $\tomega_{k}$ is invariant under translation by $2^{-k}\bZ^2$ {and reflection across the line $y=x$.}
Note that $\tsigma$ is measurable with respect to $\tomega$.
Further, note that we stipulated the configurations $\tomega_k$ to be independent across $k$, while not requiring them to be individually i.i.d. fields.
This is because in our later applications, they will possess a local $1$-dependent structure.

\subsubsection{Horizontal and vertical crossings}
Fix any rectangle $[0,m) \times [0,n)$ with $n,m\in \bZ_{\ge 1}$ and any directed path $\pi$.
We say that $\pi$ \emph{crosses $[0,m)\times [0,n)$ vertically} if the {closure} of $\pi$ (as a subset of $\R^2$) connects some point $a \in [0,m)\times \{0\}$ to some point  $b\in [0,m)\times \{n\}$.
Similarly, we say that $\pi$ \emph{crosses $[0,m)\times [0,n)$ horizontally} if the closure of $\pi$ connects some point in $\{0\}\times [0,n)$ to some point in $\{m\}\times [0,n)$. When $m > n,$ we refer to vertical crossings of $[0,m)\times [0,n)$ and horizontal crossings of $[0,n)\times [0,m)$ rectangle as \emph{easy crossings} conforming to the standard percolation terminology.

To make use of the above notions of easy crossings we introduce some further notation.
We consider a fractal percolation field $(\tomega,\tsigma)$. For any integers $k > k' \ge0$ and any box $\mathsf{u} \in \Lambda_{k'}$, define the events
\begin{align*}
    &V(\tsigma_{k^\prime,k},\su):=\left\{\text{$\exists$ directed path $\pi_{k}\subset \tsigma_{k',k}$ crossing $\bl(\su) + [0,2^{-k'+1}) \times [2^{-k'},2^{-k'+1})$ vertically}\right\},\\
    &H(\tsigma_{k^\prime,k}, \su):= \left\{\text{$\exists$ directed path $\pi_{k}\subset \tsigma_{k',k}$ crossing $\bl(\su) + [2^{-k'},2^{-k'+1}) \times [0,2^{-k'+1})$ horizontally}\right\}.
\end{align*}
Thus the above events entail that the on-scale $2\times 1$ square $\bl(\su) + [0,2^{-k'+1}) \times [2^{-k'},2^{-k'+1})$ and its reflection across $y=x$ line admit easy crossings.  
See Figure~\ref{f:easycrossing}.
Given the above, we next define the following
percolation configuration  $\Easy(\tsigma_{k',k}) \in \{0,1\}^{\Lambda_{k'}}$:
\begin{align}\label{e:def-easy}
        \Easy(\tsigma_{k',k})(\mathsf{u})
        := 
        \begin{cases}
            1, & \text{$V(\tsigma_{k',k},\su) \cup H(\tsigma_{k',k},\su)$ occurs}\\
            0, & \text{otherwise}
        \end{cases}
        \qquad\quad \text{for }\su \in \Lambda_{k'}.
\end{align}
\begin{figure}[htb]
	\includegraphics[width=0.6\textwidth]{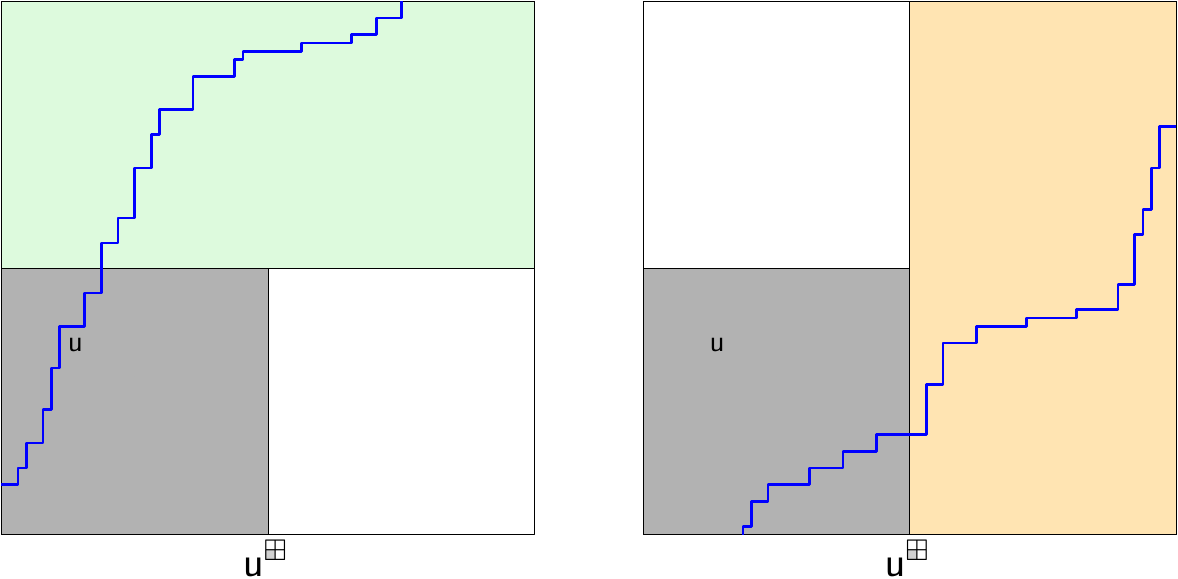}	
    \caption{
    \textbf{Left:} An illustration of the event $V(\tsigma_{k',k},\su)$ ($\su$ is the gray box in the bottom-left corner) that a directed path in $\tsigma_{k',k}$ (the blue path) crosses the green rectangle vertically.
    \textbf{Right:} An illustration of the event $H(\tsigma_{k',k},\su)$ that a directed path in $\tsigma_{k',k}$ crosses the orange rectangle  horizontally. While the events simply stipulate easy crossings, in the above figure the intersection of the blue path with the box $\su$ is also illustrated to aid the discussion in Remark \ref{r:easy}.}
    \label{f:easycrossing}
\end{figure}
{Note that the law of $\Easy(\tsigma_{k',k})$ is invariant under translation by $2^{-k'}\Z^2$. }
\begin{remark}\label{r:easy}
    At first glance, the definition of $\Easy(\tsigma_{k',k})(\mathsf{u})$ may appear counterintuitive, as it depends on the percolation configuration $\tsigma_{k',k}$ within the region $\mathsf{u}^{\boxup} \setminus \mathsf{u}$ (recall from \eqref{e:boxup} that the shaded region is $\mathsf{u}$) rather than $\mathsf{u}$ itself.
    However, this construction is motivated by the following observation: any directed path $\pi_{k}$ connecting the left or bottom edge of $\mathsf{u}$ to the right or top edge of $\mathsf{u}^{\boxup}$ must inevitably either cross $\bl(\mathsf{u}) + [0,2^{-k'+1}) \times [2^{-k'},2^{-k'+1})$ vertically or  $\bl(\mathsf{u}) + [2^{-k'},2^{-k'+1}) \times [0,2^{-k'+1})$ horizontally. 
    Consequently, if $\Easy(\tsigma_{k',k})(\mathsf{u}) = 0$, then there must exist a $2^{-k} \times 2^{-k}$ box $\mathsf{v} \in \mathsf{u}^{\boxup} \cap \pi_k$ such that $\tsigma_{k', k}(\mathsf{v})=0$.
    The primary advantage of this definition is that it relaxes the restriction on the starting point of the directed path $\pi_k$.
    In our setting, segments of directed paths entering $\mathsf{u}$ can originate from any point on its bottom or left edges; this definition naturally accommodates such flexibility without requiring the path to start precisely at the bottom-left corner.
    This will be quite useful later in the proof of Theorem \ref{t:sublinear}.
\end{remark}

\subsubsection{The main a priori estimate for the last passage time}
We now proceed to the main proposition of this section which will serve as the key technical input in the proof of Theorem \ref{t:sublinear}. It states that with high probability, every directed path loses a constant fraction of its length just from the deletion of boxes in the preceding $O(1)$ many scales relative to the scale of the path. 
This makes rigorous the informal estimate \eqref{e:one-scale-iop}  presented earlier.
The proof of Theorem \ref{t:sublinear} will then proceed by bootstrapping this via a multiscale argument (see Section \ref{s:multiscale}).

We will also state the result somewhat generally to aid later applications and consequently need a bit more notation that is presented next.
For any $k\in \Zpos$ and $(\mu,\nu)\in\{1,2,3\}^{2}$, we define the sublattice of $\Lambda_{k}$ (recall $\Lambda_k$ is naturally identified with $2^{-k}\Z^2$) obtained by considering the natural copy of $3\cdot 2^{-k}\Z^2 \subset \Lambda_k$ translated by $(\mu 2^{-k},\nu 2^{-k})$:
\begin{equation}\label{e:separated-lattice}
    \begin{split}
        \Lambda_{k}^{(\mu,\nu)}
        &:= 
        \left\{
            \su\in\Lambda_k : 
            \bl(\su) - (\mu 2^{-k},\nu 2^{-k}) \in 3\cdot 2^{-k}\Z^2
        \right\}\\
        &= \left\{
            \su \in \Lambda_k : 
            2^k \bl(\su) \equiv (\mu ,\nu) \pmod{3}
        \right\}.
    \end{split}
\end{equation}

\begin{proposition}\label{p:onescale}    
    Let $(\tomega, \tsigma)$ be a fractal percolation field as defined in Section~\ref{s:fractalfield} satisfying the following assumptions:
    \begin{enumerate}[label={\rm(A\arabic*)}]
        \item\label{assump:A1} \textup{(Easy directed crossings are rare uniformly across scales).} For any $\e \in (0,1)$, there exists $\tell_{0}(\varepsilon) \in \Z_{\ge 1}$ such that for all $k \ge \tell_{0}$,
        \begin{align*}
            \P\left(\textup{there exists a directed path $\pi_{k}\subset \tsigma_{k-\tell_0, k}$ crossing $[0,2)\times [0,1)$ vertically}\right) \le \varepsilon.
        \end{align*}
        \item\label{assump:A2} \textup{(Finite-range interactions).} For any $(\mu,\nu)\in\{1,2,3\}^{2}$ and any $k > k' \ge1$, 
        \begin{align*}
            \left\{\Easy(\tsigma_{k',k})(\su): \su \in \Lambda_{k'}^{(\mu,\nu)}\right\} \textup{ are independent Bernoulli random variables.}
        \end{align*}
        \textup{It follows from translation invariance that they are in fact i.i.d..}
    \end{enumerate}
    Then there exist constants $c>0$, $\tell\in\Z_{\ge 1}$ and $\tal \in (0,1)$, all depending on $\tell_{0}(\e_0)$ for a universal constant $\e_0$ such that the following holds.

    Fix $a, b \in \bR^2$ with $a\preceq b$, and let $\sa,\sfb$ be the unique boxes in $\Lambda_{\tell}$ containing $a,b$ respectively.
    Assume that $\dist(\bl(\sa),\tr(\sfb)) \ge 2$.
    Then for all $k \ge \tell + 7$, 
    with probability at least $1-e^{-c2^{k}}$ it holds that
    \begin{center}
        \textup{for any directed path $\pi_k$ in $\Lambda_k$ connecting $a$ and $b$, there exists $(\mu,\nu)\in\{1,2,3\}^2$ such that}
        \begin{equation}\label{e:onescale}
            \begin{split}
            \left|\left\{\mathsf{u} \in \Proj_{k-\tell}(\pi_{k}) \cap \Lambda_{k-\tell}^{(\mu,\nu)}:\Easy(\tsigma_{k-\tell,k})(\mathsf{u})=0\right\}\right| 
            &\ge 
            \frac{1}{18}\cdot \Size_{k-\tell}(\Proj_{k-\tell}(\pi_{k}))\\
            &\ge
            \tal \cdot \dist(a,b) \cdot 2^{k}.
            \end{split}
        \end{equation}
    \end{center}
   In fact, the value of $\tal$ can be chosen to be $\frac{1}{36} \cdot 2^{-\tell}$. The following last passage time bound also holds on the above event:
    \begin{align}\label{e:constant-frac-loss}
        L_{\tsigma_k}(a,b)
        \le \left(1-\frac{1}{36} \cdot 2^{-\tell}\right)\cdot \dist(a,b)\cdot 2^{k}.
    \end{align}
\end{proposition}

The proof relies on a renormalization argument.
The key idea is to construct a coarse Bernoulli environment with small density $\overline{p}$ using the $O(1)$-many, i.e, {$k-\tell$ to $k$},  finest scales of the fractal percolation model.
We then derive the stated bounds using the large deviations estimate for i.i.d. Bernoulli last passage percolation recorded earlier in Lemma \ref{blpp tail}.

\begin{proof}[Proof of Proposition \ref{p:onescale}] Fix $a \preceq b$ as in the statement.
Write $K := \Size_{k-\tell}(\Proj_{k-\tell}(\pi_{k}))$, where $\pi_{k}$ is any directed path in $\Lambda_{k}$ connecting $a$ and $b$, and $\tell$ will be chosen later.
Note that $K$ does not depend on the choice of $\pi_k$ since by Lemma \ref{projection of path}, $\Proj_{k-\tell}(\pi_{k})$ is a directed path connecting $a$ and $b$ in $\Lambda_{k - \tell}$.
Moreover, by Lemma \ref{length of path}, we have
\begin{align}\label{lower bound K}
     K \ge 2^{k - \tell} \dist(a,b) - 1
     \ge 2^{k - \tell + 1} - 1
     \ge 2^{k - \tell},
\end{align}
by the inequality $2x - 1 \ge x$ for $x \ge 1$.
We now carry out the two steps outlined above.\\

\noindent
\textbf{Step 1: Creating a Bernoulli environment.}
Fix any $\varepsilon>0$ satisfying 
$180 c_* \varepsilon^{1/2} < 1$,
where $c_*$ is the universal constant appearing in Lemma~\ref{blpp tail}.
To construct the desired environment, we consider $\Easy(\tsigma_{k-\tell,k})$ on $\Lambda_{k-\tell}$, where $\tell=\tell_0(\e/2)$ is as in \ref{assump:A1}.
Then we claim that for any $k\ge\tell$ and any $\mathsf{u} \in \Lambda_{k-\tell}$, we have
\begin{align}\label{e: para-bound}
    \overline{p} := \bP\left(\Easy(\tsigma_{k-\tell,k})(\mathsf{u}) = 1\right) \le \varepsilon.
\end{align}
Note that $\overline{p}$ does not depend on $\mathsf{u}$ due to the translation invariance of $\Easy$.
Given \eqref{e: para-bound}, by our choice of $\varepsilon$, it follows that for all $k\ge\tell + 7$ we have
\begin{align}\label{e:pbar-bound}
    c_*\overline{p}^{1/2}  < 
    \frac{1}{180}.
\end{align}
To see \eqref{e: para-bound}, we recall the definition of $\Easy(\tsigma_{k-\tell,k})(\mathsf{u})$ in \eqref{e:def-easy} and have 
\begin{align*}
\P\left(\Easy(\tsigma_{k',k})(\mathsf{u}) = 1\right) \le \P(V(\tsigma_{k-\tell,k},\su))+\P(H(\tsigma_{k-\tell,k},\su)).
\end{align*}
Note that the event $V(\tsigma_{k-\tell,k},\su)$ is a rescaled version of the event in \ref{assump:A1} at scale $\tell$, and the event $H(\tsigma_{k-\tell,k},\su)$ is a rescaled and reflected (across the line $y = x$) version of the same at scale $\tell$.
Thus, by \ref{assump:A1} and our choice of $\tell$, we have
\[
\P(V(\tsigma_{k-\tell,k},\su))+\P(H(\tsigma_{k-\tell,k},\su)) \le \e /2 + \e/2 = \e.
\]

\noindent
\textbf{Step 2: Deriving the exponential tail.}  
Since the $\Easy(\tsigma_{k-\tell,k})$ variables are not assumed i.i.d. but may have a local $1$-dependency structure, it will be cleaner, at the cost of a finite union bound, to separate the dependency.
Towards this we construct nine auxilliary i.i.d. Bernoulli percolation models.
First, for each $(\mu,\nu) \in \{1,2,3\}^{2}$, we sample i.i.d. $\mathrm{Bernoulli}(\overline{p})$ random variables $\{\boldsymbol{\omega}_{k-\tell}(\mathsf{u},(\mu,\nu)) : \mathsf{u} \in \Lambda_{k-\tell}\}$, independent of everything else.
Then we define
\begin{align*}
    \omega_{k-\tell}^{(\mu,\nu)}(\mathsf{u}) = \begin{cases}
        \Easy(\tsigma_{k-\tell,k})(\mathsf{u}), & \text{if } \mathsf{u} \in \Lambda_{k-\tell}^{(\mu,\nu)}, \\
        \boldsymbol{\omega}_{k-\tell}(\mathsf{u},(\mu,\nu)), & \text{otherwise}.
    \end{cases}
\end{align*}
By \ref{assump:A2}, $\{\Easy(\tsigma_{k-\tell,k})(\mathsf{u})\}_{\mathsf{u}\in\Lambda_{k-\tell}^{(\mu,\nu)}}$ are i.i.d. Bernoulli random variables with parameter $\overline{p}$.
Moreover, by construction, the field $\omega_{k-\tell}^{(\mu,\nu)} = \{\omega_{k-\tell}^{(\mu,\nu)}(\mathsf{u}) : \mathsf{u} \in \Lambda_{k-\tell}\}$ defines an i.i.d. Bernoulli percolation model on $\Lambda_{k-\tell}$ (because for any two distinct boxes $\su, \sv\in \Lambda^{(\mu,\nu)}_{k-\tell}$, the enlarged boxes $\su^{\boxup},\sv^{\boxup}$ are disjoint and hence the respective easy crossing events occur independently).
Furthermore, since $\omega_{k-\tell}^{(\mu,\nu)}(\mathsf{u}) \ge 0$ for all $\mathsf{u} \in \Lambda_{k-\tell}$, we have (recalling the definition of $L_{\Easy(\tsigma_{k-\tell,k})}(a,b)$ from \eqref{e:lpp})
\begin{align*}
    L_{\Easy(\tsigma_{k-\tell,k})}(a,b) \le \sum_{\mu,\nu=1}^{3} L_{\omega_{k-\tell}^{(\mu,\nu)}}(a,b).
\end{align*}
Then by a union bound and Lemma \ref{blpp tail} for $n=K$, we obtain the following tail estimate for all $t > 9c_* \overline{p}^{1/2} K$:
\begin{align*}
    \bP\left(L_{\Easy(\tsigma_{k-\tell,k})}(a,b) > t\right) 
    &\le \sum_{\mu,\nu=1}^3 \bP\left(L_{\omega_{k-\tell}^{(\mu,\nu)}}(a,b) > \frac{t}{9} \right)\\
    &\le \sum_{\mu,\nu=1}^3
    \P_{\overline{p}}\left(
        S(K) > \frac{t}{9}
    \right)\\
    &\le 9 e^{-ct/9}.
\end{align*}
In particular, applying the above with $t=10c_* \overline{p}^{1/2}K$ yields
\begin{align}\label{e:exptail}
    \P\left(
        L_{\Easy(\tsigma_{k-\tell,k})}(a,b) \le 10 c_*\overline{p}^{1/2} K
    \right)
    &\ge 1 - 9e^{-\frac{10}{9}cc_* \overline{p}^{1/2} K} 
    \overset{\eqref{lower bound K}}{\ge} 1 - 9e^{-c' 2^{k-\tell}} 
    \ge 1 - e^{-c''2^{k}},
\end{align}
{where $c',c''>0$ depend on $\overline p$ and $\tell$ and hence only on the law of $(\tomega, \tsigma)$.}
We will now argue that \eqref{e:onescale} holds on the event in the LHS, and afterwards deduce \eqref{e:constant-frac-loss} from \eqref{e:onescale}.
Assume that $L_{\Easy(\tsigma_{k-\tell,k})}(a,b) \le 10c_*\overline{p}^{1/2} K$.
By \eqref{e:pbar-bound}, this implies
\begin{align*}
    L_{\Easy(\tsigma_{k-\tell,k})}(a,b) \le
    \frac{K}{18}.
\end{align*}
On the other hand, fix any directed path $\pi_k$ on $\Lambda_k$ connecting $a$ and $b$.
Recalling that $K:=\Size_{k-\tell}(\Proj_{k-\tell}(\pi_{k}))$, we see that by the pigeonhole principle, there exists $(\mu,\nu) \in \{1,2,3\}^{2}$ such that
\begin{align}\label{deletion on sublattice eq}
    \Size_{k-\tell}\left(\Proj_{k-\tell}(\pi_{k}) \cap \Lambda_{k-\tell}^{(\mu,\nu)}\right) 
    \ge
    \frac{K}{9}.
\end{align}
We deduce that
\begin{equation}\label{e:1234}
    \begin{split}
        \Bigl|\Bigl\{
        \mathsf{u} \in \Proj_{k-\tell}(\pi_{k}) \cap \Lambda_{k-\tell}^{(\mu,\nu)}:
        \Easy&(\tsigma_{k-\tell,k})(\mathsf{u})= 0
        \Bigr\}\Bigr|\\
        &\ge \Size_{k-\tell}\left(\Proj_{k-\tell}(\pi_{k}) \cap \Lambda_{k-\tell}^{(\mu,\nu)}\right) - L_{\Easy(\tsigma_{k-\tell,k})}(a,b)\\
        &\ge \frac{K}{9}- \frac{K}{18}\\
        &= \frac{1}{18} \Size_{k-\tell}(\Proj_{k-\tell}(\pi_{k})).
    \end{split}
\end{equation}
This establishes the first inequality in \eqref{e:onescale}.
The second inequality in \eqref{e:onescale} can be deduced from the following relation between $\dist(a,b)$ and $\Size_{k-\tell}(\Proj_{k-\tell}(\pi_{k}))$:
\begin{align*}
    \Size_{k-\tell}(\Proj_{k-\tell}(\pi_{k})) \ge  \dist(a,b) \cdot 2^{k-\tell} - 1.
\end{align*}     
Plugging this into the first inequality in \eqref{e:onescale} yields a further lower bound
\begin{align*}
    \frac{1}{18} \cdot \left(\dist(a,b) \cdot 2^{k-\tell} - 1\right) \ge \frac{1}{18} \cdot \frac{1}{2} \cdot \dist(a,b) \cdot 2^{k-\tell} = \frac{1}{36} \cdot 2^{-\tell} \cdot \dist(a,b) \cdot 2^{k},
\end{align*}
where the first inequality is because $\dist(a,b) \cdot 2^{k-\tell} \ge 2 \cdot 2^{7} >2$. This completes the proof of the second inequality in \eqref{e:onescale} by choosing $\tal = \frac{1}{36}\cdot 2^{-\tell}$.

To deduce \eqref{e:constant-frac-loss} from \eqref{e:onescale}, we just need to unpack the definition of the percolation model $\Easy(\tsigma_{k-\tell,k})$.
Denote by $\mathcal{R}=\mathcal{R}(\pi_k)$ the set of $2^{-k+\tell}\times 2^{-k+\tell}$ boxes appearing on the LHS of \eqref{e:onescale} and \eqref{e:1234}.
Note that for any $\su \in \mathcal{R}$, by directedness $\pi_k$ enters $\su$ from its left or bottom edge, and exits $\su^{\boxup}$ from its right or top edge, where we recall from \eqref{e:boxup} that $\su^{\boxup}$ is the box of size $2^{-k-\tell+1}\times 2^{-k-\tell+1}$ with the same bottom-left corner as $\su$.
Therefore, by definition of $\Easy$, for any 
$\mathsf{u}\in\mathcal{R}$, there exists $\mathsf{v} \in \pi_{k}\cap \su^{\boxup}$ such that $\tsigma_{k-\tell,k}(\mathsf{v}) = 0$; otherwise, $\pi_k$ would contain an easy crossing for $\tsigma_{k-\tell,k}$.
Additionally, since $\mathsf{u}_1^{\boxup}$ and $\mathsf{u}_2^{\boxup}$ are disjoint for all distinct $\su_1,\su_2\in\mathcal{R}$ (by definition of $\Lambda_{k-\tell}^{(\mu,\nu)}$),
it follows that $N_{\tsigma_{k}}(\pi_k) \le \Size_{k}(\pi_{k}) - |\mathcal{R}(\pi_k)|$,
where $N_{\tsigma_k}(\pi_k)$ was defined in \eqref{e:passage time}.
By Lemma \ref{length of path}, we have 
\begin{align*}
    \Size_{k}(\pi_{k}) \le 2^{k} \dist(a,b) + 3 \quad \text{and} \quad \Size_{k-\tell}(\Proj_{k-\tell}(\pi_{k})) \ge 2^{k-\tell} \dist(a,b) - 1.
\end{align*}
Therefore, we conclude from \eqref{e:onescale}
that with probability at least $1-e^{-c'2^{k}}$,
for any directed path $\pi_{k}\in \Pi_{k}$, we have
\begin{align}\label{passage time square eq}
    N_{\tsigma_{k}}(\pi_{k}) 
    &\le
    \Size_{k}(\pi_{k})  
    - \left|\left\{
        \mathsf{u} \in \Proj_{k-\tell}(\pi_{k}) \cap \Lambda_{k-\tell}^{(\mu,\nu)}: \Easy(\tsigma_{k-\tell,k})(\mathsf{u}) = 0
    \right\}\right| \nonumber \\ 
    \overset{\eqref{e:onescale}}&{\le} \Size_{k}(\pi_{k}) - \frac{1}{18} \Size_{k-\tell}(\Proj_{k-\tell}(\pi_{k})) \nonumber\\
    &\le 2^{k} \dist(a,b) + 3 - \frac{1}{18} \left(2^{k-\tell} \dist(a,b) - 1\right) \nonumber\\
    &\le \left(1 - \frac{1}{18} \cdot 2^{-\tell} + \frac{4}{2^{k}\dist(a,b)}\right) \cdot 2^{k} \dist(a,b)\nonumber\\
    &\le \left(1 - \frac{1}{36} \cdot 2^{-\tell}\right) \cdot 2^{k} \cdot \dist(a,b),
\end{align}
where in the last inequality we used that $k \ge \tell + 7$ and $\dist(a,b) \ge 2$ to bound 
\begin{align*}
    \frac{4}{2^{k}\dist(a,b)} 
    \le \frac{4}{2^{8}} \cdot 2^{-\tell} 
    \le \frac{1}{36} \cdot 2^{-\tell}.
\end{align*}
Maximizing over directed paths $\pi_k$ completes the proof of \eqref{e:constant-frac-loss}.
\end{proof}

\subsection{Multiscale analysis}\label{s:multiscale}
In this section we prove Theorem \ref{t:sublinear}.
As indicated in Section \ref{s:iop}, the idea is to iterate Proposition \ref{p:onescale} at many scales to show that with high probability, any path $\pi=\pi_n$ on $\Lambda_n$ must lose a constant fraction of its length per dyadic scale $k\in\lb 0,n\rb$.

Applying Proposition \ref{p:onescale} (viz. \eqref{e:constant-frac-loss}) to the original fractal percolation model $(\omega,\sigma)$, we get that with high probability, for every directed path $\pi=\pi_n$ in $\Lambda_n$, for $0\ll k \le n$ we can find many deleted  $2^{-k}\times 2^{-k}$ boxes on $\Proj_{k}(\pi)$. At this point a geometric issue arises. Namely, that the preceding fact by itself does \emph{not} imply a significant loss of length for the microscopic path $\pi$: in general $\pi$ can have a tiny overlap with a given $\sv\in\Proj_{k_1}(\pi)$ (for example, $\pi$ can enter $\sv\cap\Lambda_n$ through its bottom-right corner and then immediately exit). 
To address this issue, our analysis will center around the following $1$-dependent fractal percolation model. For any $k \in \Z_{\geq 1}$, recall that $\omega_k = \{\omega_{k}^{p}(\su)\}_{\su \in \Lambda_k}$ is a Bernoulli percolation model on $\Lambda_{k}$ with parameter $p$, i.e. a collection of i.i.d. Bernoulli random variables with parameter $p$. Given this, 
for any $k \in \Z_{\geq 1}$ and $\su \in \Lambda_{k}$, we define
\begin{align}\label{e:def-dep}
    \Dep(\omega_{k})(\su) := \begin{cases}
        0, & \text{if $\omega_{k}(\mathsf{v})=0$ for all $\mathsf{v}\in \Lambda_{k}$ with $\dist (\cen(\mathsf{u}),\cen(\mathsf{v})) \le 2^{-k}$}\\
        1, & \text{otherwise}
    \end{cases}.
\end{align}
In words, $\Dep(\omega_k)(\su)=0$ if and only if
$\su$ and the four boxes adjacent to it are all deleted.
For any integers $k>k'\ge 0$ we denote
\begin{equation}
    \begin{split}
        &\Dep(\sigma_{k',k}) := \bigcap_{k' < r \le k}\Dep(\omega_{r}),\\
        &\Dep(\sigma_{k}) := \Dep(\sigma_{0,k}).
    \end{split}
\end{equation}

It is {straightforward} to check that $(\{\Dep(\omega_k)\}_{k \ge 1}, \{\Dep(\sigma_{k',k})\}_{k>k'\ge 0})$ is a fractal percolation field as defined in Section~\ref{s:fractalfield}, {i.e. that $\Dep(\omega_k)$ is invariant under translation by $2^{-k}\Z^2$ and reflection across the line $y=x$ for every $k\ge 1$.}

In order to apply Proposition~\ref{p:onescale}, we need to verify the two assumptions \ref{assump:A1} and \ref{assump:A2} for $(\{\Dep(\omega_k)\}_{k \ge 1}, \{\Dep(\sigma_{k',k})\}_{k>k'\ge 0})$.
A key ingredient is the following lemma, which rules out directed percolation in the $1$-dependent fractal percolation model $\Dep(\sigma_k)$ in a non-quantitative manner, and in particular verifies assumption \ref{assump:A1} for any $p\in(0,1)$. This is a dependent version of Chayes's theorem \cite[Theorem A]{ChaAbsenceDirectedFractal1995}. 
\begin{lemma}[No directed easy crossings, $1$-dependent case]\label{l:easycross-dep}
    For any $p \in (0,1)$ and $\varepsilon \in (0,1)$, there exists $\ell_{0}(p,\varepsilon) \in \Zpos$ such that for all $k > \ell_{0}$,
    \begin{align*}
        \bP_{p}\left(\textup{there exists a directed path $\pi_{k}\subset \Dep(\sigma_{k})$ crossing $[0,2)\times [0,1)$ vertically}\right)
        \le \varepsilon.
\end{align*}
\end{lemma}
The proof of the above involves adapting Chayes's argument and is deferred to Appendix \ref{a:easycross}. 
That \ref{assump:A2} holds is an immediate consequence of the definitions and the details are provided shortly. 

Given this we can now apply Proposition~\ref{p:onescale} to $(\{\Dep(\omega_k)\}_{k \ge 1}, \{\Dep(\sigma_{k',k})\}_{k>k'\ge 0})$. This is recorded in the following proposition:

\begin{proposition}\label{p:onescale-dep}
    For every $p\in(0,1)$, there exist constants $c,\al>0$ and $\ell\in\Z_{\ge 1}$, all depending on $p$, such that the following holds.
    
    Fix $a, b \in \R^2$ with $a\preceq b$, and let $\sa,\sfb$ be the unique boxes in $\Lambda_{\ell}$ containing $a,b$ respectively.
    Assume that $\dist(\bl(\sa),\tr(\sfb)) \ge 2$.
    Then for all $k \ge \ell + 7$, it holds with probability at least $1-e^{-c2^{k}}$ that
   {for any directed path $\pi_k$ on $\Lambda_k$ connecting $a$ and $b$, there exists $(\mu,\nu)\in\{1,2,3\}^2$ (recall the well-separated lattice $\Lambda_{k-\ell}^{(\mu,\nu)}$  defined in \eqref{e:separated-lattice}) such that} 
        \begin{align}\label{e:onescale-dep}
            \left|\left\{
                \mathsf{u} \in \Proj_{k-\ell}(\pi_{k}) \cap \Lambda_{k-\ell}^{(\mu,\nu)}:\Easy(\Dep(\sigma_{k-\ell,k}))(\mathsf{u})=0
            \right\}\right| 
            &\ge 
            \al \cdot \dist(a,b)\cdot 2^{k}.
        \end{align}
\end{proposition}

\begin{proof}[Proof of Proposition \ref{p:onescale-dep}]
    By Lemma~\ref{l:easycross-dep}, the crossing estimate \ref{assump:A1} in Proposition~\ref{p:onescale} holds for the fractal percolation field $(\{\Dep(\omega_k)\}_{k \ge 1}, \{\Dep(\sigma_{k',k})\}_{k>k'\ge 0})$ for any $p\in(0,1)$. 
    It is also not hard to see that \ref{assump:A2} holds for this model.
    Indeed, for any box $\su\in\Lambda_{k'}$, by \eqref{e:def-dep} the random variable $\Dep(\sigma_{k',k})(\su)$ depends only on $\omega_r(\sv)$ for $k<r\le k'$ and boxes $\sv\in\Lambda_r$ whose centers lie within distance $2^{-k'}$ of $\cen(\su)$.
    Recalling the definition of $\Easy$ from \eqref{e:def-easy}, we deduce that $\Easy(\Dep(\sigma_{k',k}))(\su)$ depends only  on $\omega_r(\sv)$ for $k<r\le k'$ and boxes $\sv\in\Lambda_r$ whose centers lie within distance $2\cdot 2^{-k'}$ of $\cen(\su)$.
    Therefore, since the boxes in the separated lattice $\Lambda_{k'}^{(\mu,\nu)}$ have centers separated by distance at least $3\cdot 2^{-k'}$ (recall \eqref{e:separated-lattice}), it follows that $\Easy(\Dep(\sigma_{k',k}))$ restricted to $\Lambda^{(\mu,\nu)}_{k'}$ is an i.i.d. field.
    We can therefore apply Proposition \ref{p:onescale}, which completes the proof.
\end{proof}

We are ready to prove Theorem \ref{t:sublinear}.
The argument, previously outlined in Section \ref{s:iop}, consists of an inductive multi-scale analysis.
For the reader's convenience, we now provide a slightly more detailed proof sketch. 
We begin by fixing any directed path $\pi=\pi_n$ and analyzing how its length decreases due to the percolation noise from the first $C\log n$ many scales (the reason for choosing $C\log n$ will be clarified momentarily).
An application of Proposition \ref{p:onescale-dep} allows us to  remove a constant fraction of $2^{-C\log n}$-boxes along $\pi$. The definition of the dependent percolation model considered further ensures that $\pi$  spends a reasonable amount of time inside each deleted $2^{-C\log n}$-box, permitting us to in fact  remove a constant fraction (say $1\%$) of $2^{-n}$-boxes along $\pi$.

Next, we consider the remaining segments of $\pi$, and reveal the percolation noise at the subsequent $C\log n$ scales.
Applying Proposition \ref{p:onescale-dep} to these segments, the total retained length decreases by a further $1\%$. This procedure is iterated through successive scales until we reach the finest few scales. Ultimately, the total length of the retained portion of $\pi$ is at most $0.99^{n/C\log n}$ times the original length, which is precisely the statement of Theorem \ref{t:sublinear}.
\begin{remark}\label{gap45}
Note that the above strategy demands  a uniform estimate holding with high probability for all the path segments whose locations produce a large entropy factor. Skipping every $C\log n$ many scales allows one to counter it. At scale $k_{i+1}$,  the failure probability in Proposition \ref{p:onescale-dep}  is $\exp(-O(2^{k_{i+1}-k_i}))$, which for $k_{i+1}-k_i = C\log n$ can be made super-polynomial in $2^{-n}$ by choosing $C$ large, permitting a union bound over all possible endpoints of such paths. 
However, we expect that with a bit more work a more efficient implementation can be carried out only requiring that the above bound holds on  average at every scale and at every location. This will allow us to replace the $\log n$ scale separation by a large constant, thereby removing the lossy $\log n$ factor in Theorem \ref{t:sublinear} (an argument of this flavor indeed appears in \cite{DZLiouvilleFirstPassage2019}). However, to prevent things from getting overly technical we do not pursue this here.
\end{remark}

\begin{proof}[Proof of Theorem \ref{t:sublinear}]

Fix $p\in(0,1)$ and let $\ell=\ell(p)\ge 1$ be the fixed integer in Proposition \ref{p:onescale-dep}.
We consider a sparse subsequence $\{k_i\}_{i\in\lb 0, I\rb} \subset \lb 0, n\rb$ of dyadic scales, where
\begin{align}\label{not7684}
    k_{i} := i\cdot \Bigl(
        10\lceil\log_2 n\rceil +\ell
        \Bigr)
        \quad\text{for } i\in \lb 0, I\rb,
        \qquad\qquad I:= \left\lfloor \frac{n}{10\ceil{\log_2 n} + \ell}\right\rfloor.
\end{align}
We assume from now on that $n$ is large enough so that $I \ge 2$ and $k_1 \ge \ell + 7$.
For any $i\in\lb 1,I\rb$ and any $\mathsf{x},\mathsf{y}\in \Lambda_{k_i}\cap [0,1]^2$ satisfying $\cen(\sx)\preceq \cen(\sy)$ and $\dist(\bl(\mathsf{x}),\tr(\mathsf{y}))\ge 2^{-k_{i-1}+1}$,
we denote by $\mathcal{G}(i;\mathsf{x},\mathsf{y})$  the following event:
    for any directed path $\pi_{k_i}$ in $\Lambda_{k_i}$ connecting  $\cen(\mathsf{x})$ and $\cen(\mathsf{y})$, there exists 
    $(\mu,\nu)\in \{1,2,3\}^{2}$ 
    such that 
    \begin{align}\label{one event}
         \left|\left\{\mathsf{u} \in \Proj_{k_i-\ell}(\pi_{k_i}) \cap \Lambda_{k_i-\ell}^{(\mu,\nu)}:\Easy(\Dep(\sigma_{k_{i} - \ell,k_i}))(\mathsf{u})=0\right\}\right| 
         \ge 
        \al\cdot 
        2^{k_i} \cdot \dist(\cen(\sx),\cen(\sy)).
    \end{align}
Note that for $i = 1$, the only admissible choices for $\mathsf{x}$ and $\mathsf{y}$ are the boxes at the bottom-left and top-right corners of $\Lambda_{k_1}\cap [0,1]^{2}$ and $\mathcal{G}(1;\mathsf{x},\mathsf{y})$ is simply the event in \eqref{e:onescale-dep} with $a = (0,0)$, $b = (1-2^{-k_1},1-2^{-k_1})$, and $k = k_1$.
Therefore, by Proposition \ref{p:onescale-dep}, we have
\begin{align}\label{one event2 eq}
    \bP(\mathcal{G}(1;\mathsf{x},\mathsf{y})^{c})
    \le 
    e^{-c 2^{k_1}}
    = 
    e^{-c2^{10\ceil{\log_2 n} + \ell}}
    \le
    e^{-cn^{10}}.
\end{align}
We next establish a matching probability bound for $\mathcal{G}(i;\sx,\sy)$ for $i\in\lb 2,I\rb$.
Fix $i\in\lb 2,I\rb$ and $\sx,\sy \in \Lambda_{k_i}$ with $\cen(\sx) \preceq \cen(\sy)$ and $\dist(\bl(\sx),\tr(\sy)) \ge  2^{-k_{i-1} + 1}$.

In fact the event $\mathcal{G}(i;\mathsf{x},\mathsf{y})$ is a rescaled version of the event in \eqref{e:onescale-dep} for $a = \cen(\sx)$, $b = \cen(\sy)$, and $k := k_i-k_{i-1}$.
Indeed, if we rescale the lattice $\Lambda_{k_{i-1}}$ and regard it as $\Lambda_{0}$, then the rescaled distance between any pair of points $(u,v)$ is $\dist(u,v) \cdot 2^{k_{i-1}}$.
Then our choice of $(a,b) = (\cen(\sx), \cen(\sy))$ satisfies the distance condition in Proposition \ref{p:onescale-dep} since if we denote by $\tilde{\sx}, \tilde{\sy}$ the unique boxes in $\Lambda_{k_i - \ell}$ such that $\cen(\sx) \in \tilde{\sx}$ and $\cen(\sy) \in \tilde{\sy}$, then
\[
    \dist(\bl(\tilde{\sx}),\tr(\tilde{\sy})) \cdot 2^{k_{i-1}} \ge \dist(\bl(\sx),\tr(\sy)) \cdot 2^{k_{i-1}}\ge 2^{-k_{i-1} + 1}\cdot 2^{k_{i-1}} = 2.
\]
Moreover, under this rescaling of the lattice $\Lambda_{k_{i-1}}$, the configuration $\sigma_{k_i - \ell,k_i}$ plays the role of $\sigma_{k_i - k_{i-1} - \ell, k_i - k_{i-1}}$ in Proposition \ref{p:onescale-dep}.
Therefore, we can rewrite \eqref{one event} as 
\begin{align*}
\left|\left\{\mathsf{u} \in \Proj_{k_i-\ell}(\pi_{k_i}) \cap \Lambda_{k_i-\ell}^{(\mu,\nu)}:\Easy(\Dep(\sigma_{k_{i} - \ell,k_i}))(\mathsf{u})=0\right\}\right| 
        \ge
        \al\cdot 
        \left(2^{k_{i-1}} \cdot \dist(\cen(\sx),\cen(\sy))\right)
        \cdot 2^{k_i - k_{i-1}},
\end{align*}
which is precisely the form of \eqref{e:onescale-dep} after rescaling the lattice and the percolation configuration; see Figure~\ref{f:rescaling} for an illustration.
In particular, by Proposition \ref{p:onescale-dep} we have the following probability bound for $\mathcal{G}(i;\mathsf{x},\mathsf{y})$:
\begin{align}\label{one event eq}
    \bP(\mathcal{G}(i;\mathsf{x},\mathsf{y})^{c}) 
    \le
    e^{-c 2^{k_i-k_{i-1}}}
    = e^{-c 2^{10\lceil\log_2 n\rceil +\ell}}
    \le e^{-cn^{10}}.
\end{align}

\begin{figure}[h]
    \includegraphics[width=0.45\textwidth]{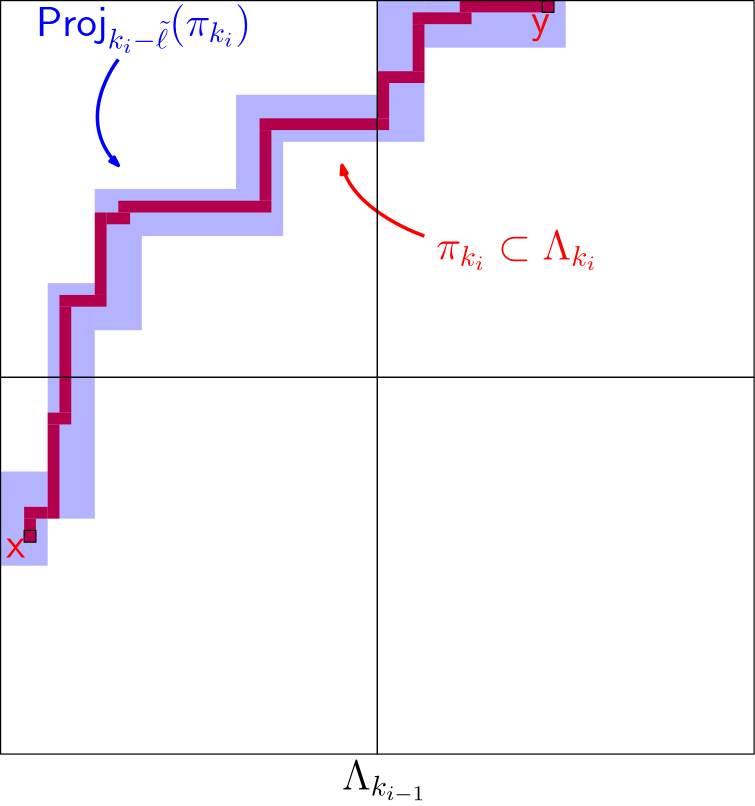}
    \caption{Rescale the lattice and treat $\Lambda_{k_{i-1}}$ as $\Lambda_0$. The event $\mathcal{G}(i;\sx,\sy)$ is a rescaled version of the event in \eqref{e:onescale-dep}: let $a$ and $b$ be the centers of $\sx$ and $\sy$ respectively (the red boxes), then on the event $\mathcal{G}(i;\sx,\sy)$, for any directed path $\pi_{k_i} \subset \Lambda_{k_i}$ connecting $a$ and $b$ (the red path), the projection $\Proj_{k_i-\tell}(\pi_{k_i})$ (the blue path) contains at least a constant fraction of deleted $2^{-k_i+\tell}\times 2^{-k_i+\tell}$ boxes in $\Easy(\Dep(\sigma_{k_{i} - \tell,k_i}))$.}
    \label{f:rescaling}
\end{figure}

Note that the probability bounds \eqref{one event2 eq}, \eqref{one event eq} are super-polynomially small in $2^{-n}$, due to our choice of dyadic scales separated by gaps of size $k_i-k_{i-1}\sim 10\log_2n$.
We now exploit this by taking a union bound over all pairs of boxes $\sx,\sy$ across all scales $k_i$.
Denoting 
\[
    \sC_{k}:=\{(\sx,\sy): \sx,\sy \in \Lambda_{k_i}\cap [0,1]^2,\; \cen(\sx) \preceq \cen(\sy),\; \dist(\bl(\sx),\tr(\sy))\ge 2^{-k_{i-1}+1}\}
\] 
and defining the event 
\begin{align}\label{e:event-G}
    \mathcal{G} := 
    \bigcap_{i=1}^I
    \; \bigcap_{(\sx,\sy)\in \sC_{k_i}} 
    \mathcal{G}(i;\mathsf{x},\mathsf{y}),
\end{align}
we obtain by a union bound
\begin{align*}
    \bP(\mathcal{G}^{c}) &\le \sum_{i=1}^{I}\sum_{(\sx,\sy)\in \sC_{k_i}}\bP\left(\mathcal{G}(i;\mathsf{x},\mathsf{y})^{c}\right) \\
    \overset{\eqref{one event2 eq}, \eqref{one event eq}}&{\le} \sum_{i=1}^{I}\sum_{(\sx,\sy)\in \sC_{k_i}} e^{-cn^{10}}\\
    &\le \sum_{i=1}^{I}2^{4k_i}\cdot e^{-cn^{10}}\\
    &\le C 2^{4n}\cdot e^{-cn^{10}}\\
    &\le e^{-c' n^{10}},
\end{align*}
where the third inequality follows from the crude bound
\begin{align*}
    |\sC_{k_i}|
    \le \left(\Size_{k_i}(\Lambda_{k_i} \cap[0,1]^2)\right)^{2} 
    = 2^{4k_i}.
\end{align*}

Next we will show that on the event $\mathcal{G}$,
for any directed path $\pi\in\Pi_n((0,0),(1-2^{-n},1-2^{-n}))$, its passage time satisfies $N_{n}(\pi) \le 2^{n}\cdot e^{-cn/\log n}$.
This combined with the above estimate for $\P(\cG^c)$ immediately implies Theorem \ref{t:sublinear}. 
So we assume from now on that $\mathcal{G}$ occurs, and we fix a directed path $\pi\in\Pi_n$.
We will frequently refer to $\pi$ as the \emph{microscopic path}, to emphasize that it lives on the finest lattice $\Lambda_n$ under consideration.

We now inductively bound the passage time of $\pi$.
We separately analyze the base case (Step 1) and the inductive step (Step 2).

\bigskip
\noindent\textbf{Step 1 (base case).}
We begin by analyzing how $\Dep(\sigma_{k_1-\ell, k_1})$ affects the passage time, aiming to show that the length of $\pi$ decreases by a constant factor due to the deletion of boxes at scale $k_1$.
By our assumption that $\mathcal{G}(1;\sa,\sfb)$ occurs with $\cen(\sa)=(0,0)$ and $\cen(\sfb)=(1-2^{-k_1},1-2^{-k_1})$,
there exists $(\mu,\nu)\in\{1,2,3\}^2$ such that \eqref{e:onescale-dep} holds for $\Proj_{k_1}(\pi)$ on $\Lambda^{(\mu,\nu)}_{k_1-\ell}$.
To ease notation, we say that a box $\su\in\Lambda_{k_1-\ell}$ is \Bad if $\Easy(\Dep(\sigma_{k_1-\ell, k_1}))(\su)=0$ (so the notion of {\Bad}ness implicitly depends on $k_1$).
We denote by $\mathsf{B}_{k_1-\ell}(\pi)$ the collection of \Bad boxes that $\pi$ passes through. 

Now we recall Remark~\ref{r:easy}: within each \Bad box $\mathsf{u}\in \mathsf{B}_{k_1-\ell}(\pi)$ there exists a $2^{-k_1}\times 2^{-k_1}$ box $\mathsf{v}\in \mathsf{u}^{\boxup} \cap \Proj_{k_1}(\pi)$ with $\Dep(\sigma_{k_1-\ell, k_1})(\mathsf{v})=0$, where $\su^{\boxup}$ is defined in \eqref{e:boxup}. Indeed, otherwise the projection $\Proj_{k_1}(\pi)$ would form an easy crossing inside $\su^{\boxup}$, contradicting the {\Bad}ness of $\su$; see Figure~\ref{f:easycrossing} for two possible cases of how $\Proj_{k_1}(\pi)$ can traverse $\su^{\boxup}$ and Figure~\ref{f:major} for the identification of $\sv$ mentioned above.

Next, using directedness and the definition of $\Dep$, we will find a deleted box $\sm\in\su^{\boxup}\cap\Proj_{k_1}(\pi)$ inside which the microscopic path $\pi$ spends a long time.
We say that a box $\sm \in \Proj_{k}(\pi)$ is \Major
(at level $k$, and with respect to $\pi$)
if
\begin{align*}
    \Size_{n}(\pi \cap \sm) \ge 2^{n-k-1}.
\end{align*}
Note that by directedness we have $\Size_{n}(\pi \cap \sm) \le 2^{n-k+1}-1$ for any $\sm\in\Lambda_k$, so a box being \Major just means that $\pi$ occupies that box for a constant fraction of the maximum time possible.
The key deterministic result ensuring the existence of deleted \Major boxes is the following lemma.

\begin{lemma}\label{claim1}
    If $\su\in\sB_{k_1-\ell}(\pi)$ is a \Bad box, then $\su^{\boxup}$ contains a deleted \Major box at level $k_1$.
\end{lemma}
\begin{proof}
Let $\sv\in\su^{\boxup}\cap\Proj_{k_1}(\pi)$ be the box with $\Dep(\sigma_{k_1-\ell, k_1})(\sv)=0$ whose existence is guaranteed by the {\Bad}ness of $\su$.
Then by definition of $\Dep$, the box $\sv$ and all the boxes adjacent to it are deleted.
In particular, there is a deleted box $\sv'\in\su^{\boxup}\cap\Proj_{k_1}(\pi)$ such that $\pi$ visits $\sv,\sv'$ consecutively.
Then at least one of $\sv,\sv'$ must be \Major, since by the directedness of $\pi$, for any two consecutive boxes $\mathsf{v}$ and $\mathsf{v}'$ on $\Proj_{k_1}(\pi)$, we have the inequality
\begin{align*}
    \Size_{n}(\mathsf{v} \cap \pi) + \Size_{n}(\mathsf{v}' \cap \pi) \ge 2^{n - k_1}.
\end{align*}
This finishes the proof.
\end{proof}

Using Lemma~\ref{claim1}, for each \Bad box $\su\in\Lambda^{(\mu,\nu)}_{k_1-\ell}$ we select a deleted \Major box $\sm\in\su^{\boxup}\cap\Proj_{k_1}(\pi)$; see Figure~\ref{f:major} for an illustration.
To avoid ambiguity we select the coordinate-wise smallest such box (but any other deterministic rule for breaking ties would also work).
We denote by $\Remove_{k_1}(\pi)$ the collection of selected deleted \Major boxes.

\begin{figure}[h]
\includegraphics[width=0.5\textwidth]{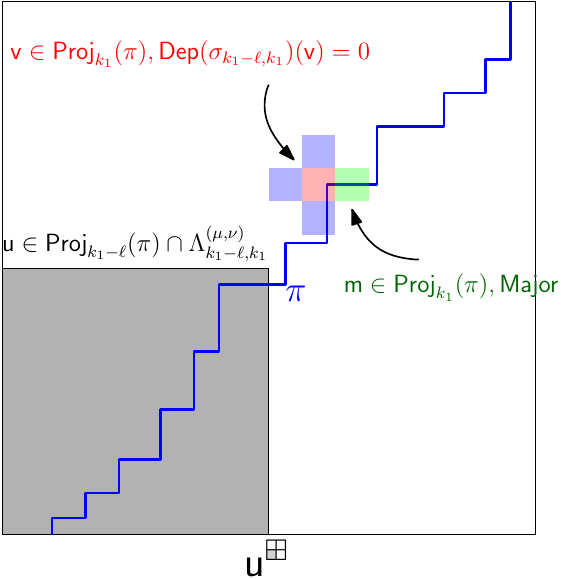}
\caption{Given that the gray box $\su$ is a \Bad box, $\pi$ (the blue path) must pass through a box $\sv \in \su^{\boxup}\cap\Lambda_{k_1} $ (the red box) satisfying $\Dep(\sigma_{k_1-\ell,k_1})(\sv) = 0$.
By the definition of $\Dep$, the box $\sv$ and its four neighbors are all deleted in $\sigma_{k_1-\ell,k_1}$.
Out of these five deleted boxes, there exists one within $\su^{\boxup}$ that is \Major (the green box $\sm$) as established in Claim~\ref{claim1}; we select this \Major box for $\Remove_{k_1}(\pi)$.}
\label{f:major}
\end{figure}

Since the boxes in $\Remove_{k_1}(\pi)$ are deleted, from now on we can just focus on the behavior of $\pi$ restricted to the complement of $\Remove_{k_1}(\pi)$. Before moving to Step 2, we record some observations about the removal set $\Remove_{k_1}(\pi)$ and its complement that will be crucial for iterating the above procedure. Let us at this point remark that while there are other boxes that are also deleted, we only include a carefully chosen subset of them in $\Remove_{k_1}(\pi)$ to maintain better geometric properties of the complement of the latter.  
\begin{lemma}\label{l:segments}
    The connected components of $\pi\setminus \bigcup_{\sv\in\Remove_{k_1}(\pi)}\sv$ are themselves directed paths in $\Lambda_n$.
    We refer to these paths as \emph{segments} of $\pi$, and denote them by
    \begin{align*}
        \Seg_{k_1}^{1},\Seg_{k_1}^{2},\dots, \Seg_{k_1}^{m_1}.
    \end{align*}
\end{lemma}
\begin{proof}
    This is an immediate consequence of the directedness of $\pi$ and the non-backtracking property Lemma~\ref{nonbacktracking}.
\end{proof}

The next lemma is a uniform lower bound for the length of segments $\{\Seg_{k_1}^{i}\}_{i=1}^{m_1}$.

\begin{lemma}\label{l:separation}
    For any $1 \leq i \leq m_1$, 
    \begin{center}
        \textup{the $\dist$-distance between the bottom-left corner of the first box in $\Proj_{k_2}(\Seg_{k_1}^{i})$ and the top-right corner of the last box in $\Proj_{k_2}(\Seg_{k_1}^{i})$ is at least $2^{-(k_1-\ell)}$.}
    \end{center}
\end{lemma}
\begin{proof}
    Note that any two distinct removed boxes in $\Remove_{k_1}(\pi)$ must lie in $\su^{\boxup}$ for distinct \Bad boxes $\su\in\Lambda_{k_1-\ell}^{(\mu,\nu)}$, since we selected just one \Major box per \Bad box.
    By the well-separatedness of $\Lambda_{k_1-\ell}^{(\mu,\nu)}$ (recall \eqref{e:separated-lattice}), the $\dist$-distance between such $\su^{\boxup}$ is at least $2^{-(k_1-\ell)}$. Meanwhile, by Lemma \ref{nonbacktracking}, for each segment $\Seg_{k_1}^{i}$, the first and last $2^{-k_2}\times 2^{-k_2}$ boxes of its projection $\Proj_{k_2}(\Seg_{k_1}^{i})$ must be adjacent to \emph{distinct} removed boxes in $\Remove_{k_1}(\pi)$.
    These two facts together imply the desired distance estimate.
\end{proof}

This distance estimate will allow us to later apply \eqref{one event} with $i=2$ (since $2^{-(k_1-\ell)} > 2^{-k_2+1}$) for $\Proj_{k_2}(\Seg_{k_1}^{i})$, $1\leq i\leq m_1$, to show that on the event $\mathcal{G}$, a constant fraction of the $2^{-k_2+\ell}\times 2^{-k_2+\ell}$ boxes $\su$ along each segment satisfy $\Dep(\sigma_{k_2-\ell,k_2})(\su)=0$; See Figure~\ref{f:seg}.

\begin{figure}[h]
        \includegraphics[width=0.5\textwidth]{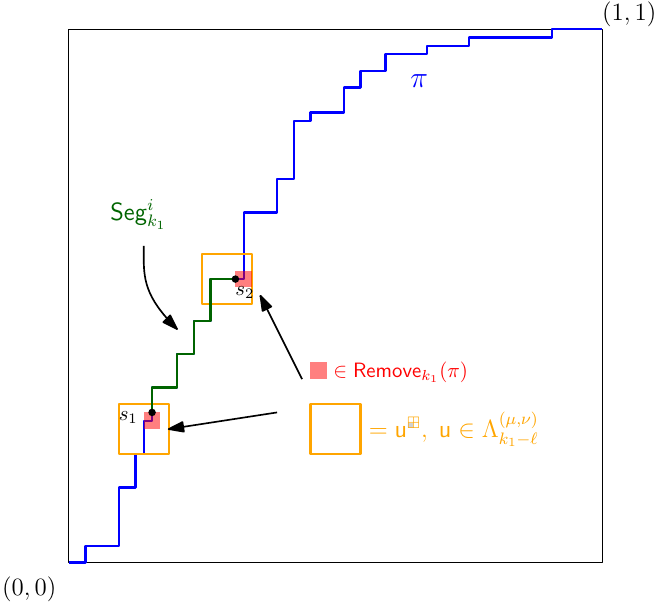}
        \caption{Illustrations of removing deleted boxes at one scale: The orange boxes represent $\su^{\boxup}$ for the $\Bad$ boxes $\su$ found in $\Lambda_{k_1-\ell}^{(\mu,\nu)}$. Following the removal procedure in Step 1, within each chosen $\su^{\boxup}$, we select a $\Major$ box (colored red) in $\Lambda_{k_1}$ that is deleted in $\sigma_{k_1-\ell,k_1}$. Doing so creates several segments of $\pi$ (such as the green segment). Due to the well-separatedness of the $\Bad$ boxes in $\Lambda_{k_1-\ell}^{(\mu,\nu)}$, for each segment, the $\dist$-distance between the bottom-left corner of the first box and the top-right corner of the last box in $\Proj_{k_2}(\Seg_{k_1}^{i})$ is at least $2^{-(k_1-\ell)}$ for each $1\leq i\leq m_1$.}
        \label{f:seg}
    \end{figure}

The last observation needed before proceeding to Step 2 is that the total length of the segments is at most a constant fraction of the length of $\pi$ 
(recall that the length of $\pi$ is $\Size_n(\pi) = 2^{n+1} - 1$).
\begin{lemma}\label{l:onescaleloss}
    For $\alpha>0$ defined in Proposition~\ref{p:onescale-dep}, we have
    \begin{align}\label{e:remaining-length}
    \Size_{n}\left(
        \bigcup_{i=1}^{m_{1}} \Seg_{k_1}^{i}
    \right) \leq \left(1 - \frac{\al}{8}\right) \cdot  2^{n + 1}.
    \end{align}
\end{lemma}
\begin{proof}
    By the construction of $\Remove_{k_1}(\pi)$, we know that any box in $\Remove_{k_1}(\pi)$ corresponds to a unique and distinct $\Bad$ box found on the LHS of \eqref{e:onescale-dep} for $\Proj_{k_1}(\pi)$. {Hence, by \eqref{e:onescale-dep}, and using that in this case $a = (0,0), b = (1-2^{-k_1},1-2^{-k_1})$,} we have
    \begin{align*}
        |\Remove_{k_1}(\pi)| 
        \ge \al \cdot 2^{k_1} \cdot (2-2^{-k_1+1}) 
        \ge \al \cdot 2^{k_1}.
    \end{align*}
    Also, all the removed boxes $\sm\in\msf{Remove}_{k_1}(\pi)$ are \Major, so by definition they satisfy 
    \begin{align*}
        \Size_{n}( \mathsf{m} \cap \pi )\ge 2^{n-k_1-1}.
    \end{align*}
    Therefore, after deleting the boxes in $\Remove_{k_1}(\pi)$ from $\pi$, the remaining length is upper bounded by 
\begin{align*}
    \Size_{n}\left(
        \bigcup_{i=1}^{m_{1}} \Seg_{k_1}^{i}
    \right)
    &= \Size_{n}(\pi) -  \Size_{n} \left(\bigcup_{\mathsf{m}\in \Remove_{k_1}(\pi)} \mathsf{m} \cap \pi\right) \nonumber\\
    &\le 2^{n+1} - |\Remove_{k_1}(\pi)| \cdot 2^{n-k_1-1}\nonumber\\
    &\le 2^{n+1} - \al \cdot 2^{k_1}\cdot 2^{n-k_1-1}\nonumber\\
    &= \left(1 - \frac{\al}{4}\right)\cdot 2^{n+1}\nonumber\\
    &\le \left(1 - \frac{\al}{8}\right)\cdot 2^{n+1},
\end{align*}
where in the second line we used that the boxes in $\Remove_{k_1}(\pi)$ are disjoint,
and where in the last line we worsened the bound to simplify notation in the sequel. This finishes the proof.
\end{proof}
This lemma completes Step 1, and we move to Step 2.

\bigskip
\noindent\textbf{Step 2 (inductive step).} 
As shown in Step 1, the removal of deleted boxes at scale $k_1$ partitions the original path $\pi$ into several segments. The next step is to analyze how revealing the percolation environment at further scales affects the passage time of each segment. This analysis can be summarized as the following general proposition.

\begin{proposition}\label{l:onescaleloss2}
    For any $i \in\lb 1, I\rb$, where $I$ was defined in \eqref{not7684}, $(\sx, \sy) \in \sC_{k_{i-1}}$ and $\Seg \in \Pi_{n}(\bl(\sx), \tr(\sy))$, we have that on the event $\mathcal{G}$, we can remove a collection of deleted \Major boxes at level $k_i$ (recalling that $k_i = i\cdot(10\lceil \log_2 n\rceil + \ell)$), denoted by $\Remove_{k_i}(\Seg)$, from $\Proj_{k_i}(\Seg)$ such that the remaining length of $\Seg$ is at most a constant fraction of its original length, i.e.
    for $\alpha>0$ defined in Proposition~\ref{p:onescale-dep}, we have
    \begin{align}\label{one segment eq}
    \Size_{n}\left(\Seg \setminus \Remove_{k_i}(\Seg)\right)
    \le\left(1 -\frac{\al}{8}\right)\cdot \Size_{n}(\Seg).
\end{align}
Furthermore, removing the boxes in $\Remove_{k_i}(\Seg)$ creates further subsegments of $\Seg$. For any of those subsegments $\mathsf{SubSeg}$, the $\dist$-distance between the bottom-left corner of the first box in $\Proj_{k_{i+1}}(\mathsf{SubSeg})$ and the top-right corner of the last box in $\Proj_{k_{i+1}}(\mathsf{SubSeg})$ is at least $2^{-(k_i-\ell)}$.
\end{proposition}

\begin{remark}
The key point of this bound is that the constant $\al/8$ is independent of the scale $k_i$.
{The proof idea is  the same as that of Lemma~\ref{l:onescaleloss}}: 
1) the number of boxes removed at scale $k_2$ is at least a scale-independent constant fraction of $\Size_{k_2}(\Proj_{k_2}(\Seg))$, regardless of the geometry of $\Seg$; 
2) the boxes we remove are \Major, which means that removing the parts of $\Seg$ inside these boxes also effectively reduces the microscopic length of $\Seg$ (i.e. its $\Size_n$).
\end{remark}

\begin{proof}[Proof of Proposition \ref{l:onescaleloss2}]
The case $i = 1$ has been established in Step 1, so we focus on the case $i \in \lb 2, I\rb$ and how $\sigma_{k_i-\ell, k_i}$ affects $\Seg$.

We apply the same procedure as in Step 1 to $\Seg$, yielding a collection $\Remove_{k_i}(\Seg)$ of deleted \Major $2^{-k_i}\times 2^{-k_i}$ boxes (at level $k_i$, with respect to $\Seg$) in $\Proj_{k_i}(\Seg)$. 
Then the exactly same argument in Lemma~\ref{l:separation} shows that for any subsegment $\mathsf{SubSeg}$ of $\Seg$ created by removing the boxes in $\Remove_{k_i}(\Seg)$, the $\dist$-distance between the bottom-left corner of the first box in $\Proj_{k_{i+1}}(\mathsf{SubSeg})$ and the top-right corner of the last box in $\Proj_{k_{i+1}}(\mathsf{SubSeg})$ is at least $2^{-(k_i-\ell)}$. Therefore, it remains to verify the upper bound \eqref{one segment eq}.

Let $\msf{s}_1$ and $\msf{s}_2$ be the first and last box in $\Proj_{k_i}(\Seg)$ respectively.
By our choice of $\Seg$, 
\[
    \dist(\bl(\msf{s}_1),\tr(\msf{s}_2)) \ge 2^{-k_{i-1} + 1}.
\]
Hence, we know that on the event $\mathcal{G}$, the estimate \eqref{one event} holds for $\Proj_{k_i}(\Seg)$.
Thus, applying \eqref{one event} to $\Proj_{k_i}(\Seg)$ yields
\begin{align}\label{k_2 eq1}
    |\Remove_{k_i}(\Seg)| 
    \overset{\eqref{one event}}&{\ge} \al\cdot \dist(\cen(\mathsf{s}_1), \cen(\mathsf{s}_2)) \cdot 2^{k_i}\nonumber\\
    &= \al \cdot \left(\Size_{k_i}(\Proj_{k_i}(\Seg)) - 1\right)\nonumber\\
    &\ge \frac{\al}{2}\cdot \Size_{k_i}(\Proj_{k_i}(\Seg)),
\end{align}
where the equality is by Lemma \ref{length of path}. For any \Major box $\sm \in \Remove_{k_i}(\Seg)$, by definition  we have 
\[
    \Size_n(\sm \cap \Seg) \ge 2^{n-k_i-1}.
\]
Then arguing as in the proof of Lemma~\ref{l:onescaleloss}, since the removed boxes are disjoint and \Major, we obtain that 
\begin{align}\label{k_2 eq2}
    \Size_{n}\left(\Seg \setminus \Remove_{k_i}(\Seg)\right) 
    &=\Size_{n}(\Seg) - \Size_{n} \left(\bigcup_{\mathsf{m}\in \Remove_{k_i}(\Seg)} \mathsf{m} \cap \Seg\right) \nonumber\\
    &\le \Size_{n}(\Seg) - |\Remove_{k_i}(\Seg)| \cdot 2^{n-k_i-1}\nonumber\\
    \overset{\eqref{k_2 eq1}}&{\le}
    \Size_{n}(\Seg) - \frac{\al}{2} \cdot \Size_{k_i}(\Proj_{k_i}(\Seg))\cdot 2^{n-k_i-1}.
\end{align}
Towards upper bounding \eqref{k_2 eq2}, we need to lower bound $\Size_{k_i}(\Proj_{k_i}(\Seg))$. 
For any $2^{-k_i} \times 2^{-k_i}$ box $\sfb \in \Proj_{k_i}(\Seg)$,
the intersection $\Seg \cap \sfb$ is a directed path crossing $\sfb$ (in $\Lambda_n$),
so the $\Size_n$ of this intersection is at most that of a directed path in $\Lambda_n$ going from the bottom-left corner to the top-right corner of $\sfb$, which is $2^{n-k_i+1}-1$.
That is, for any $\sfb \in \Size_{k_i}(\Proj_{k_i}(\Seg))$, we have
\begin{align}\label{e: intersection size}
    \Size_n(\Seg\cap\mathsf{b})\leq 2^{n-k_i+1} - 1 < 2^{n-k_i+1}.
\end{align}
This implies the following lower bound for $\Size_{k_i}(\Proj_{k_i}(\Seg))$ in terms of $\Size_n(\Seg)$:
\begin{align}\label{k_2 eq3}
    \Size_{k_i}(\Proj_{k_i}(\Seg)) \ge \Size_n(\Seg) \cdot 2^{-n+k_i-1}.
\end{align}
Indeed, by \eqref{e: intersection size} we have
\begin{align*}
    \Size_n(\Seg)
    &=\sum_{\mathsf{b}\in\Proj_{k_i}(\Seg)}\Size_n(\Seg\cap\mathsf{b})\\
    &\le \Size_{k_i}(\Proj_{k_i}(\Seg))\cdot \max_{\mathsf{b}\in\Proj_{k_i}(\Seg)}\Size_n(\Seg\cap\mathsf{b})\\
    \overset{\eqref{e: intersection size}}&{\le} \Size_{k_i}(\Proj_{k_i}(\Seg))\cdot 2^{n-k_i+1}.
\end{align*}
Now plugging \eqref{k_2 eq3} into \eqref{k_2 eq2}, we get the desired upper bound for the remaining length of $\Seg$:
\begin{align}\label{one segment}
    \Size_{n}\left(\Seg \setminus \Remove_{k_2}(\Seg)\right)
    &\le
    \Size_n(\Seg) - \frac{\al}{8}\cdot \Size_n(\Seg)\nonumber\\
    &=\left(1 -\frac{\al}{8}\right)\cdot \Size_{n}(\Seg).
\end{align}
\end{proof}

With this proposition in hand, we can now see how to iterate the procedure. We first apply Proposition \ref{l:onescaleloss2} to each segment $\Seg_{k_1}^{1},\dots,\Seg_{k_1}^{m_1}$ created in Step 1, thereby obtaining collections $\Remove_{k_2}(\Seg_{k_1}^{1}), \dots, \Remove_{k_2}(\Seg_{k_1}^{m_1})$ of deleted \Major boxes. By \eqref{one segment eq}, for each $1\leq i \leq m_1$, we have
\begin{align}\label{one segment eq2}
    \Size_{n}\left(\Seg_{k_1}^{i} \setminus \Remove_{k_2}(\Seg_{k_1}^{i})\right)
    \le \left(1 -\frac{\al}{8}\right)\cdot \Size_{n}(\Seg_{k_1}^{i}).
\end{align}
Moreover, the remaining parts of $\{\Seg_{k_1}^{i}\}_{i=1}^{m_1}$ decompose into several directed paths, each of which has starting and ending boxes of $\dist$-distance at least $2^{-k_2 + \ell}$, yielding a new family of path segments $\Seg_{k_2}^{1},\dots, \Seg_{k_2}^{m_{2}} \subset \bigcup_{i=1}^{m_1}\Seg_{k_1}^i$. We can further upper bound the total length of these new segments using \eqref{e:remaining-length} and \eqref{one segment eq2}:
\begin{align*}
    \Size_{n} \left(
            \bigcup_{j=1}^{m_{2}}
            \Seg_{k_2}^{j}
        \right) 
    &= \sum_{i=1}^{m_{1}}\Size_{n}\left(\Seg_{k_1}^{i}\setminus \Remove_{k_2}(\Seg_{k_1}^{i})\right) \nonumber\\
    \overset{\eqref{one segment eq2}}&{\le} \left(1-\frac{\al}{8}\right)\sum_{i=1}^{m_{1}}\Size_{n}\left(\Seg_{k_1}^{i}\right) \nonumber\\
    &=  \left(1-\frac{\al}{8}\right)\cdot \Size_{n} \left( \bigcup_{i=1}^{m_{1}}\Seg_{k_1}^{i}\right).
\end{align*}
We then inductively perform the procedure of Step 2 at scales $k_3,k_4,\dots$.
As we saw above, having performed this process up to scale $k_{i-1}$, the key property that allows us to repeat at scale $k_i$ is:
\begin{center}
    \emph{For each segment at scale $k_i$, the bottom-left point of its starting box and the top-right point of its ending box are of $\dist$-distance at least $2^{-k_{i-1} + \ell}$,
    which lets us apply \eqref{one event} for each segment to get a constant fraction of deleted \Major boxes at this scale.}
\end{center}
Since by induction this condition is guaranteed by Proposition \ref{l:onescaleloss2}, for every $i$, the same reasoning as at scale  $k_2$ allows to produce new segments at the next scale $k_{i+1}$, each of which also satisfies the above condition at scale $k_{i+1}$, and their total length is upper bounded by $1-\frac{\al}{8}$ times that of the segments at scale $k_i$:
\[
 \Size_{n} \left(
            \bigcup_{j=1}^{m_{i-1}}
            \Seg_{k_{i-1}}^{j}
        \right)\le\left(1-\frac{\al}{8}\right)\cdot \Size_{n} \left( \bigcup_{j=1}^{m_{i}}\Seg_{k_i}^{j}\right).
\]

Since between each iteration we skip $k_{i}-k_{i-1} = 10\ceil{\log_2 n} + \ell$ many scales,
we can repeat this procedure $I$ times, where we recall $I=\left\lfloor\frac{n}{10\ceil{\log_2 n} + \ell}\right\rfloor$.
After the $I\th$ round, we have segments $\Seg_{k_I}^{1},\dots,\Seg_{k_I}^{m_I} \subset \pi$ whose total length is $o(2^n)$:
\begin{align*}
    \Size_n\left(\bigcup_{i=1}^{m_I} \Seg^i_{k_I}\right)
    &\le \left(1-\frac{\al}{8}\right)^{I-1} \cdot \Size_n\left( \bigcup_{j=1}^{m_{1}}\Seg_{k_1}^{j}\right)\\
    \overset{\eqref{e:remaining-length}}&{\le} \left(1-\frac{\al}{8}\right)^I \cdot 2^{n+1}\\
    &\le 2^{n}\cdot e^{-cn/\log n}.
\end{align*}
Moreover, all other microscopic boxes along $\pi$ are deleted:
\begin{align*}
    \sigma_{n}(\su) = 0\quad \text{for all}\quad \mathsf{u} \in 
    \left(\pi \setminus \bigcup_{i=1}^{m_{I}}\Seg_{k_{I}}^{i}\right)\cap \Lambda_{n},
\end{align*} 
because they are contained in the removed \Major boxes at some scale $k_i$ for some $1\leq i \leq I$.
Finally, recalling the definition of passage time (Definition \ref{d:lpp}), we deduce that
\begin{align*}
        N_{\sigma_n}(\pi) 
        \le \Size_{n}\left(
            \bigcup_{i=1}^{m_{I}}
            \Seg_{k_{I}}^{i}
        \right)
        &\le 2^{n} \cdot e^{-cn/\log n}.
\end{align*}
Maximizing over all directed paths $\pi$ completes the proof of Theorem \ref{t:sublinear}.
\end{proof}

\section{Unrectifiable directed fractal percolation}\label{s:8x10}

In this section we prove \cref{t:8x10}.  As indicated already, the proof strategy is to consider appropriate ``lifts'' of paths  in $d$ dimensions that are rather tortuous. The construction of such fractal undirected paths is strongly inspired by the arguments in \cite{DGZPercolationThickPoints2026}. Accordingly, we begin by recording some pertinent definitions and results from the latter. We will then, in Section \ref{s:lifting}, present our construction.

\subsection{Constructing paths in $\R^d$}
Recall that $\ee_i$ is the $i\th$ standard basis vector.
A subset $T\subset\Z^d$ is called a \emph{tube} if there exists $x\in 8\Z^d$ and $y=x+8\eta\ee_i$, for some  $\eta\in\{-1,+1\}$ and some $i\in\lb 1,d \rb$,
such that
$T=\{z\in\Z^d : \dist(z,\sL) \le 2\}$,
where $\sL$ is the line segment between $x$ and $y$.
We write $T=T_{x,y}$ to emphasize the dependence on $x,y$.

The next lemma is a slight reformulation of \cite[Lemma 2.4]{DGZPercolationThickPoints2026}.
We write
\begin{align}\label{e:q}
    \qq:=\left\lfloor\frac{d}{10}\right\rfloor,
\end{align}
and we assume from now on that $d\ge 1000$ so that $\qq \ge 100$.
\begin{lemma}[{\cite[Lemma 2.4]{DGZPercolationThickPoints2026}}]\label{l:tube}
    Fix $x\in\Z^d$ and let $y:=x + 8\eta \ee_i$ for some $i\in\lb 1,d\rb$ and some $\eta\in\{-1,1\}$.
    Also let $z:=x+\ee_j$ for some $j\in\lb 1,d\rb$ (note that $j=i$ is allowed).
    Then there exists a set $\cT^z_{x,y}$ 
    consisting of $\qq$-many
    paths with the following properties:
    \begin{enumerate}[label=(\alph*)]
        \item \label{p1} Each path in $\cT^z_{x,y}$ is a self-avoiding path of length $11$ contained in $T_{x,y}$,
        where length means the number of vertices on the path and containment is implied in terms of sets of vertices.
        \item \label{p2} Each path $\Path\in\cT^z_{x,y}$ starts at $\Path(0)=z$ and ends at distance $1$ from $y$, i.e. $\dist(\Path(10), y)=1$.
        \item \label{p3} Different paths in $\cT_{x,y}^z$ intersect only at their starting point $z$.
        \item \label{p4} 
        Write $B_x:=x+[0, 8)^d$ and $B_y:=y+[0, 8)^d$.
        Every path $\Path\in\cT^z_{x,y}$ starts in $B_x$, ends in $B_y$, and is contained in $B_x\cup B_y$.
    \end{enumerate}

    Moreover, if we do not fix the starting point $z$, there exists a set $\cT_{x,y}$ of $\qq$-many disjoint paths, each of which starts at distance $1$ from $x$, ends at distance $1$ from $y$, and satisfies properties 
    \ref{p1} and \ref{p4}.
\end{lemma}
\begin{proof}
    We will give an explicit construction.
    We only describe the construction in the case where $x=\0$, where $y=\pm 8\ee_1$, and where $z\in\{\ee_1,\ee_2\}$.
    From this special case, we can deduce the result for arbitrary $x,y,z$ by translating and using the symmetry of the coordinates.

    We split the construction of $\cT_{x,y}^z$ into the following four cases.

    \noindent
    \textbf{Case 1 ($y=8 \ee_1, z=\ee_1$):}
    Define $\cT^z_{x,y} :=  \{\Path^\ell : \ell\in\lb 1,\qq\rb\}$, where $\Path^\ell:\lb0,10\rb \to T_{x,y}$ is the path
    \begin{align}\label{e:case1}
        \Path^\ell(t) :=
        \begin{cases}
            z&\quad t=0\\
            tz+\ee_{2\ell}&\quad t\in\lb 1, 8\rb\\
            8z + \ee_{2\ell} + \ee_{2\ell+1}&\quad t=9\\
            8z + \ee_{2\ell+1}&\quad t=10
        \end{cases}.
    \end{align}
    In words, $\Path^\ell$ starts at $z$, takes one $\ee_{2\ell}$-step, seven $\ee_1$-steps, one $\ee_{2\ell+1}$-step, and one $-\ee_{2\ell}$-step to arrive at $8z+\ee_{2\ell+1} = y + \ee_{2\ell+1}$.
    See Figure \ref{411-1}.

    \bigskip
    \noindent
    \textbf{Case 2 ($y=8 \ee_1, z=\ee_2$):}
    Define $\cT^z_{x,y} :=  \{\Path^\ell : \ell\in\lb 3,\qq+2\rb\}$, where $\Path^\ell:\lb0,10\rb \to T_{x,y}$ is the path
    \begin{align}\label{e:case2}
        \Path^\ell(t) :=
        \begin{cases}
            z&\quad t=0\\
            z+ (t-1)\ee_1+\ee_\ell&\quad t\in\lb 1, 9\rb\\
            8\ee_1 + \ee_\ell&\quad t=10
        \end{cases}.
    \end{align}
    In words, $\Path^\ell$ starts at $z$, takes one $\ee_{\ell}$-step, eight $\ee_1$-steps, and one $-z$-step to arrive at $y+\ee_\ell$.
    See Figure \ref{411-2}.

    \bigskip

    \noindent
    \textbf{Case 3 ($y=-8 \ee_1, z=\ee_1$):}
    Define $\cT^{z}_{x,y} := \{\Path^\ell : \ell\in\lb 2,\qq+1\rb\}$, where $\Path^\ell:\lb0,10\rb\to T_{x,y}$ is the path
    \begin{align}\label{e:case3}
        \Path^\ell(t) :=
        \begin{cases}
            z&\quad t=0\\
            z + \ee_\ell&\quad t=1\\
            -(t-2)z+\ee_\ell&\quad t\in\lb 2, 10\rb
        \end{cases}.
    \end{align}
    In words, $\Path^\ell$ starts at $z$ and takes one $\ee_{\ell}$-step and nine $-\ee_1$-steps to arrive at $y+\ee_\ell$.
    See Figure \ref{411-3}.

    \bigskip

    \noindent
    \textbf{Case 4 ($y=-8 \ee_1, z=\ee_2$):}
    Define $\cT^{z}_{x,y} := \{\Path^\ell : \ell\in\lb 3,\qq+2\rb\}$, where $\Path^\ell:\lb0,10\rb\to T_{x,y}$ is the path
    \begin{align}\label{e:case4}
        \Path^\ell(t) :=
        \begin{cases}
            z&\quad t=0\\
            z + \ee_\ell&\quad t=1\\
            z - (t-1)\ee_1 +\ee_\ell&\quad t\in\lb 2, 9\rb\\
            - 8\ee_1 + \ee_\ell&\quad t=10
        \end{cases}.
    \end{align}
    In words, $\Path^\ell$ starts at $z$, takes one $\ee_\ell$ step, eight $-\ee_1$-steps, and one $-z$-step to arrive at $y+\ee_\ell$.
    See Figure \ref{411-4}.

    \bigskip

    Finally, we define $\cT_{x,y}$, again assuming that $x=\0$ and $y=\pm 8\ee_1$, and considering separately the following two cases.

    \noindent
    \textbf{Case 1 ($y=8 \ee_1$):}
    Define $\cT_{x,y} := \{\Path^\ell : \ell\in\lb 1,\qq\rb\}$, where $\Path^\ell:\lb0,10\rb\to T_{x,y}$ is the path
    \begin{align}\label{e:Txycase1}
        \Path^\ell(t) :=
        \begin{cases}
            \ee_{2\ell}&\quad t=0\\
            \ee_{2\ell} + t\ee_1&\quad t\in\lb 1,8\rb\\
            \ee_{2\ell} + 8\ee_1 + \ee_{2\ell+1}&\quad t=9\\
            8\ee_1 + \ee_{2\ell+1}&\quad t=10
        \end{cases}.
    \end{align}
    In words, $\Path^\ell$ starts at $\ee_{2\ell}$, takes eight $\ee_1$-steps, one $\ee_{2\ell+1}$-step, and one $-\ee_{2\ell}$-step to arrive at $y+\ee_{2\ell+1}$.
    See Figure \ref{411-5}.

    \bigskip
    \noindent\textbf{Case 2 ($y = -8 \ee_1$):}
    Define $\cT_{x,y} := \{\Path^\ell : \ell\in\lb 1,\qq\rb\}$, where $\Path^\ell:\lb0,10\rb\to T_{x,y}$ is the path
    \begin{align}\label{e:Txycase2}
        \Path^\ell(t) :=
        \begin{cases}
            \ee_{2\ell}&\quad t=0\\
            \ee_{2\ell} - t\ee_1&\quad t\in\lb 1,8\rb\\
            \ee_{2\ell} - 8\ee_1 + \ee_{2\ell+1}&\quad t=9\\
            -8\ee_1 + \ee_{2\ell+1}&\quad t=10
        \end{cases}.
    \end{align}
    In words, $\Path^\ell$ starts at $\ee_{2\ell}$, takes eight $-\ee_1$-steps, one $\ee_{2\ell+1}$-step, and one $-\ee_{2\ell}$-step to arrive at $y+\ee_{2\ell+1}$.
    See Figure \ref{411-6}.

    \bigskip

It is straightforward to verify that all the above sets satisfy the conclusions of Lemma \ref{l:tube}.
Indeed, properties \ref{p1}, \ref{p2}, \ref{p4} are straightforward to check by direct inspection of \eqref{e:case1}--\eqref{e:Txycase2}.
For Property \ref{p3} in \eqref{e:case1}--\eqref{e:case4}, it suffices to note that in each of the four cases, for $\ell\ne\ell'$, after departing from $z$ the paths $\Path^\ell$ and $\Path^{\ell'}$ spend the rest of their journeys inside disjoint hyperplanes.
A similar argument implies the counterpart disjointness in \eqref{e:Txycase1}--\eqref{e:Txycase2}.
\end{proof}

\begin{figure}[tbhp]
    \centering

    \begin{subfigure}[t]{0.4\textwidth}
        \centering
        \includegraphics[width=\textwidth]{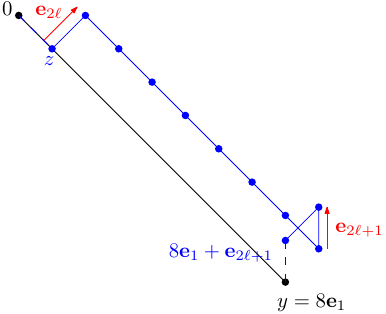}
        \caption{Case 1: $y=8\ee_1$, $z=\ee_1$.}
        \label{411-1}
    \end{subfigure}
    \hfill
    \begin{subfigure}[t]{0.4\textwidth}
        \centering
        \includegraphics[width=\textwidth]{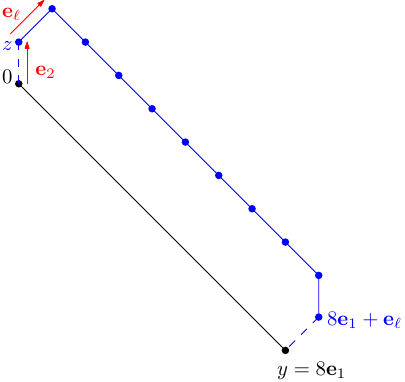}
        \caption{Case 2: $y=8\ee_1$, $z=\ee_2$.}
        \label{411-2}
    \end{subfigure}
    \vskip\baselineskip
    \begin{subfigure}[t]{0.4\textwidth}
        \centering
        \includegraphics[width=\textwidth]{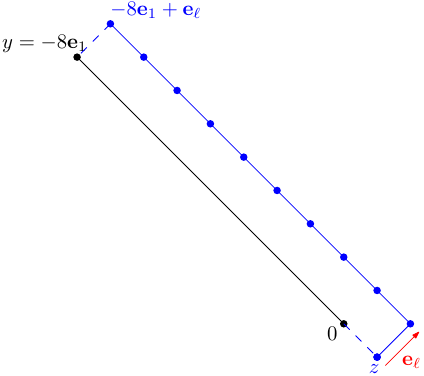}
        \caption{Case 3: $y=-8\ee_1$, $z=\ee_1$.}
        \label{411-3}
    \end{subfigure}
    \hfill
    \begin{subfigure}[t]{0.4\textwidth}
        \centering
        \includegraphics[width=\textwidth]{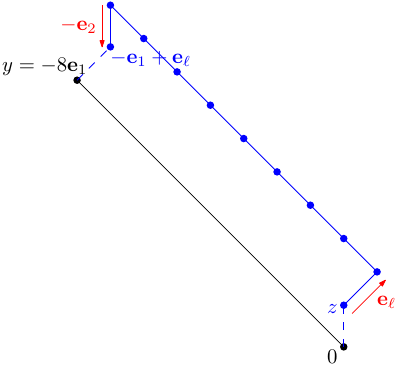}
        \caption{Case 4: $y=-8\ee_1$, $z=\ee_2$.}
        \label{411-4}
    \end{subfigure}
    \vskip\baselineskip
    \begin{subfigure}[t]{0.4\textwidth}
        \centering
        \includegraphics[width=\textwidth]{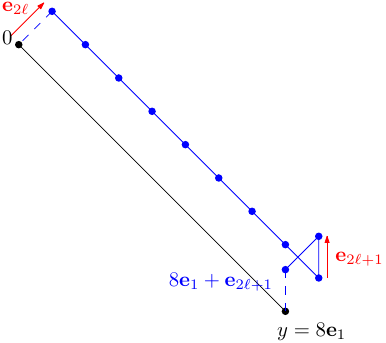}
        \caption{Case 5: $y=8\ee_1$}
        \label{411-5}
    \end{subfigure}
    \hfill
    \begin{subfigure}[t]{0.4\textwidth}
        \centering
        \includegraphics[width=\textwidth]{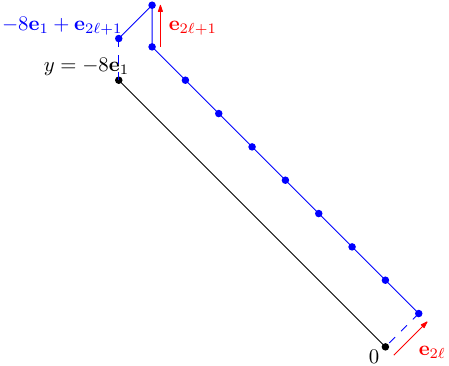}
        \caption{Case 6: $y=-8\ee_1$}
        \label{411-6}
    \end{subfigure}
    \caption{The construction underlying Lemma \ref{l:tube}. 
    Each panel depicts a path $\Path^\ell$ (blue) corresponding to one of the cases \eqref{e:case1}--\eqref{e:Txycase2}.}
    \label{f:tube}
\end{figure}

\begin{remark}[Scaled tubes]
    The notion of a tube and \cref{l:tube} have natural analogues in any scaled lattice $8^{-(k+1)}\Z^d$, for $k\ge 0$.
    Namely, for adjacent vertices $x,y\in 8^{-k}\Z^d$,
    the tube $T_{x,y}$ is given by $\{z\in 8^{-(k+1)}\Z^d : \dist(z,\sL) \le 2\cdot 8^{-(k+1)}\}$, and the natural analogue of \cref{l:tube} in $8^{-(k+1)}\Z^d$ 
    has an identical proof after scaling everything by $8^{-(k+1)}$.
\end{remark}

Our aim now is to use \cref{l:tube} to construct the paths whose lifted versions will underlie the proof of \cref{t:8x10}.
For this we need another definition:
\begin{definition}[Refinements of an edge, {\cite[Definition 2.5]{DGZPercolationThickPoints2026}}]\label{def:refinements}
    For an integer $k\ge 0$ consider vertices $x,y$ which are adjacent in $8^{-k}\Z^d$, and the tube $T_{x,y}$ in $8^{-(k+1)}\Z^d$.
    The \emph{refinements} of the  edge $(x,y)$ are those paths in $\cT_{x,y} \cup \bigcup_{z}\cT_{x,y}^z$, where the union is over all $z$ of the form $z = x+8^{-(k+1)}\ee_i$ for $i\in\lb 1,d\rb$ (as in the statement of \cref{l:tube}).

    To help parse this, observe that the refinements are connected paths on $8^{-(k+1)}\Z^d$, whose endpoints are adjacent to $x$ and $y$ (here adjacency is with respect to the graph structure of $8^{-(k+1)}\Z^d$).
    Further, the endpoints lie in $x + [0,8^{-k})^d$ and $y+[0,8^{-k})^d$ respectively.
\end{definition}

We now construct the set of paths which will be lifted in the upcoming section.
Recall we are assuming $d$ is large enough that $\qq:=\floor{d/10}\ge 100$.
We fix an integer $M\ge 10$ and write
\begin{align}\label{e:length}
    \cL_0 := M,\qquad\qquad
    \cL_{k+1} := 10(\cL_k-1) + 1
    \quad\text{for } k\ge 0.
\end{align}
For each $k\ge 0$, we now inductively define a set $\Paths{k}$ of self-avoiding lattice paths on $8^{-k}\Z^d$ of length $\cL_k$.

\begin{definition}[The set of paths $\Paths{k}$ {\cite[Definition 2.8]{DGZPercolationThickPoints2026}}]\label{d:paths}

    Let $\Paths{0}$ be the set of coordinate-wise increasing paths of length $M$ in $\Z^d$ starting from the origin $\0$.
    That is, the paths in $\Paths{0}$ are all those of the form $\Path = (\0,\; \ee_{i_1},\; \ee_{i_1}+\ee_{i_2},\;\dots,\;\ee_{i_1}+\ee_{i_2}+\cdots+\ee_{i_{M-1}})$ for $i_1,\dots,i_{M-1}\in\lb 1,d\rb$.

    We assume now that $\Paths{k}$ has been defined for some integer $k \ge 0$ and consists of paths on $8^{-k}\Z^d$, and define $\Paths{k+1}$ inductively.
    To each path $\Path = (x_0,\dots,x_{\cL_k-1}) \in \Paths{k}$ on $8^{-k}\Z^d$, we associate a family of paths on $8^{-(k+1)}\Z^d$ given by sequentially refining the edges $(x_0,x_1),(x_1,x_2),\dots$ as follows.
    \begin{enumerate}[label=(\arabic*)]
        \item\label{refine1} We deterministically choose a subset of cardinality $\qq-10$ from $\cT_{x_0,x_1}$, and then choose a refinement of the edge $(x_0,x_1)$ from this subset.
        By \cref{l:tube}\ref{p4}, any such refinement ends at $x_1 + 8^{-(k+1)}\ee_j$ for $j\in\lb 1,d\rb$.

        \item\label{refine2} Suppose we have specified refinements of the edges $(x_0,x_1),(x_1,x_2),\ldots, (x_{i-1}, x_i)$ for some $i<\cL_k-1$.
        Let $z$ be the endpoint of the refinement of $(x_{i-1}, x_i)$.
        We claim that at most $10$ paths in $\cT_{x_i, x_{i+1}}^z$ intersect the refinement of $(x_{i-1}, x_i)$ at a point other than $z$.
        To see this, note that the refinement of $(x_{i-1}, x_i)$ contains $11$ vertices (one of which is $z$), so there are only $10$ vertices in $8^{-(k+1)}\Z^d$ besides $z$ at which a path in $\cT_{x_i,x_{i+1}}^z$ can intersect the refinement of $(x_{i-1},x_i)$.
        Since different paths in $\cT_{x_i,x_{i+1}}^z$ are disjoint except at their starting point $z$, the claim follows.
        Therefore, we can deterministically choose a subset of $\cT_{x_i, x_{i+1}}^z$ with cardinality $\qq - 10$ whose constituent paths do not intersect the refinement of $(x_{i-1}, x_i)$ except at $z$,
        and then choose a refinement of $(x_i,x_{i+1})$ from this subset.
        Here we used that, by Lemma \ref{l:tube}\ref{p4}, we have $z = x_i + 8^{-(k+1)}\ee_j$ for some $j\in\lb 1,d\rb$, which permits us to apply Lemma \ref{l:tube} with $x=x_i$ and $y = x_{i+1}$.

        \item\label{refine3} We concatenate the chosen refinements of $(x_i, x_{i+1})$ for $i\in\lb 0,\cL_k-2\rb$, yielding a self-avoiding path on $8^{-(k+1)}\Z^d$ of length $\cL_{k+1}$ (see \eqref{e:length}).
        Below we indicate why the path is indeed self-avoiding.
        First, by the previous step of the construction, concatenating the refinements of any two consecutive edges yields a self-avoiding path.
        So fix any two non-consecutive edges $e=(x_i,x_{i+1})$ and $e'=(x_j,x_{j+1})$ on $\Path$.
        By the induction hypothesis $\Path$ is self-avoiding, so $\{x_i,x_{i+1}\}\cap\{x_j,x_{j+1}\}=\varnothing$ and hence $e\cap e'=\varnothing$.
        It follows by lattice geometry that $\dist(e,e')\ge 8^{-k}$, where $\dist$ is the $\ell^1$-distance on $\R^d$.
        On the other hand, the construction in the proof of Lemma \ref{l:tube} implies that any refinement of $e$ lies within $\ell^1$-distance $2\cdot 8^{-(k+1)}$ of $e$, and similarly for $e'$.
        It follows that the $\ell^1$-distance between any two such refinements is at least $8^{-k} - 4\cdot 8^{-(k+1)} > 0$, implying self-avoidance.
\end{enumerate}
Repeating the above construction for every possible sequence of refinements yields a collection of $(\qq-10)^{\cL_k-1}$ many self-avoiding paths on $8^{-(k+1)}\Z^d$ of length $\cL_{k+1}$; we call these paths the \emph{refinements} of $\Path$.

Finally, we form the set $\Paths{k+1}$ by collecting all the refinements of all paths in $\Paths{k}$.
By induction, this defines $\Paths{k}$ for all $k \ge 0$.
\end{definition}

\begin{remark}[Skeletons of paths]\label{r:refinements}
    For each $\Path\in\Paths{n}$, there is a unique sequence of paths
    \[
        \Path_0,\Path_1,\dots,\Path_n=\Path
    \]
    such that $\Path_k\in\Paths{k}$ for all $k\in\lb 0,n\rb$, and such that $\Path_k$ is a refinement of $\Path_{k-1}$ for all $k\in\lb1,n\rb$. 
    {This follows from a similar argument as in Definition \ref{d:paths}\ref{refine3}: any refinement of $\Path_{k-1}\in\cP^{k-1,M}$ must lie within distance $2\cdot 8^{-k}$ of $\Path_{k-1}$, whereas for any two distinct edges $e,e'$ in $8^{-(k-1)}\Z^d$, there must exist  points $x\in e$ and $x'\in e'$ separated by distance at least $8^{-(k-1)}$.}

    We refer to the path $\Path_k$ as the \emph{$k$-skeleton} of $\Path$.
    Also, for any $\ell>k$ and any $\Path_k\in\Paths{k}$, we refer to the paths in $\Paths{\ell}$ obtained by repeatedly refining $\Path_k$ as \emph{refinements of $\Path_k$ at scale $\ell$}.
    We will also use this terminology for iterated refinements of a single edge $(x,y)$ in $8^{-k}\Z^d$ as in Definition \ref{def:refinements}.
    It may be helpful to note that we have an equipartition of $\cP^{\ell,M}$ indexed by the paths in $\cP^{k,M}$: the partition element corresponding to $P_k\in \cP_{k,M}$ is just the set of all paths in $\cP^{\ell,M}$ obtained from iteratively refining $P_k$.
    In other words, viewing the path construction as a tree, the refinements of $P_k$ at scale $\ell>k$ are exactly the vertices at depth $\ell$ lying on the subtree below $P_k$.
\end{remark}

\begin{remark}[Path endpoints]\label{r:27-displacement}
    Let $\bS^{d-1}$ be the unit $\ell^1$-sphere in $\R^d$.
    Recall that the paths $\cP^{0,M}$ all start at the origin and end in $(M-1)\bS^{d-1}$.
    However, this is not true for paths in $\cP^{n,M}$:
    each time a path is refined, its endpoints are slightly displaced.
    These displacements decay geometrically, and hence the total displacement of the endpoints is $O(1)$.
    In fact, the $\ell^1$-displacement of each endpoint is upper bounded by $\frac{2}{7} = 2 \sum_{k=1}^\infty 8^{-k}$.
    In particular, for any $n\ge 1$, every path $\Path\in\cP^{n,M}$ 
    satisfies
    \begin{align}\label{e:displacement-27}
        \norm{\Path(0)}_1 \le \frac{2}{7}
        \qquad\text{and}\qquad
        \norm{\Path(M-1)}_1 \ge M-1-\frac{2}{7}.
    \end{align}
    This can be proved by applying \cite[Lemma 2.12, Claim 1]{DGZPercolationThickPoints2026} to the first and last edge of $P_0$.
\end{remark}

\subsection{Lifting paths to $\R^{d+1}$}\label{s:lifting}

We next discuss \emph{lifting} the paths in $\Paths{n}$ to $\R^{d+1}$.
This will provide a family of paths in $\R^{d+1}$ that will be used to prove \cref{t:8x10}.

To start, we parametrize the paths in $\Paths{k}$ constructed in \cref{d:paths}.

\begin{definition}[Time parametrization]\label{r:time-parametrization}
    For $k\ge 0$, we regard each $\Path_k\in\Paths{k}$ as a function
    \[
        \Path_k : 
        10^{-k}\lb 0, \cL_k-1\rb
        \to 8^{-k}\Z^d,
    \]
    where recall $\cL_k$ is the number of vertices on $\Path_k$.
    By \eqref{e:length} we have $10^{-k}(\cL_k - 1) = M-1$, so we can identify $\Path_k$ with a continuous path $\Path_k:[0,M-1]\to\R^d$ by including the edges between the lattice points in $8^{-k}\Z^d$, so that each edge of length $8^{-k}$ is traversed in time $10^{-k}$.
    We emphasize that $\Path_k(t)\in 8^{-k}\Z^d$ if and only if $t\in 10^{-k}\Z$.

    Moreover, note that for any refinement $\Path_{k+1}$ of $\Path_k$ and any edge $(\Path_k(t),\Path_k(t+10^{-k}))$ in the lattice $8^{-k}\Z^d$ (note that this means $t\in 10^{-k}\Z$),
    the portion of $\Path_{k+1}$ that refines the edge $(\Path_k(t),\Path_k(t+10^{-k}))$ uses \emph{exactly} the times in $[t, t+10^{-k}]$. 
    This is because the refinement contains exactly $10$ edges in $8^{-(k+1)}\Z^d$, each of which is traversed in time $10^{-(k+1)}$.
    It follows by induction that for all $\ell>k$ and any refinement $\Path_\ell$ of $\Path_k$ at scale $\ell$, the portion of $\Path_\ell$ corresponding to successive refinements of the edge $(\Path_k(t),\Path_k(t+10^{-k}))$ uses \emph{exactly} the times in $[t, t+10^{-k}]$. 
\end{definition}

We now lift the paths to $\R^{d+1}$:

\begin{definition}[Lifts of paths]\label{d:lift}
    Fix $n\ge 1$ and $M\ge 10$.
    For any path $\Path\in \Paths{n}$, its \emph{lift} is the path $\wt{\Path} : [0, M-1]\to \R^{d+1}$ given by
    \begin{align*}
        \wt{\Path}(t) := (\Path(t), t),
    \end{align*}
    where the time parametrization for $\Path$ is defined in \cref{r:time-parametrization}.
\end{definition}

The next technical lemma ensures that the lifted paths are candidates for directed percolation as in \eqref{e:8x10-critical-prob}.

\begin{lemma}\label{l:lifts-are-admissible}
    Fix $n\ge 1$ and $M\ge 10$.
    Let $\wt{\Zdir_n^{d+1}}$ be the scaled directed lattice defined in \eqref{e:8x10-lattice}.
    {For any path $\Path \in \Paths{n}$, its lift $\wt{\Path}$, restricted to times $t\in10^{-n}\Z\cap [0,M-1]$, is a path in the directed lattice $\wt{\Zdir_n^{d+1}}$.
    Further, the continuous path $\wt{\Path}$ connects $\mathbb{S}^d$ to $M\mathbb{S}^d$,}
    where $\mathbb{S}^d$ is the unit $\ell^1$-sphere in $\R^{d+1}$.
\end{lemma}
\begin{proof}
    Fix $t\in 10^{-n}\Z \cap [0,M-1]$.
    By \cref{r:time-parametrization}, we have $\Path(t) \in 8^{-n}\Z^d$,
    and moreover $\Path(t+10^{-n}) - \Path(t) = \pm 8^{-n}\ee_i$ for some $i\in\lb 1,d\rb$. 
    Therefore,
    $\wt{\Path}(t) \in 8^{-n}\Z^d \times 10^{-n}\Z$, and 
    \begin{align*}
        \wt{\Path}(t+10^{-n}) - \wt{\Path}(t)
        &= \pm 8^{-n}\ee_i
        +  10^{-n}\ee_{d+1}.
    \end{align*}
    This implies the first claim.
    It remains to show that $\wt{\Path}$ connects $\mathbb{S}^d$ and $M\mathbb{S}^d$.
    By Remark \ref{r:27-displacement}, each path $\Path\in \Paths{n}$ satisfies $\norm{\Path(0)}_1 \le \frac{2}{7}$ and $\norm{\Path(M-1)}_1 \ge M-1-\frac{2}{7}$.
    Since $\lVert\wt{\Path}(t)\rVert_1 = \norm{\Path(t)}_1 + t$, we get
    \begin{align*}
        \norm{\wt{\Path}(0)}_1 \le \frac{2}{7}
        \qquad\textup{and}\qquad \norm{\wt{\Path}(M-1)}_1 \ge 2(M-1)-\frac{2}{7} \ge M.
    \end{align*}
    Since $\wt{\Path} : [0,M-1]\to\R^{d+1}$ is continuous, it must pass through $\mathbb{S}^d$ and $M\mathbb{S}^d$.
\end{proof}

We turn next to the proof of \cref{t:8x10}. The proof implements a second moment argument adapting the proof of \cite[Theorem 1.5]{DGZPercolationThickPoints2026}.
To aid exposition, we first review the latter in the upcoming \cref{s:second-moment}, and then deduce 
\cref{t:8x10}
in \cref{s:proof-8x10}.

\subsection{Second moment calculation}\label{s:second-moment}
While the fractal percolation model is defined by independently retaining compact boxes at different scales, it will be technically convenient to in fact prove a mild strengthening of \cref{t:8x10} where the boxes are half-open (so that they are pairwise disjoint).
This was also done in the proof of Theorem \ref{t:sublinear}, see Remark \ref{r:half-open}. 
Towards this we introduce the following notation.

For $k\in\Zpos$ we denote by $\BB_k$ the set of boxes in $\R^d$ of the form
\begin{align}\label{boxpoint}
    x+\kbox{k},
    \qquad x\in 8^{-k}\Z^d,
\end{align}
and denote by $\BBlift_k$ the set of space-time boxes in $\R^{d+1}$ of the form
\begin{align}\label{boxpoint1}
    (x,t) + \kboxlift{k},
    \qquad (x,t)\in 8^{-k}\Z^d \times 10^{-k}\Z.
\end{align}
Note that the fractal percolation configuration obtained by independently retaining each box in $\BB_k$ (resp. $\BBlift_k$) with probability $p$ is naturally identified with a subset of $A_\infty(d,8)$ (resp. of $\wt{A}_\infty(d,8)$),
where the latter were defined in the discussion following \eqref{e:fractal-perc-definition} and Section \ref{s:8x10-intro}.

For a path $\Path\in\cP^{n,M}$, we denote by $\BB_k(\Path)$ the set of boxes $B\in\BB_k$ with $\Path\cap B\ne \varnothing$.
Similarly, for the lift $\Pathlift$ of $\Path$ (\cref{d:lift}), we denote by $\BBlift_k(\Pathlift)$ the set of boxes $\wt{B}\in \BBlift_k$ with $\Pathlift\cap \wt{B}\ne\varnothing$.

For $n\ge 1$, we say that a path $\Path\in\Paths{n}$ is \emph{good} if
\[
    \text{all of the boxes in } \bigcup_{k=0}^n \BB_k(\Path)
    \text{ are retained.}
\]
\begin{remark}
    Let us mention that the above notion of ``good path'' differs somewhat from that in \cite{DGZPercolationThickPoints2026}.
    The difference stems from the fact that \cite{DGZPercolationThickPoints2026} actually study the set of thick points of the $d$-dimensional branching random walk.
    Since the notion of thickness does not depend on the first $k$ scales for any fixed $k$, \cite{DGZPercolationThickPoints2026} only need to construct a path that is good at a $1-o(1)$ fraction of the scales.
    However, in fractal percolation, there is always a positive probability of the unit ball being disconnected from infinity, which motivates the above notion of goodness.
\end{remark}

The above preparation allows us to state the key assertion of \cite{DGZPercolationThickPoints2026}, which is proved using the second moment method to establish existence of good paths.
\begin{proposition}
    [{\cite[Proposition 2.1]{DGZPercolationThickPoints2026}}]
    \label{p:DGZ2.1}

    Fix any $p\in(0,1)$.
    For all sufficiently large $d$, we have
    \begin{align*}
        \inf_{M\ge 10}\,
        \inf_{n\ge 1}
        \P_p\left(
            \textup{there exists a good path in } \Paths{n}
        \right) > 0.
    \end{align*}
\end{proposition}

For completeness we sketch the proof of \cref{p:DGZ2.1}.
Consider the random variable
\[
    \cN=\cN(p,n,M) := \sum_{\Path \in \Paths{n}} \frac{1}{\P_p(\Path \textup{ is good})}\1_{\Path \textup{ is good}}.
\]
By the Paley--Zygmund inequality,
\begin{align}\label{e:pz}
    \P_p(\cN > 0) \ge \frac{\E_p[\cN]^2}{\E_p[\cN^2]}.
\end{align}
Thus to prove \cref{p:DGZ2.1}, it suffices to show that the ratio $\frac{\E_p[\cN]^2}{\E_p[\cN^2]}$ is bounded away from $0$ uniformly in $n$ and $M$.
Towards this, observe that since $\E_p[\cN]=|\Paths{n}|$, we have
\begin{align}\label{e:exp-moment}
    \frac{\E_p[\cN^2]}{\E_p[\cN]^2}
    &= 
    \frac{1}{|\Paths{n}|^2}
    \sum_{\Path, Q \in \Paths{n}} \frac{\P_p(\Path \text{ and } Q \text{ are good})}{\P_p(\Path \text{ is good}) \cdot \P_p(Q \text{ is good})}\nonumber\\
    &= 
    \frac{1}{|\Paths{n}|^2}
    \sum_{\Path, Q \in \Paths{n}} \frac{p^{\sum_{k=0}^{n}|\BB_k(\Path)\cup\BB_k(Q)|}}{p^{\sum_{k=0}^n |\BB_k(\Path)|} \cdot p^{\sum_{k=0}^n |\BB_k(Q)|}}\nonumber\\
    &= 
    \frac{1}{|\Paths{n}|^2}
    \sum_{\Path, Q \in \Paths{n}} p^{-\sum_{k=0}^{n}|\BB_k(\Path)\cap\BB_k(Q)|}\nonumber\\
    &= \bfE\left[
        p^{-\sum_{k=0}^{n}|\BB_k(\Path)\cap\BB_k(Q)|}
    \right],
\end{align}
where $\bfE$ denotes the expectation with respect to two independent uniform paths $\Path,Q\in\Paths{n}$.
Thus the problem reduces to upper bounding the exponential moments of the random variable $\sum_{k=0}^{n}|\BB_k(\Path)\cap\BB_k(Q)|$.
For this, the key observation is that $\sum_{k=0}^{n}|\BB_k(\Path)\cap\BB_k(Q)|$ is dominated by a branching process that is subcritical for sufficiently large $d$.
We turn now to making this precise.

The hierarchical description of the paths manifests in the following bound, quoted from
\cite[(2.9)]{DGZPercolationThickPoints2026}:
\begin{align}\label{e:100}
    \sum_{k=0}^n |\BB_k(\Path)\cap\BB_k(Q)| \le 
    100 \sum_{k=0}^n |\Path_k \cap Q_k|.
\end{align}
Note that if there were an exact tree-like structure, the constant $100$ could be replaced by $1$. 
Nonetheless, the above reduces upper bounding \eqref{e:exp-moment} to upper bounding exponential moments of $\sum_{k=0}^n |\Path_k \cap Q_k|$.

Finally, the key advantage of the ``space'' or ``room'' in high enough dimensions is that there can be many paths that are nearly disjoint of each other.
This is reflected in the following bound, which
captures the aforementioned branching process structure.

\begin{lemma}[{\cite[Lemma 2.11]{DGZPercolationThickPoints2026}}]\label{l:domination}
    There exists $c>0$ such that the following holds for all $n\ge 1$, all $M\ge 10$, and all $d\ge 1000$.
    Let $\Path,Q$ be chosen independently and uniformly at random from $\Paths{n}$.
    Write $Y_k := |\Path_k \cap Q_k|$ for $k\in\lb 0,n\rb$.
    Then for any $k\in\lb 1,n\rb$, conditional on 
    $Y_0,Y_1,\dots,Y_{k-1}$,
    \[
        Y_k \quad\textup{is stochastically dominated by}\quad
        11\sum_{i=1}^{2Y_{k-1}}Z_i,
    \]
    where $Z_i$ are i.i.d. Bernoulli random variables with success probability 
    $\frac{c}{d}$.
\end{lemma}

Since the success probability $\frac{c}{d}$ tends to $0$ as $d\to\infty$, 
\cref{l:domination} implies that $\sum_{k=0}^n |\Path_k\cap Q_k|$ is stochastically dominated by a branching process that can be made arbitrarily subcritical by increasing $d$.
As a result, one can show that there exists $C>0$ such that for all $\lambda>0$ and all sufficiently large $d$ (depending only on $\lambda$ and $p$),
\begin{align}\label{e:expmoment}
    \bfE\left[
        p^{-\lambda\sum_{k=0}^n |\Path_k\cap Q_k|}
    \right]    
    \le \bfE\left[
        p^{-C\lambda |\Path_0\cap Q_0|}
    \right].
\end{align}
We next bound $|\Path_0\cap Q_0|$ using the following lemma, which is essentially a strengthening of the fact that random walk is transient in dimension $d\ge 3$:

\begin{lemma}[$0$-skeleton intersections]\label{l:Y0-geometric}
    Let $P_0,Q_0$ be as above. 
    There exists $q_d\in(0,1)$ depending only on $d$ such that 
    \begin{align*}
        \bfP(|P_0 \cap Q_0| \ge k) \le q_d^k
        \qquad\textup{for each integer $k\ge 0$}.
    \end{align*}
    Moreover, $q_d\to 0$ as $d\to\infty$.
\end{lemma}
\begin{proof}
    By \cref{d:paths},
    the number of paths in $\Paths{n}$ with a given $0$-skeleton does not depend on the $0$-skeleton (see also Remark \ref{r:refinements}).
    Therefore, since $P_n,Q_n$ are independent and uniformly distributed in $\Paths{n}$, 
    their $0$-skeletons $P_0,Q_0$ are independent and uniformly distributed in $\Paths{0}$.
    
    Let $X$ be a random walk in $\Z^d$ started at $\0$, with independent increments $X(t+1)-X(t)$ drawn uniformly from $\{\ee_i - \ee_j : i,j\in\lb 1,d\rb\}$.
    By \cref{d:paths}, $|P_0\cap Q_0|$ is equal in distribution to the number of times $X$ visits the origin up to time $M-1$.
    So $|P_0\cap Q_0|$ is stochastically dominated by the total number of visits to the origin, denoted  $N:=|\{t\ge 0 : X(t) = \0\}|$.
    By the strong Markov property, $N$ is a Geometric random variable with $\bfP(N\ge k) = q_d^k$,
    where $q_d := \bfP(\exists t\ge 1 : X(t)=\0)$ is the probability of $X$ returning to $\0$.
    Finally, that $q_d \to 0$ will be proved later in \cref{l:return-probability}.
\end{proof}

By \cref{l:Y0-geometric} and a simple computation with Geometric random variables, for $d$ large enough that $q_d\, p^{-C\lambda} \le 1/2$, we have
\begin{equation}\label{e:expmoment2}
    \begin{split}
        \bfE\left[
            p^{-C\lambda |P_0 \cap Q_0|}
        \right]
        &\le 
        \frac{(1-q_d)p^{-C\lambda}}{1 - q_d\, p^{-C\lambda}}\\
        &\le 2 p^{-C\lambda}.
    \end{split}
\end{equation}
Plugging in $\lambda = 100$, 
we get that for all sufficiently large $d$,
\begin{align*}
    \frac{\E_p[\cN^2]}{\E_p[\cN]^2}
    \overset{\eqref{e:exp-moment}, \eqref{e:100}}{\ge}
    \frac{1}{\bfE\left[
            p^{-100 \sum_{k=0}^n |P_k \cap Q_k|}
        \right]}
    \overset{\eqref{e:expmoment}}{\ge}
    \frac{1}{\bfE\left[
            p^{-100C |P_0 \cap Q_0|}
        \right]}
    \overset{\eqref{e:expmoment2}}{\ge}
    \frac12 p^{100C},
\end{align*}
uniformly in $n\ge 1$ and $M\ge 10$.
This together with \eqref{e:pz} implies \cref{p:DGZ2.1}.
\qed \\

Versions of the above estimates continue to hold for the lifted paths (Definition \ref{d:lift}), which we turn to next.

\subsection{Proof of Theorem \ref{t:8x10}}\label{s:proof-8x10}

We start by redefining the notion of good paths in terms of lifts:
for $n\ge 1$, a path $\Path\in\Paths{n}$ is \emph{good} if
\[
    \text{all boxes in } \bigcup_{k=0}^n \BBlift_k(\Pathlift)
    \text{ are retained.}
\]
By \cref{l:lifts-are-admissible}, in order to prove Theorem \ref{t:8x10}, it suffices to show that for all $p\in(0,1)$ and all sufficiently large $d$,
\begin{align}\label{e:8x10-reformulated}
    \inf_{M\ge 10}\;
    \inf_{n\ge 1}
    \P_p\left(
        \textup{there exists a good path in }\Paths{n}
    \right)
    > 0.
\end{align}
By the same argument used to prove \cref{p:DGZ2.1}, 
we find that to prove \eqref{e:8x10-reformulated}, it will suffice to show that the exponential moment
\[
    \bfE\left[
        p^{-\sum_{k=0}^n |\BBlift_k(\Pathlift)\cap\BBlift_k(\wt{Q})|}
    \right]
\]
is bounded uniformly in $n\ge 1$ and $M\ge 10$.
This will be deduced from the following technical result which relates the intersection size of the lifted paths and their projections.
\begin{proposition}\label{p:boxcount-final}
    For any $n\ge 1$ and $M\ge 10$ and $k\in\lb 0,n\rb$, and any two paths $\Path,Q\in \Paths{n}$,
    we have
    \begin{align*}
        |\BBlift_k(\Pathlift) \cap \BBlift_k(\wt{Q})|
        \le 5 |\BB_k(\Path) \cap \BB_k(Q)|.
    \end{align*}
\end{proposition}

Postponing the proof of \cref{p:boxcount-final} for a moment, we now quickly deduce \cref{t:8x10}:

\begin{proof}[Proof of \cref{t:8x10}]
    As just mentioned, by \cref{l:lifts-are-admissible}, it suffices to prove \eqref{e:8x10-reformulated}.
The same argument as in \eqref{e:pz}--\eqref{e:exp-moment} implies that 
    \begin{align*}
        \P_p\left(
            \textup{there exists a good path in }\Paths{n}
        \right)
        &\ge 
        \frac{1}{\bfE\left[
            p^{-\sum_{k=0}^n |\BBlift_k(\Pathlift)\cap\BBlift_k(\wt{Q})|}
        \right]},
    \end{align*}
    where the expectation is over $P,Q$ drawn independently and uniformly at random from $\Paths{n}$.
    By \cref{p:boxcount-final} and \eqref{e:100},
    \begin{align*}
        \frac{1}{\bfE\left[
            p^{-\sum_{k=0}^n |\BBlift_k(\Pathlift)\cap\BBlift_k(\wt{Q})|}
        \right]}
        &\ge
        \frac{1}{\bfE\left[
            p^{-500\sum_{k=0}^n |\Path_k\cap Q_k|}
        \right]}.
    \end{align*}
    Applying \eqref{e:expmoment}--\eqref{e:expmoment2} with $\lambda=500$, we conclude that for all sufficiently large $d$, uniformly in $n,M$,
    \begin{align*}
        \P_p\left(
            \textup{there exists a good path in }\Paths{n}
        \right)
        &\ge \frac12 p^{500C},
    \end{align*}
    proving \eqref{e:8x10-reformulated} and hence also \cref{t:8x10}.
\end{proof}

It remains to prove \cref{p:boxcount-final}.

\addtocontents{toc}{\SkipTocEntry}
\subsection*{Proof of Proposition \ref{p:boxcount-final}}

We will actually prove the following stronger statement:

\begin{proposition}\label{p:boxcount}
    Fix any path $\Path\in\Paths{n}$, any $k\in\lb 0,n\rb$, and any box $B\in \BB_k(\Path)$.
    Then there exist at most five boxes of the form
    \begin{align*}
        B \times [t, t+10^{-k})
        \quad\text{in}\quad \BBlift_k(\Pathlift),\qquad t\in10^{-k}\Z.
    \end{align*}
\end{proposition}

This indeed implies \cref{p:boxcount-final}:

\begin{proof}[Proof of Proposition \ref{p:boxcount-final}]
    Fix paths $\Path,Q\in\Paths{n}$, and consider the map
    \[
        \phi : \BBlift_k(\Pathlift) \cap \BBlift_k(\wt{Q}) \longrightarrow \BB_k(\Path) \cap \BB_k(Q)
        \qquad\text{given by}\qquad
        B\times [t,t+10^{-k}) \mapsto B.
    \]
    By \cref{p:boxcount}, the preimage of any box in $\BB_k(\Path) \cap \BB_k(Q)$ contains at most five boxes in $\BBlift_k(\Pathlift) \cap \BBlift_k(\wt{Q})$.
    Therefore, 
    \begin{align*}
        |\BBlift_k(\Pathlift) \cap \BBlift_k(\wt{Q})|
        &= \sum_{B\in \BB_k(\Path)\cap\BB_k(Q)} |\phi^{-1}(B)|
        \le 5|\BB_k(\Path)\cap\BB_k(Q)|
    \end{align*}
    as desired.
\end{proof}

We turn now to proving Proposition \ref{p:boxcount}.
Let us first motivate the argument. In essence, Proposition \ref{p:boxcount} says that the paths in $\Paths{n}$ do not \emph{backtrack} much at any scale.
That is, for $\Path \in \Paths{n}$ and $B\in\BB_k(\Path)$,
after $\Path$ visits $B$ for the first time, any subsequent visits to $B$ must be completed within an ``on-scale'' time horizon of length $O(10^{-k})$.
The proof amounts to making this last statement precise.

We now state three technical preparatory lemmas before proving \cref{p:boxcount}.
The first statement quantifies that $P_\ell$ is a \emph{local} refinement of $P_k$ for any $\ell > k$.

\begin{lemma}[{\cite[Lemma 2.12, Claim 3]{DGZPercolationThickPoints2026}}]\label{l:DGZ-2.12}
    Fix $k\ge 0$ and a path $\Path_k\in\Paths{k}$.
    Also fix $\ell>k$ and a refinement $\Path_\ell\in\Paths{\ell}$ of $\Path_k$ at scale $\ell$.
    Then the following hold.
    \begin{enumerate}[label=(\alph*)]
        \item\label{212_1}
        We have $\Path_\ell \subset \Path_k + \kbox{k}$.
        Moreover, for all vertices $x\in 8^{-k}\Z^d\cap \Path_k$, we have $\Path_\ell \cap (x+\kbox{k})\ne \varnothing$.
        \item\label{212_3} 
        We have $\BB_{k}(\Path_\ell)=\BB_{k}(\Path_{k})$.
        \item\label{212_2} 
        For any edge $(x,y)=(\Path_k(t), \Path_k(t+10^{-k}))$ in $\Path_k$, we have
        \begin{align*}
            \{\Path_\ell(s) : s\in[t,t+10^{-k}]\} \subset (x+\kbox{k}) \cup (y+\kbox{k}).
        \end{align*}
    \end{enumerate}   
\end{lemma}
\begin{proof}

    Claim \ref{212_1} is exactly \cite[Lemma 2.12, Claim 3]{DGZPercolationThickPoints2026}.
    Recalling the definition of $\cB_k$ from \eqref{boxpoint},
    Claim \ref{212_3} follows immediately from Claim \ref{212_1}, and actually appears below \cite[(2.9)]{DGZPercolationThickPoints2026}.

{Let us prove Claim \ref{212_2}. By Definition \ref{r:time-parametrization}, the set $\{\Path_\ell(s) : s\in[t,t+10^{-k}]\}$ is exactly the refinement at scale $\ell$ of the edge $(x,y)$.
    By Lemma \ref{l:tube}\ref{p1} and induction, any such refinement lies within distance $2\sum_{j=k+1}^\ell 8^{-j} \le 
    \frac{2}{7}\cdot 8^{-k}$ of the edge $(x,y)$.
    The claim then follows from Lemma \ref{l:tube}\ref{p4} and induction.}
\end{proof}

The next lemma upper bounds the time that a path $\Path_k\in\Paths{k}$ can spend in a given box $B\in\BB_k$.
The proof is just brute force casework using only that $\Path_k$ is a self-avoiding lattice path on $8^{-k}\Z^d$.

\begin{lemma}\label{l:time-between-entry-exit}
    Fix $k\ge 0$, any path $\Path_k\in\Paths{k}$, and any box $B\in\BB_{k}(\Path_k)$.
    Let
    \[
        t_k^- := \inf\{t\in[0,M-1]:\Path_k(t)\in B\},
        \qquad t_k^+ := \sup\{t\in[0,M-1]:\Path_k(t)\in B\}.
    \]
    Then we have $t^+_k - t^-_k  \le 2\cdot 10^{-k}$.
\end{lemma}
\begin{proof}
    We may write $B=x+\kbox{k}$ for some $x\in 8^{-k}\Z^d$.
    In particular, $8^{-k}\Z^d\cap B=\{x\}$.
    Thus, since $\Path_k$ is a self-avoiding lattice path on $8^{-k}\Z^d$, which intersects $B$ by hypothesis, one of the following three cases (depicted in Figure \ref{f:entry-exit}) must occur:
    \begin{enumerate}[label={\rm (\arabic*)}]
        \item\label{case1} $\Path_k\cap B = \{x\}$, or
        \item\label{case2} $\Path_k\cap B$ is a single edge of $B$, or
        \item\label{case3} $\Path_k\cap B$ is the union of two edges of $B$.
    \end{enumerate}

    In Case \ref{case1} (Figure \ref{f:entry-exit}, top right) we have $t^+_k = t^-_k$.
    In Case \ref{case2} (Figure \ref{f:entry-exit}, bottom left) we have $t^+_k = t^-_k + 10^{-k}$, because $\Path_k$ spends time $10^{-k}$ traversing each of its edges 
    by \cref{r:time-parametrization}.

    Finally, suppose Case \ref{case3} occurs (Figure \ref{f:entry-exit}, bottom right).
    Since $\Path_k$ is self-avoiding, there is a unique time $t\in[0,M-1]$ with $\Path_k(t)=x$.
    But \emph{every} edge of $B$ contains the point $x$, so $\Path_k$ must traverse the two edges in $\Path_k\cap B$ in consecutive order.
    It follows that $t^+_k = t^-_k + 2\cdot 10^{-k}$, since $\Path_k$ spends time $10^{-k}$ on each edge.

    In all three cases we have $t^+_k - t^-_k \le 2\cdot 10^{-k}$ as claimed.
\end{proof}

\begin{figure}[tbhp]
    \centering

    \includegraphics[width=\textwidth]{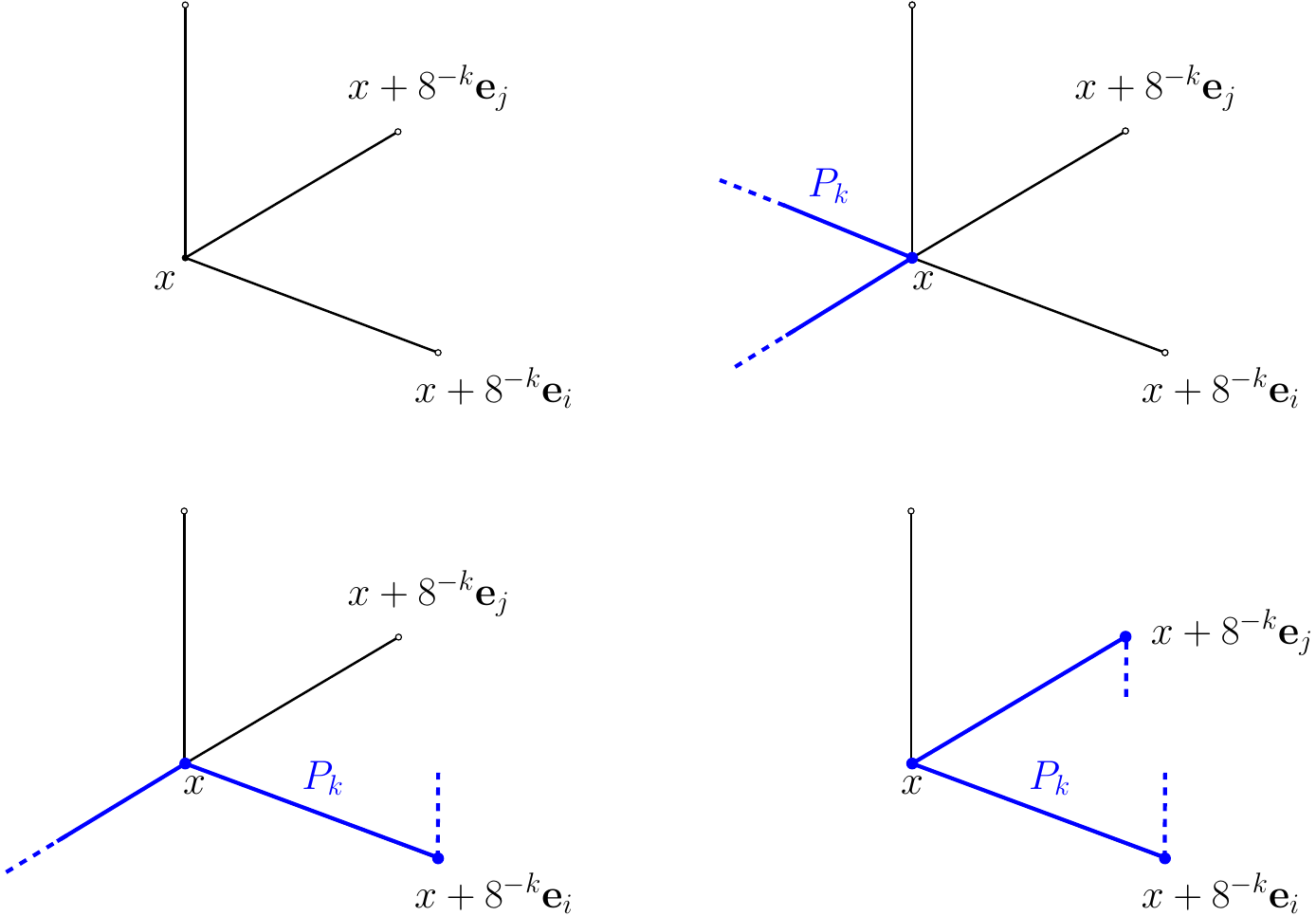}

    \caption{The proof of Lemma \ref{l:time-between-entry-exit}.
    A self-avoiding path $\Path_k$ in $8^{-k}\Z^d$ intersects a box $B=x+[0,8^{-k})^d$ for some $x\in 8^{-k}\Z^d$.
    Since $B\cap 8^{-k}\Z^d = \{x\}$, self-avoidance implies that $\Path_k$, extended to a continuous path in $\R^d$, traverses either zero, one, or two edges of $B$.
    \\
    \textbf{Top left:}
    The vertex $x\in 8^{-k}\Z^d$, and some edges of the box $B = x + [0,8^{-k})^d$. The edges of $B$ are exactly those line segments with endpoints $x$ and $x + 8^{-k}\ee_i$ for $i\in\lb 1,d\rb$.
    Note that the endpoints $x+8^{-k}\ee_i$ do not belong to $B$.
    \\
    \textbf{Top right:}
    Case \ref{case1}.
    The path $\Path_k$ (blue) exits $B$ immediately after entering, so that $\Path_k\cap B = \{x\}$.
    \\
    \textbf{Bottom left:}
    Case \ref{case2}.
    The path $\Path_k$ traverses a single edge in $B$ before exiting (depicted here as $(x, x+8^{-k}\ee_i)$).
    \\
    \textbf{Bottom right:}
    Case \ref{case3}.
    The path $\Path_k$ traverses two edges in $B$ before exiting (depicted here as $(x, x+8^{-k}\ee_i)$ and $(x, x+ 8^{-k}\ee_j)$).
    }
    \label{f:entry-exit}
\end{figure}

Recall that to prove \cref{p:boxcount}, we must show that once $\Path\in\Paths{n}$ visits $B\in\BB_k$ for the first time, after a time interval of length $O(10^{-k})$, it exits $B$ for the last time.
While this was established for the $k$-skeleton of $\Path$ in the above \cref{l:time-between-entry-exit},  
this alone does not imply the same for $\Path$ itself: in principle, iteratively refining the $k$-skeleton might create a path that returns to $B$ long after the $k$-skeleton exits $B$ for the last time.
This possibility is ruled out by the following lemma which is itself a consequence of Lemma \ref{l:DGZ-2.12}, which proves that refinements are local.

\begin{lemma}[Comparing box entrance/exit times across scales]\label{l:box-times-controlled-by-top-scale}
    Fix $k\ge 0$, any path $\Path_k\in\Paths{k}$, and any box $B\in\BB_{k}(\Path_k)$.
    Fix also any sequence of refinements $\{\Path_\ell\}_{\ell > k}$ of $\Path_k$.
    For $\ell\ge k$, let
    \[
        t_\ell^- := \inf\{t\in[0,M-1]:\Path_\ell(t)\in B\},
        \qquad t_\ell^+ := \sup\{t\in[0,M-1]:\Path_\ell(t)\in B\}.
    \]
    Then for all $\ell>k$, we have
    \[
        [t^-_\ell, \; t^+_\ell]
        \subset
        [t_k^- - 10^{-k}, \;\; t_k^+ + 10^{-k}].
    \]
\end{lemma}
\begin{proof}

It is equivalent to show that for all $s\in[0,M-1]$ and all $\ell>k$, 
we have the implication
\[
    \Path_\ell(s)\in B \qquad\implies\qquad s\in [t^-_k - 10^{-k},\;\; t^+_k+10^{-k}].
\]
Fix $s\in[0,M-1]$ and $\ell>k$, and assume that $\Path_\ell(s)\in B$.
Let $s_*$ be the unique time in $10^{-k}\Z$ satisfying $s\in [s_*,s_*+10^{-k})$, and write $x:=\Path_k(s_*)$ and $y:=\Path_k(s_*+10^{-k})$.
Then the point $\Path_\ell(s)$ lies on a refinement of the edge $(x,y)$ at scale $\ell$.
It follows by \cref{l:DGZ-2.12}\ref{212_2} that during the time interval $[s_*,\;s_*+10^{-k}]$, the refinement $\Path_\ell$ is contained in $(x+\kbox{k})\cup (y+\kbox{k})$.
But since we assumed that $\Path_\ell(s)\in B$, we must have either $B=x+\kbox{k}$ or $B=y+\kbox{k}$, and in particular either $x\in B$ or $y\in B$.
Now, if $x\in B$ then $s_* \in [t^-_k,\; t^+_k]$, and it follows that $[s_*,\; s_*+10^{-k}] \subset [t^-_k,\; t^+_k + 10^{-k}]$.
Similarly, if $y\in B$ then $s_*+10^{-k} \in [t^-_k,\; t^+_k]$, and it follows that $[s_*,\; s_*+10^{-k}]\subset [t^-_k - 10^{-k},\; t^+_k]$.
In either case we have $s\in [t^-_k - 10^{-k},\;\; t^+_k + 10^{-k}]$ as desired.
\end{proof}

We now prove \cref{p:boxcount}.

\begin{proof}[Proof of \cref{p:boxcount}]
    We analyze the cases $k<n$ and $k=n$ separately.
    
    Suppose $k<n$.
    Fix $B\in\BB_{k}(\Path) = \BB_{k}(\Path_k)$, where the equality is by \cref{l:DGZ-2.12}\ref{212_3}.
    For $\ell\in\{k,n\}$, let $t^-_\ell, t^+_\ell$ be defined as in \cref{l:box-times-controlled-by-top-scale} (they are the first and last time the $\ell$-skeleton is inside $B$).
    Applying \cref{l:box-times-controlled-by-top-scale} followed by \cref{l:time-between-entry-exit} yields
    \begin{align*}
        [t_n^-,\; t_n^+]
        &\subset [t^-_{k} - 10^{-k}, \;\; t^+_{k} + 10^{-k}]\\
        &\subset [t^-_{k} - 10^{-k}, \;\; t^-_{k} + 3\cdot 10^{-k}].
    \end{align*}
    Now, suppose that
    \begin{align}\label{e:five-boxes-1}
        B\times [t,t+10^{-k}) \in \BBlift_{k}(\Pathlift)\qquad\text{for some } t\in 10^{-k}\Z.
    \end{align}
    Then $t\in [t^-_n,\; t^+_n]$ by definition of $t^-_n,t^+_n$, and hence by the preceding display we have
    \[
        t\in[t^-_{k} - 10^{-k},\;\; t^-_{k} + 3\cdot 10^{-k}].
    \]
    Since $t\in 10^{-k}\Z$ by assumption, and since any interval of length $4\cdot 10^{-k}$ contains at most five points in $10^{-k}\Z$, we deduce that there exist at most five boxes of the form \eqref{e:five-boxes-1}, as desired.

    The case $k=n$ follows by a simpler version of the above argument.
    By \cref{l:time-between-entry-exit} we have
    \[
        [t^-_n,\; t^+_n] \subset 
        [t^-_n,\;\; t^-_n+ 2\cdot 10^{-n}].
    \]
    Since the interval on the right-hand side has length $2\cdot 10^{-n}$, there exist at most \emph{three} boxes of the form $B\times [t,t+10^{-n})\in \BBlift_n(\Pathlift)$ with $t\in 10^{-n}\Z$.
    This finishes the proof of \cref{p:boxcount}. 
\end{proof}

\appendix

\section{Proof of Lemma \ref{l:easycross-dep}}\label{a:easycross}
In this appendix, we prove Lemma \ref{l:easycross-dep}.
The proof follows the same strategy 
as the main result of \cite{ChaAbsenceDirectedFractal1995}.
\begin{figure}[h]
    \includegraphics[width = 0.48\textwidth]{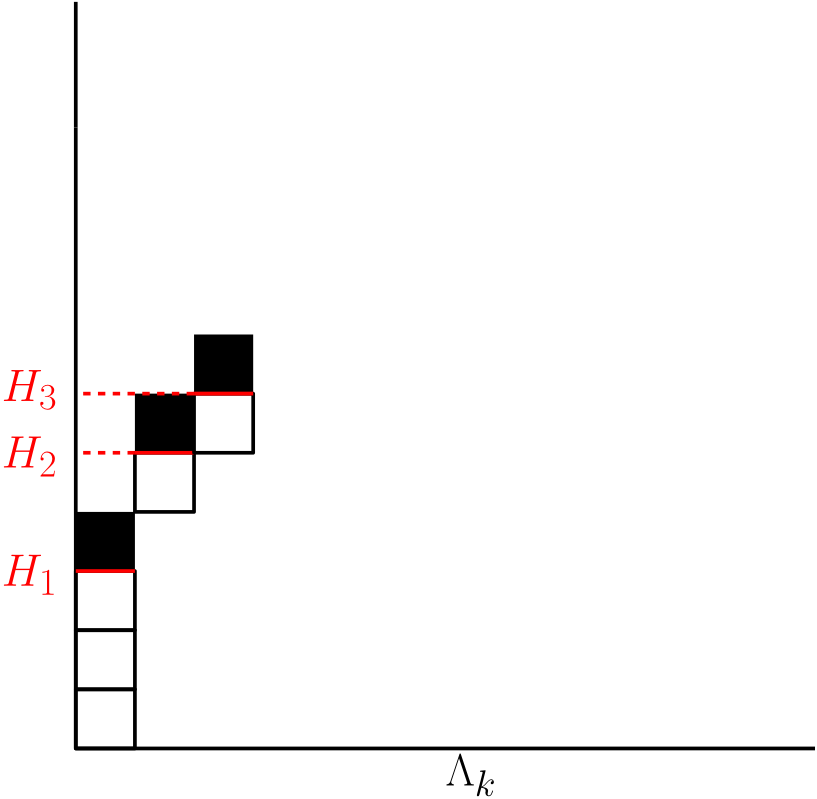}
    \includegraphics[width = 0.48\textwidth]{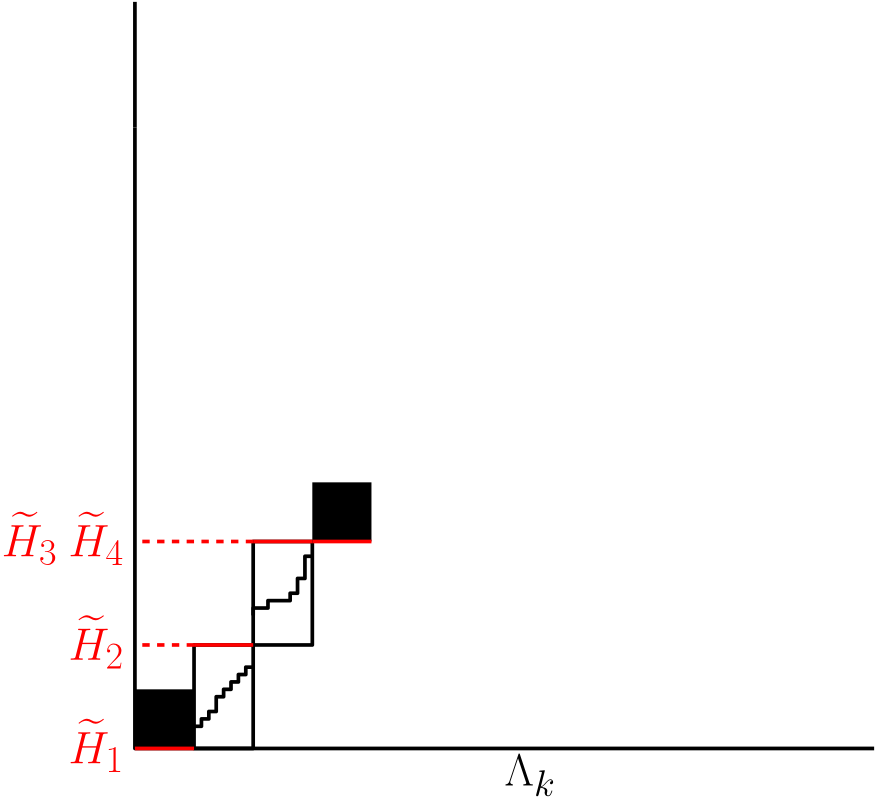}
    \includegraphics[width = 0.48\textwidth]{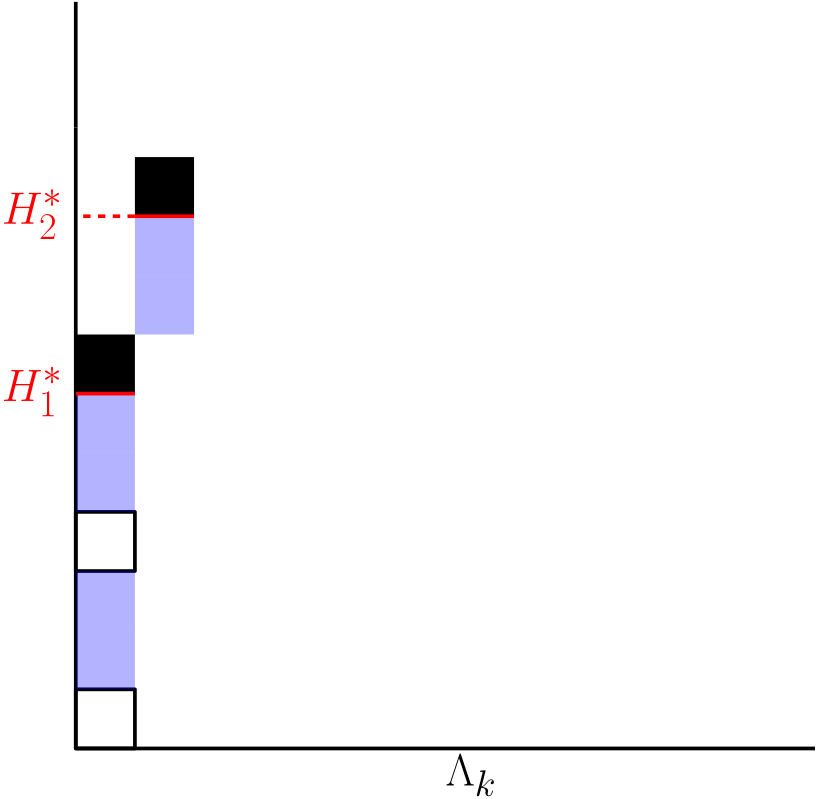}
    \includegraphics[width = 0.48\textwidth]{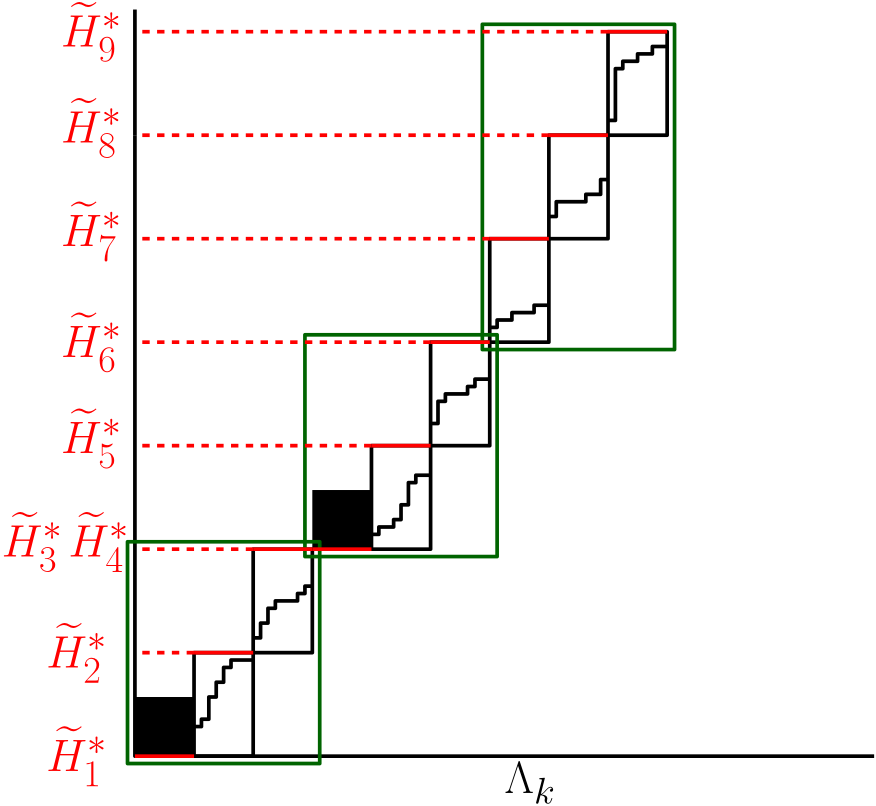}
    \caption{Illustrations of the processes $\{H_i\}$, $\{\widetilde{H}_{i}\}$, $\{H^{*}_i\}$ and $\{\widetilde{H}^{*}_i\}$ in the proof of Lemma~\ref{l:easycross-dep}.
    The main difference between $\{H_i\}$ and $\{H^{*}_i\}$ is that in the latter we skip two boxes after each test, which ensures the independence of the tests.
    Similarly, the construction of $\{\widetilde{H}_i\}$ maintains independence between tests by skipping two columns, which is achieved by inserting two contours after each test.}
    \label{f:H-processes}
\end{figure}
Therefore, we will first briefly review the argument in \cite{ChaAbsenceDirectedFractal1995}, and afterwards explain how to adapt the proof to our $1$-dependent model $\Dep(\sigma_{k})$.
In the proof, for any $k \ge 1$, $\su\in\Lambda_k$, and $\tomega_k \in \{0,1\}^{\Lambda_k}$, we say $\su$ is \emph{$\tomega_k$-open} if $\tomega_k(\su) = 1$, and \emph{$\tomega_k$-closed} if $\tomega_k(\su) = 0$.

\begin{proof}
    First, we show that the statement holds for $\sigma_k$ with some sufficiently large $k$.
    The proof begins by establishing that the probability of the existence of a directed path in $\sigma_k$ crossing the rectangle $[0,1) \times [0,L)$ vertically vanishes as $k \to \infty$, for some large constant $L > 1$.

To this end, consider the configuration $\omega_k$ restricted to the first vertical column $\cH_1$ of $\Lambda_{k}$, i.e.
\[
    \cH_1:= \left\{\su \in \Lambda_k: {\bl(\su)_{1}} = 0\right\},
\]
where $\bl(\su)=(\bl(\su)_1,\bl(\su)_2),$
and define the random variable $H_1$ to be the height of the first $\omega_k$-closed box in this column.
More precisely,
\[
    H_1 := \min\left\{h \in 2^{-k} \bZ_{\geq 0}: (0, h) + [0,2^{-k})^{2} \notin \omega_k\right\}.
\]
By definition, $2^{k} \cdot H_1$ is distributed as a Geometric random variable with success probability $1-p$; equivalently, for all $n \in \bZ_{\geq 0}$,
\[
    \mathbb{P}(H_1 = n \cdot 2^{-k}) = p^{n} (1-p).
\] 
Having defined $H_1$, we proceed to the next column $\cH_2:= \{\su \in \Lambda_k: \bl(\su)_{1} = 2^{-k}\}$, and define $H_2$ as the height of the first $\omega_k$-closed box in $\cH_2$ above $H_1$, namely
\[
    H_2 := \min\left\{h \in 2^{-k} \bZ_{\geq 0}: h \geq H_1 + 2^{-k}\quad\text{and}\quad(2^{-k}, h) + [0,2^{-k})^{2} \notin \omega_k \right\}.
\]
We iterate this process, defining $H_i$ for each $1 \leq i \leq 2^k$ recursively, and refer to the sequence of closed boxes identified in each column as a \emph{contour} (this is illustrated in Figure \ref{f:H-processes}, top left).
Observe that the increments $\{H_i - H_{i-1}\}_{i=2}^{2^k}$ are i.i.d.\ random variables,
with the same distribution as $H_1 + 2^{-k}$.
By the law of large numbers, for any 
\[
L > 2^k \mathbb{E}[H_1 + 2^{-k}] = \frac{1}{1-p},
\]
we have
\begin{align*}
    \mathbb{P}(H_{2^k} \geq L) \to 0 \quad \text{as } k \to \infty.
\end{align*}
Note that if $H_{2^k} < L$, then there cannot exist a directed path in $\omega_k$ crossing $[0,1) \times [0,L)$ vertically.
This implies that the vertical crossing probability tends to zero as $k \to \infty$, provided $L > \frac{1}{1-p}$.

One now bootstraps the previous result to show that for any $\widetilde{L} > pL$ the crossing probability of $[0,1)\times[0,\widetilde{L})$ also tends to zero. 
For simplicity we assume that $L$ is an integer. The general case follows by a similar argument with an extra rounding step  (see \cite{ChaAbsenceDirectedFractal1995} for details).
First, for all $k\ge0$ and $t\ge 1$, we define the events
\begin{align*}
    \cA_{k,k+t}(i,j) &:= \left\{
        \text{there exists a contour in $\sigma_{k,k+t}$ blocking $(i2^{-k},j2^{-k})+[0,2^{-k})\times[0,2^{-k}L)$}
    \right\},\\
    \cA_{k,k+t} &:= \bigcap_{\substack{0\le i \le 2^k-1\\ 0\le j\le 2^k-1}}\cA_{k,k+t}(i,j).
\end{align*}
By the previous result, for any fixed $k\ge 0$, we have $\P(\cA_{k,k+t})\to1$ as $t\to\infty$.
We now condition on $\cA_{k,k+t}$ and note that $\cA_{k,k+t}$ is independent of the configuration $\sigma_k$.

To estimate the conditional crossing probability we define another column process.
We start with the first column $\cH_1$ and examine the initial box $[0,2^{-k}) \times [0,2^{-k})$.
If this initial box is $\omega_k$-closed, we set $\widetilde{H}_1 := 0$ and proceed to the next column.
Otherwise, suppose the initial box is $\omega_k$-open.
Since we conditioned on $\cA_{k,k+t}$, there exists a contour in $\sigma_{k,k+t}$ that blocks vertical crossings in the rectangle $[0,2^{-k}) \times [0,2^{-k}L)$.
In this case, we define $\widetilde{H}_1 := 2^{-k}L$.
Since the event $\cA_{k,k+t}$ is independent of the configuration $\omega_k$, it follows that 
\[
    \E[\widetilde{H}_1\,|\,\cA_{k,k+t}] 
    =2^{-k}pL.
\]
We iterate this procedure to define $\widetilde{H}_i$ for each $1 \leq i \leq 2^k$.
Namely, starting from the height of the previous rectangle, we check if the next $2^{-k}\times 2^{-k}$ box is $\omega_k$-closed or $\omega_k$-open.
If it is $\omega_k$-closed then we set $\wt{H}_i = \wt{H}_{i-1}$, and if it is $\omega_k$-open we use the presence of the contour guaranteed by $\cA_{k,k+t}$ and set $\wt{H}_{i} = \wt{H}_{i-1} + 2^{-k}L$. 
This is depicted in Figure \ref{f:H-processes} (top right).

Conditional on $\cA_{k,k+t}$, the increments $\{\widetilde{H}_i - \widetilde{H}_{i-1}\}_{i=1}^{2^k}$ are i.i.d. random variables with expectation $\mathbb{E}[\widetilde{H}_i - \widetilde{H}_{i-1}\,|\,\cA_{k,k+t}] = p 2^{-k}L$, where we set $\wt{H}_0:=0$.
Consequently, for any $\widetilde{L} > pL$, we have
\begin{align*}
    \P(\widetilde{H}_{2^k} \geq \widetilde{L}\,|\,\cA_{k,k+t})
    \to 0 \quad \text{as } k \to \infty,
\end{align*}
and the convergence is uniform in $t$.
Therefore, for any $\varepsilon >0$, we can choose $k$ sufficiently large such that for all $t \ge 1$,
\[
    \P(\widetilde{H}_{2^k} \geq \widetilde{L}\,|\,\cA_{k,k+t}) \le \frac{\varepsilon}{2}.
\]
Moreover, since $\P(\cA_{k,k+t}) \to 1$ as $t \to \infty$, we can choose $t$ sufficiently large such that
\[
\P(\cA_{k,k+t}) \geq 1-\frac{\varepsilon}{2}.
\]
Combining these two estimates yields
\[
\P(\widetilde{H}_{2^k} \geq \widetilde{L}) \le \P(\cA_{k,k+t}^{c}) + \P(\widetilde{H}_{2^k} \geq \widetilde{L}\,|\,\cA_{k,k+t}) \leq \frac{\varepsilon}{2} + \frac{\varepsilon}{2} = \varepsilon.
\]
As argued in the first part, $\widetilde{H}_{2^k} \leq \widetilde{L}$ implies the existence of a contour in $\sigma_k$ blocking vertical crossings in the rectangle $[0,1) \times [0,\widetilde{L})$.
This implies that the vertical crossing probability vanishes as $k\to\infty$, for any $\widetilde{L} > pL$.
Iterating this argument yields the same conclusion for any $\widetilde{L}>p^mL$ and any $m\ge1$, and in particular for $\widetilde{L}=\frac12$.

We now explain how to adapt the previous argument to our $1$-dependent model $\Dep(\sigma_k)$.
In the first step, we examine the initial column $\cH_1$ under the percolation configuration $\Dep(\omega_k)$. 

We start by checking whether the box $[0,2^{-k})^{2}$ is $\Dep(\omega_k)$-open.
If it is, we skip the next two boxes in $\cH_1$ and proceed to check $(0, 3\cdot 2^{-k}) + [0,2^{-k})^2$, and continue moving upward in this manner---each time skipping two boxes between checked boxes---until we find one that is $\Dep(\omega_k)$-closed.
We then define $H^{*}_1$ to be the height of the bottom edge of the first $\Dep(\omega_k)$-closed box thus encountered.
Note that the tested boxes are independent due to the $1$-dependence structure of $\Dep(\sigma_k)$.

Next, we move to the adjacent column $\cH_2$ and repeat the same procedure as in the first column.
However, instead of starting from $[2^{-k},H^{*}_{1}+2^{-k}) + [0,2^{-k})^{2}$ as in the independent case, 
here we skip two boxes and start from $[2^{-k},H^{*}_{1}+3\cdot 2^{-k}) + [0,2^{-k})^{2}$.
We then sequentially check for $\Dep(\omega_k)$-closed boxes like in the first column, 
and define $H^{*}_2$ as the height of the bottom edge of the first $\Dep(\omega_k)$-closed box encountered in $\cH_2$.
Continuing inductively, we define $H^{*}_i$ for all $1 \leq i \leq 2^k$
(see Figure \ref{f:H-processes}, bottom left).

Due to the $1$-dependence structure of the percolation configuration $\Dep(\sigma_k)$, the increments $\{H^{*}_i - H^{*}_{i-1}\}_{i=2}^{2^k}$ form an i.i.d.\ sequence (this is why we started from $H_{i-1}^* + 3\cdot 2^{-k}$ instead of $H_{i-1}^* + 2^{-k}$).
Therefore, by the law of large numbers, for any $L > 2^k\, \mathbb{E}[H^{*}_2-H^{*}_1]$, we have
\begin{align*}
    \mathbb{P}(H^{*}_{2^k} \geq L) \to 0 \quad \text{as } k \to \infty.
\end{align*}

The next step is to bootstrap the parameter $L$ to a smaller value $\widetilde{L}$.
The procedure, illustrated in Figure \ref{f:H-processes} (bottom right), is again analogous to the independent case.

We start in column $\cH_1$, testing the box $[0,2^{-k})^2$.
If it is $\Dep(\omega_k)$-closed, we define $\widetilde{H}^{*}_1 = 0$.
If instead it is $\Dep(\omega_k)$-open, we further consider the configuration $\Dep(\sigma_{k, k+\ell})$ for some large $\ell$. With high probability, there exists a contour in $\Dep(\sigma_{k, k+\ell})$ that blocks directed vertical crossings of the rectangle $[0,2^{-k}) \times [0,2^{-k}L)$.
We define $\widetilde{H}^{*}_1 = 2^{-k}L$ in this case.

We next move to columns $\cH_2$ and $\cH_3$, placing similar contours inside the rectangles
\begin{align*}
    (2^{-k}, \widetilde{H}^{*}_1) + [0,2^{-k})\times [0,2^{-k}L) \qquad \text{and} \qquad 
    (2^{-k +1}, \widetilde{H}^{*}_1 + 2^{-k}L) + [0,2^{-k})\times [0,2^{-k}L).
\end{align*}
Note that we do not test $2^{-k}\times 2^{-k}$ boxes in $\cH_2$ or $\cH_3$.
We then define
\[
    \widetilde{H}^{*}_2 = \widetilde{H}^{*}_1 + 2^{-k}L \quad \text{and} \quad \widetilde{H}^{*}_3 = \widetilde{H}^{*}_1 + 2 \cdot 2^{-k}L.
\]

We now move to the next three columns $\cH_4,\cH_5,\cH_6$ and repeat the above procedure.
Namely, we first test the $2^{-k}\times 2^{-k}$ box in $\cH_4$ whose bottom edge has height $\wt{H}^*_3$, and define $\wt{H}^*_4$ analogously to $\wt{H}^*_1$.
We then place contours in $\cH_5,\cH_6$ like we did in $\cH_2,\cH_3$, and define $\wt{H}^*_5,\wt{H}^*_6$ accordingly.

We continue inductively, repeating the above in every three consecutive columns, to obtain a contour that blocks directed crossings.
The key point is that after testing a box (i.e. revealing its status in $\Dep(\omega_k)$), we skip two columns before performing the next test.
Since the presence of contours in these two columns only depends on $\Dep(\sigma_{k,k+\ell})$, this ensures that the random events under consideration are independent.
Consequently, the sequence $\{\widetilde{H}^{*}_{3i - 2}\}_{i \geq 1}$ has i.i.d.\ increments.

To analyze the growth of height, note that
\begin{align*}
    \mathbb{E}[\widetilde{H}^{*}_3]
    &= \P\left([0,2^{-k})^2 \text{ is }\Dep(\omega_k)\text{-closed}\right)
    \cdot 2\cdot 2^{-k}L
    +
    \P\left([0,2^{-k})^2 \text{ is }\Dep(\omega_k)\text{-open}\right)\cdot 3\cdot 2^{-k}L\\
    &= (1-p)^5\cdot 2\cdot 2^{-k}L
    + (1 - (1-p)^5)\cdot 3\cdot 2^{-k}L\\
    &= (3 - (1-p)^5) \cdot 2^{-k}L.
\end{align*}
Thus, by the law of large numbers, the probability of a vertical directed crossing in the rectangle $[0,1) \times [0, \widetilde{L})$ vanishes as $k \to \infty$ whenever
\begin{align*}
    \widetilde{L} > \frac{3 - (1-p)^5}{3} \cdot L.
\end{align*}
We now iterate this argument with $\wt{L}$ in place of $L$ to obtain the same conclusion for all $\wt{L} > (\frac{3-(1-p)^5}{3})^m \cdot L$, for any $m\ge 1$. 
Since $\frac{3 - (1-p)^5}{3} < 1$, we can make this threshold less than $\frac12$, which completes the proof of Lemma~\ref{l:easycross-dep}.
\end{proof}

\section{Random walk return probability in high dimensions}\label{a:return-prob}

This appendix is devoted to proving the following lemma stating that as dimension $d\to\infty$, the probability of a certain random walk in $\Z^d$ returning to the origin tends to $0$. 
In the case of simple random walk, this is quite classical and can be proven in several ways (for instance, see \cite[(5.4)]{MonRandomWalksMultidimensional1956} for an asymptotic $1/d$-expansion of the SRW return probability). The walk that we consider here is slightly non-standard and we provide a Fourier-theoretic argument for completeness, although it is not too difficult to develop a probabilistic proof by comparing to the simple random walk. 
\begin{lemma}\label{l:return-probability}
    Let $X=\{X(t)\}_{t\ge 0}$ be a random walk in $\Z^d$ started at the origin $\0$, with increments $X(t+1)-X(t)$ uniformly distributed in $\{\ee_i-\ee_j : i,j\in\lb 1,d\rb\}$.
    Let
    \begin{align*}
        q_d := \bfP(\textup{there exists $t\ge 1$ such that $X(t)=\0$}).
    \end{align*}
    Then
    \begin{align*}
        \lim_{d\to\infty} q_d
        =0.
    \end{align*}
\end{lemma}
\begin{proof}
    We will actually prove that $q_d = O(\frac{1}{d})$, which is sharp since $\bfP(X(1)=\0) = \frac{1}{d}$.
    Let $N$ be the number of times $X$ visits $\0$ besides $t=0$, so $N$ is Geometric with failure probability $q_d$.
    Writing $N = \sum_{t=1}^\infty \1_{X(t)=\0}$ and taking expectations, we get
    \begin{align}\label{e:green-0}
        \frac{q_d}{1-q_d}
        &= \sum_{t=1}^\infty \bfP(X(t)=\0).
    \end{align}
    Each increment $X(t+1)-X(t)$  has characteristic function
    \begin{align*}
        \phi_d(\theta) := \bfE\left[
            e^{\mrm{i} (X(t+1)-X(t)) \cdot \theta}
        \right]
        &=\frac{1}{d^2}\sum_{(i,j)\in\lb 1,d\rb^2} e^{\mrm{i}(\theta_i-\theta_j)}\\
        &= \left|\frac{1}{d}\sum_{j=1}^d e^{\mrm{i}\theta_j}\right|^2,
    \end{align*}
    where $\mrm{i}=\sqrt{-1}$.
    By Fourier inversion, we have
    \begin{align*}
        \bfP(X(t) = \0) &= \frac{1}{(2\pi)^d}\int_{[-\pi,\pi]^d} \phi_d(\theta)^t \,\mrm{d}\theta.
    \end{align*}
    Applying this in \eqref{e:green-0} along with Fubini (note that $\phi_d(\theta)\ge 0$) yields
    \begin{align*}
        \frac{q_d}{1-q_d}
        &=
        \frac{1}{(2\pi)^d}\int_{[-\pi,\pi]^d} 
            \sum_{t=1}^\infty 
            \phi_d(\theta)^t
        \,\mrm{d}\theta\\
        &= \frac{1}{(2\pi)^d} 
        \int_{[-\pi,\pi]^d}
        \frac{\phi_d(\theta)}{1-\phi_d(\theta)}\,\mrm{d}\theta,
    \end{align*}
    where we used that, by Cauchy--Schwarz, $\phi_d(\theta)< 1$ except on the diagonal $\{\theta_1=\cdots=\theta_d\}$ which has measure zero.
    It remains to show that the RHS of the above display is $O(\frac{1}{d})$.

    To simplify notation, let $P$ be the law of $\theta \sim \mathrm{Uniform}([-\pi,\pi]^d)$, and let $E$ be expectation with respect to $P$.
    Then the integral above is $E\left[\frac{\phi_d}{1-\phi_d}\right]$.
    We fix $\d\in(0,\frac12)$, to be specified later, and decompose the expectation as follows:
    \begin{align}\label{e:green-function}
        E\left[
            \frac{\phi_d}{1-\phi_d}
        \right]
        &= 
        E\left[
            \frac{\phi_d}{1-\phi_d}
            \1_{1-\phi_d > \frac12}
        \right]
        + E\left[
            \frac{\phi_d}{1-\phi_d}
            \1_{\d \le 1-\phi_d \le \frac12}
        \right]
        +
        E\left[
            \frac{\phi_d}{1-\phi_d}
            \1_{1-\phi_d < \d}
        \right]\nonumber\\
        &\le
        E\left[
            \frac{\phi_d}{1-\phi_d}
            \1_{1-\phi_d > \frac12}
        \right]
        + 
        E\left[
            \frac{1}{1-\phi_d}
            \1_{\d \le 1-\phi_d \le \frac12}
        \right]
        +
        E\left[
            \frac{1}{1-\phi_d}
            \1_{1-\phi_d < \d}
        \right],
    \end{align}
    because $0 \le \phi_d \le 1$ pointwise.

    For the first term in \eqref{e:green-function}, using that $\phi_d \ge 0$, we get
    \begin{align}\label{e:green-1}
        E\left[
            \frac{\phi_d}{1-\phi_d}
            \1_{1-\phi_d > \frac12}
        \right]
        \le
        2 E[\phi_d]
        = \frac{2}{d},
    \end{align}
    where the fact that $E[\phi_d]=\frac{1}{d}$ can be deduced from the identity
    \begin{align}\label{e:phi-identity}
        \phi_d(\theta) 
        &= 
        \frac{1}{d^2}\sum_{i,j\in\lb 1,d\rb} \cos(\theta_i-\theta_j).
    \end{align}

    For the second term in \eqref{e:green-function}, we write
    \begin{align*}
        \phi_d(\theta) = 
        \left(\frac{1}{d}
        \sum_{i=1}^d \cos\theta_i\right)^2
        + \left(\frac{1}{d}
        \sum_{i=1}^d \sin\theta_i\right)^2.
    \end{align*} 
    By a union bound and Hoeffding's inequality, 
    \begin{align*}
        P\left(
            1-\phi_d \le \frac{1}{2}
        \right)
        &= P\left(\phi_d \ge \frac{1}{2}\right)\\
        &\le
        2P\left(
            \left(\frac{1}{d}
            \sum_{i=1}^d \cos\theta_i\right)^2
            \ge \frac{1}{4}
        \right)=
        2P\left(
            \left|\frac{1}{d}\sum_{i=1}^d
            \cos \theta_i\right|
            \ge \frac{1}{2}
        \right)
        \le C e^{-cd}.
    \end{align*}
    Therefore, the second term in \eqref{e:green-function} is bounded by
    \begin{align}\label{e:green-2}
        E\left[
            \frac{1}{1-\phi_d}
            \1_{\d \le 1-\phi_d \le \frac12}
        \right]
        &\le \frac{C}{\d} e^{-cd}.
    \end{align}

    Finally we bound the third term in \eqref{e:green-function}.
    By \eqref{e:phi-identity}, we have
    \begin{align*}
        1 - \phi_d(\theta) = \frac{2}{d^2} 
        \sum_{1 \le i < j \le d} (1 - \cos(\theta_i - \theta_j)).
    \end{align*}
    Let $Y_{i,j} := 1 - \cos(\theta_i-\theta_j)$.
    Since $Y_{i,j} \ls (\theta_i-\theta_j)^2$, we have the lower tail bound
    \begin{align*}
        P\left(
            Y_{i,j}
            < \e
        \right)
        \le C \sqrt{\e}
        \qquad\text{for $\e>0$.}
    \end{align*}
    Since $Y_{i,j}\ge 0$, we have the crude bound $1-\phi_d \ge \frac{1}{d^2}\max_{1\le i\le d-1} Y_{i,d}$.
    Since $Y_{1,d},Y_{2,d},\dots,Y_{d-1,d}$ are i.i.d., we get
    \begin{align*}
        P\left(
            1-\phi_d
             < \e
        \right)
        &\le
        P\left(
            \max_{1\le i \le d-1} Y_{i,d} < \e d^2
        \right)\\
        &\le (C\e^{1/2} d)^{d-1}.
    \end{align*}
    We now bound the third term in \eqref{e:green-function} using a dyadic decomposition:
    \begin{align*}
        E\left[
            \frac{1}{1-\phi_d}
            \1_{1-\phi_d < \d}
        \right]
        &\le \sum_{k=0}^\infty
        \frac{2^{k+1}}{\d}
        P\left(
            \frac{\d}{2^{k+1}} \le 1-\phi_d < \frac{\d}{2^k}
        \right)\\
        &\ls
        \sum_{k=0}^\infty 
        \frac{2^{k}}{\d}
        \left(C \frac{\d^{1/2}}{2^{k/2}} d\right)^{d-1} \\
        &= (Cd)^{d-1} \d^{(d-3)/2}
        \sum_{k=0}^\infty \frac{1}{2^{(d-3)k/2}}\\
        &\ls(Cd)^{d-1} \d^{(d-3)/2},
    \end{align*}
    where the last inequality is valid for $d \ge 4$.

    Combining the above with \eqref{e:green-1}, \eqref{e:green-2}, \eqref{e:green-function},
    we get that for all $d\ge 4$ and $\d\in(0,1/2)$,
    \begin{align*}
        E\left[
            \frac{\phi_d}{1-\phi_d}
        \right]
        &\le \frac{2}{d}
        + \frac{C}{\d}e^{-cd}
        + (C d)^{d-1} \d^{(d-3)/2}.
    \end{align*}
    Plug in $\d = d^{-100}$, so that $d^{d-1} \d^{(d-3)/2} = d^{-49d + 149} < d^{-d}$ for all $d \ge 4$.
    We get
    \begin{align*}
        E\left[
            \frac{\phi_d}{1-\phi_d}
        \right]
        &\le 
        \frac{2}{d}
        + C d^{100} e^{-cd}
        + \frac{C^{d-1}}{d^d}\\
        &\ls \frac{1}{d}.
    \end{align*}
    The LHS equals $\frac{q_d}{1-q_d}$, so we get $q_d \ls \frac{1}{d}$
    as desired.
\end{proof}

\emergencystretch=1em
\printbibliography


\end{document}